\documentclass{article}
\usepackage[utf8]{inputenc}
\usepackage{amssymb,amsmath, amsthm, mathabx,mathtools}
\usepackage{amsmath,amsfonts,amssymb,amsthm,epsfig,epstopdf,titling,url,array}
\usepackage{tikz}
\usetikzlibrary{shapes,snakes}
\usepackage{color, enumerate}
\usepackage{comment, authblk}
\usepackage{hyperref}
\usepackage{pgfplots}
\usepackage{bm}
\usepackage{stmaryrd}
\usepackage{graphicx}
\usepackage{caption}
\usepackage{subcaption}
\usepackage{a4wide}
\usepackage{titlesec}
\usepackage[a4paper, total={6in, 8in}]{geometry}
\usepackage{epstopdf}

\usepackage{sectsty}
\usepackage{lipsum}
\usepackage{algorithm}
\usepackage[noend]{algpseudocode}
\usepackage{algpseudocode}
\usepackage{pifont}

\usepackage{showlabels}

\titleformat*{\subsection}{\large\bfseries}
\numberwithin{equation}{section}

\pgfplotsset{compat=newest}
\pgfplotsset{plot coordinates/math parser=false}
\newlength\figureheight
\newlength\figurewidth

\usepackage{enumerate}
\usepackage{mathrsfs}
\usepackage{wrapfig}

\usepackage{color}

\numberwithin{equation}{section}

\newcommand{\wO}{\widetilde{\mathcal{O}}}

\newcommand{\beq}{\begin{equation}}
\newcommand{\bEq}{\end{equation}}

\newcommand{\bx}{{\bf{x}}}
\newcommand{\by}{{\bf{y}}}

\newcommand{\bz} {{\bf {z}}}

\newcommand{\al}{\alpha}

\newcommand{\be}{\begin{equation}}
\newcommand{\ee}{\end{equation}}

\newcommand{\e}{{\varepsilon}}

\newcommand{\fa}{{\mathfrak a}}
\newcommand{\fb}{{\mathfrak b}}
\newcommand{\fc}{{\mathfrak c}}

\newcommand{\bU}{ {\bf  U}}
\newcommand{\bV}{ {\bf  V}}
\newcommand{\bW}{ {\bf  W}}

\usepackage{amsmath} 
\usepackage{amssymb}
\usepackage{amsthm}

\renewcommand{\b}[1]{\bm{\mathrm{#1}}} 

\renewcommand{\cal}{\mathcal}

\newcommand{\wh}{\widehat}
\newcommand{\wt}{\widetilde}

\newcommand{\ii}{\mathrm{i}} 
\newcommand{\dd}{\mathrm{d}}

\renewcommand{\epsilon}{\varepsilon}
\renewcommand{\leq}{\leqslant}
\renewcommand{\geq}{\geqslant}

\renewcommand{\le}{\leq}
\renewcommand{\ge}{\geq}

\newcommand{\E}{\mathbb{E}}
\newcommand{\R}{\mathbb{R}}
\newcommand{\C}{\mathbb{C}}
\newcommand{\N}{\mathbb{N}}
\newcommand{\Z}{\mathbb{Z}}

\newcommand{\IE}{\mathbb{I} \mathbb{E}}

\DeclareMathOperator{\diag}{diag}
\DeclareMathOperator{\tr}{Tr}

\DeclareMathOperator{\re}{Re}
\DeclareMathOperator{\im}{Im}

\DeclareMathOperator{\OO}{O}
\DeclareMathOperator{\oo}{o}

\DeclareMathOperator{\bO}{{\mathbf O}}

\DeclareMathOperator{\bv}{\mathbf{v}}
\DeclareMathOperator{\bu}{\mathbf{u}}
\DeclareMathOperator{\bw}{\mathbf{w}}

\DeclareMathOperator{\bbE}{\mathbb{E}}
\DeclareMathOperator{\bbN}{\mathbb{N}}
\DeclareMathOperator{\bbP}{\mathbb{P}}

\DeclareMathOperator{\sI}{\mathcal{I}}

\theoremstyle{plain} 
\newtheorem{theorem}{Theorem}[section]
\newtheorem*{theorem*}{Theorem}
\newtheorem{lemma}[theorem]{Lemma}
\newtheorem{assumption}[theorem]{Assumption}
\newtheorem*{lemma*}{Lemma}

\newtheorem*{corollary*}{Corollary}
\newtheorem{proposition}[theorem]{Proposition}
\newtheorem*{proposition*}{Proposition}
\newtheorem{claim}[theorem]{Claim}

\newtheorem{definition}[theorem]{Definition}
\newtheorem*{definition*}{Definition}
\theoremstyle{remark}

\newtheorem*{example*}{Example}

\newtheorem*{remark*}{Remark}
\newtheorem*{remarks*}{Remarks}

\renewcommand{\Im}{{\rm{Im}}}

\makeatletter
\def\env@dmatrix{\hskip -\arraycolsep
	\let\@ifnextchar\new@ifnextchar
	\extrarowheight=2ex
	\array{*\c@MaxMatrixCols{>{\displaystyle}c}}}

\makeatother

\allowdisplaybreaks

\title{Limiting distribution of the sample canonical correlation coefficients of high-dimensional random vectors} 
\author[1]{Fan Yang  \thanks{E-mail: fyang75@wharton.upenn.edu. This work was partially supported by the Wharton Dean’s Fund for Postdoctoral Research.}}
\affil[1]{Department of Statistics and Data Science, University of Pennsylvania}

\begin{document}
\maketitle

\begin{abstract} 
In this paper, we prove a CLT for the sample canonical correlation coefficients between two high-dimensional random vectors with finite rank correlations. More precisely, consider two random vectors $\wt{\bx}=\mathbf x  + A \mathbf z  $ and $\wt{\by}=\mathbf y + B \mathbf z $, where $\mathbf x \in \R^p$, $\mathbf y \in \R^q$ and $\mathbf z\in \R^r$ are independent random vectors with i.i.d.\;entries of mean zero and variance one, and $A \in \R^{p\times r}$ and $B\in \R^{q\times r}$ are two arbitrary deterministic matrices. Given $n$ samples of $\wt{\bx}$ and $\wt{\by}$, we stack them into two matrices $\cal X= X+AZ$ and $\cal Y= Y+BZ$, where $X\in \R^{p\times n}$, $Y\in \R^{q\times n}$ and $Z\in \R^{r\times n}$ are random matrices with i.i.d.\;entries of mean zero and variance one. Let $\wt\lambda_1 \ge \wt\lambda_2\ge \cdots \ge \wt\lambda_{r}$ be the largest $r$ eigenvalues of the sample canonical correlation (SCC) matrix $\cal C_{\cal X\cal Y}=(\cal X\cal X^\top)^{-1/2}\cal X\cal Y^\top (\cal Y\cal Y^\top)^{-1}\cal Y \cal X^\top (\cal X\cal X^\top)^{-1/2}$, and let $t_1\ge t_2 \ge \cdots\ge t_r$ be the squares of the population canonical correlation coefficients between $\wt{\bx}$ and $\wt{\by}$. Under certain moment assumptions, we show that there exists a threshold $t_c \in(0, 1)$ such that if $t_i>t_c$, then \smash{$\sqrt{n} (\wt\lambda_i-\theta_i)$} converges weakly to a centered normal distribution, where $\theta_i $ is a fixed outlier location determined by $t_i$. Our proof uses a self-adjoint linearization of the SCC matrix and a sharp local law on the inverse of the linearized matrix.
\end{abstract}


\section{Introduction}\label{sec_intro}

Given two random vectors $\wt{\mathbf x}\in \R^p$ and $\wt{\mathbf y}\in \R^q$, canonical correlation analysis (CCA) has been one of the most classical methods to study the correlations between them since the seminal work by Hotelling \cite{Hotelling}. More precisely, CCA seeks two sequences of orthonormal vectors, such that the projections of $\wt{\mathbf x}$ and $\wt{\by}$ onto these vectors have maximized correlations. These correlations are referred to as {\it canonical correlation coefficients} (CCCs), 
which can be characterized as the square roots of the eigenvalues of the population canonical correlation (PCC) matrix 
\be\nonumber
\wt{\bm\Sigma}:=\Sigma_{xx}^{-1/2}\Sigma_{xy}\Sigma_{yy}^{-1}\Sigma_{yx}\Sigma_{xx}^{-1/2},
\ee  
where $\Sigma_{xx}$, $\Sigma_{yy}$, $\Sigma_{xy}$ and $\Sigma_{yx}$ are the population covariance and cross-covariance matrices defined by
$$\Sigma_{xx}:= \mathbb E (\wt{\bx}\wt{\bx}^\top )-(\E\wt{\mathbf x})(\E\wt{\mathbf x})^\top,\quad \Sigma_{yy}:=\E (\wt{\by}\wt{\by}^\top )-(\E\wt{\mathbf y})(\E\wt{\mathbf y})^\top,$$
$$ \Sigma_{xy}=\Sigma_{yx}^\top:=\E (\wt{\bx}\wt{\by}^\top )-(\E\wt{\mathbf x})(\E\wt{\mathbf y})^\top.$$ 
In this paper, we consider the following standard signal-plus-noise model for $\wt{\bx}$ and $\wt{\by}$:
\be\label{data_model}
\wt{ \mathbf x} =\mathbf x + A \mathbf z   , \quad \wt \by=\mathbf y+ B \mathbf z ,
\ee
where $ \mathbf x\in \R^p$ and $\by\in \R^q$ are two independent noise vectors with i.i.d.\;entries of mean zero and variance one,  $\bz\in \R^r$ is a shared signal vector with i.i.d.\;entries of mean zero and variance one (which yields a rank-$r$ correlation), and $A \in \R^{p\times r}$ and $B \in \R^{q\times r}$ are two arbitrary deterministic matrices. Under the model \eqref{data_model}, the PCC matrix is given by a rank-$r$ matrix
\be\label{def_PCC} \wt{\bm\Sigma}=(I_p + AA^\top)^{-1/2}AB^\top (I_p + BB^\top)^{-1}BA^\top(I_p + AA^\top)^{-1/2} ,\ee
and we denote the $r$ non-trivial eigenvalues of $\wt{\bm\Sigma}$ as $t_1\ge t_2\ge \cdots \ge t_r \ge 0$.

We can study $\wt{\bm\Sigma}$ and the population CCCs via their sample counterparts, i.e., the {\it sample canonical correlation} (SCC) matrix and the sample CCCs. More precisely, let $(\wt{\bx}_i,\wt{\by}_i)$, $1\le i \le n$, be $n$ i.i.d.\;samples of $( \wt{\mathbf x}, \wt{\mathbf y})$.
We stack them (as column vectors) into two matrices 
\be\label{assm data}
{\cal X}:= n^{-1/2}\begin{pmatrix}\wt{\bx}_1,\wt\bx_2, \cdots ,\wt\bx_n\end{pmatrix}=X+ AZ,\quad {\cal Y}:= n^{-1/2}\begin{pmatrix}\wt{\by}_1,\wt\by_2, \cdots ,\wt\by_n\end{pmatrix}=Y+ B Z ,\ee
where $n^{-1/2}$ is a convenient scaling, with which we can write the sample covariance and cross-covariance matrices concisely as
$$\wt S_{xx}:={\cal X}{\cal X}^\top  , \quad \wt S_{yy}:={\cal Y}{\cal Y}^\top , \quad \wt S_{xy}=\wt S_{yx}^\top :={\cal X}{\cal Y}^\top  ,$$ 
and $X$, $Y$ and $Z$ are respectively $p\times n$, $q\times n$ and $r\times n$ matrices with i.i.d. entries of mean zero and variance $n^{-1}$. 
Then, we define the SCC matrix as
$$\cal C_{\cal X\cal Y}:=\wt S_{xx}^{-1/2}\wt S_{xy}\wt S_{yy}^{-1}\wt S_{yx}\wt S_{xx}^{-1/2}$$
and denote their eigenvalues by
$\smash{\wt\lambda_1 \ge \wt\lambda_2\ge \cdots \ge \wt\lambda_{p\wedge q}\ge 0}.$ 
The square roots of these eigenvalues are referred to as \emph{sample canonical correlation coefficients}. Equivalently, the sample CCCs are the cosines of the principal angles between the two subspaces spanned by the rows of $\cal X$ and $\cal Y$, respectively. If $n\to \infty$ while $p,$ $q$ and $r$ are fixed, it is easy to see that the SCC matrix converges to the PCC matrix almost surely by the law of large numbers, and hence every sample CCC converges almost surely to the corresponding population CCC. 
On the other hand, in this paper, we focus on the high-dimensional setting with a low-rank signal: ${p}/{n}\to c_1$ and ${q}/{n}\to c_2$ as $n\to \infty$ for some constants $c_1\in (0,1)$ and $c_2 \in (0, 1-c_1)$, and $r$ is a fixed integer that does not depend on $n$. In this case, the behavior of the SCC matrix deviates greatly from that of the PCC matrix.

\vspace{5pt}
\noindent{\bf Related work.} In the null case with $r=0$, the eigenvalue statistics of the SCC matrix have been well-understood. If $\cal X$ and $\cal Y$ are Gaussian matrices, then the eigenvalues of $\cal C_{\cal X\cal Y}$ reduce to those of a double Wishart matrix, which belongs to the famous Jacobi ensemble \cite{CCA_TW}. It was shown in \cite{Wachter} that, almost surely, the empirical spectral distribution (ESD) of the double Wishart matrix converges weakly to a deterministic probability distribution (cf. \eqref{LSD} below). By analyzing the joint eigenvalue density of the Jacobi ensemble, Johnstone \cite{CCA_TW} proved that the largest eigenvalues of double Wishart matrices satisfy the Tracy-Widom law asymptotically. Alternatively, the Tracy-Widom law of double Wishart matrices can also be obtained as a consequence of the results in \cite{CCA_TW2} for F-type matrices. In the general non-Gaussian case, the convergence of the ESD of $\cal C_{\cal X\cal Y}$ was proved in \cite{CCA_ESD},  the CLT of the linear spectral statistics for $\cal C_{\cal X\cal Y}$ was proved in \cite{CCA_CLT}, and  the Tracy-Widom law of the largest eigenvalue of $\cal C_{\cal X\cal Y}$ was proved in \cite{CCA2} under the assumption that the entries of $\wt\bx$ and $\wt\by$ have finite moments up to any order. The moment assumption for the Tracy-Widom law was later relaxed to the finite fourth moment condition in \cite{PartIII}. 

Some arguments in the literature for the null case are based on the fact that the subspaces spanned by the rows of $\cal X$ and $\cal Y$ are approximately uniformly (Haar) distributed random subspaces, which, however, does not hold for the non-null case with $r> 0$. This makes the study of the non-null case more challenging. Assuming that $\cal X$ and $\cal Y$ are both Gaussian matrices, the asymptotic behaviors of the likelihood ratio processes of CCA under the null hypothesis of no spikes (i.e., $r=0$) and the alternative hypothesis of a single spike (i.e., $r=1$) were studied in \cite{johnstone2020}. If either $p$ or $q$ is fixed as $n\to \infty$, the asymptotic distributions of the sample CCCs were derived in \cite{Fujikoshi2017} under the Gaussian assumption. On the other hand, if $p$ and $q$ are both proportional to $n$, the limiting distributions of the sample CCCs have been established under the Gaussian assumption in \cite{CCA}, which we discuss in more detail now.

\vspace{5pt}
\noindent{\bf BBP transition.} Suppose $X$, $Y$ and $Z$ are independent random matrices with i.i.d.\;Gaussian entries. Bao et al. \cite{CCA}  proved that for any $1\le i \le r$, the behavior of $\wt\lambda_i$ undergoes a sharp transition across the threshold $t_c$ defined by
\be\label{tc}t_c :=  \sqrt{\frac{c_1 c_2}{(1-c_1)(1-c_2)}}.\ee
More precisely, the following dichotomy occurs: 
\begin{itemize}
	\item[(1)] if $t_i < t_c$, then $\wt\lambda_i$ sticks to the right edge $\lambda_+$ (cf. \eqref{lambdapm} below) of the limiting bulk eigenvalue spectrum of the SCC matrix, and $n^{2/3}( \wt\lambda_i - \lambda_+)$ converges weakly to the Tracy-Widom law; 
	
	\item[(2)] if $t_i>t_c$, then $\wt\lambda_i$ lies around a fixed location $\theta_i \in (\lambda_+,1)$ (cf. \eqref{gammai} below), and $n^{1/2}(\wt\lambda_i - \theta_i)$ converges weakly to a centered normal random variable.
\end{itemize}
Following the notation in random matrix theory literature, we call $\wt\lambda_i$ in case (2) an \emph{outlier}. The above abrupt change of the behavior of \smash{$\wt\lambda_i$} when $t_i$ crosses $t_c$ is generally referred to as a \emph{BBP transition}, which dates back to the seminal work of Baik, Ben Arous and P\'{e}ch\'{e} \cite{BBP} on spiked sample covariance matrices. The phenomenon of BBP transition has been observed in many random matrix ensembles deformed by low-rank perturbations. Without attempting to be comprehensive, we refer the reader to \cite{capitaine2012,capitaine2009,FP_2007,KY,KY_AOP,Peche_2006} about deformed Wigner matrices,  \cite{bai2008_spike, BBP, Baik2006, principal,FP_2009, IJ2,DP_spike} about spiked sample covariance matrices, \cite{DY2, yang2018,YLDW2020} about spiked separable covariance matrices, and \cite{belinschi2017,benaych-georges2011,BENAYCHGEORGES2011,DY_JMVA,DY_IEEE,wang2017,ZhangPan2020} about several other types of deformed random matrix ensembles. The SCC matrix $\cal C_{{\cal X}{\cal Y}}$ considered in this paper can be regarded as a low-rank perturbation of the SCC matrix in the null case with $r=0$.

\vspace{5pt}
\noindent{\bf Main results and basic ideas.} A natural question is whether the above BBP transition holds universally if we only assume certain moment conditions on the entries of $X$, $Y$ and $Z$. Answering this question is not only theoretically interesting from the point of view of random matrix theory, but also crucial for modern applications of CCA in e.g., statistical learning, wireless communications, financial economics and population genetics. In this paper, we solve this problem and prove that the BBP transition occurs as long as the entries of $X$ and $Y$ satisfy the bounded $(8+\e)$-th moment condition (with $\e$ denoting an arbitrarily small positive constant). More precisely, we obtain the following results when $t_i>t_c$. 
\begin{itemize}
	\item[(i)] In Theorem \ref{main_thm1}, assuming that the entries of $X$, $Y$ and $Z$ have bounded moments up to any order, we prove that \smash{$n^{1/2}(\wt\lambda_i - \theta_i)$} converges weakly to a centered normal random variable. 
	
	\item[(ii)] In Theorem \ref{main_thm2}, we prove the CLT for \smash{$\wt\lambda_i$} under a relaxed bounded $(8+\e)$-th moment condition on the entries of $X,Y$ and a bounded $(4+\e)$-th moment condition on the entries of $Z$. 
\end{itemize}
On the other hand, when $t_i<t_c$, the Tracy-Widom law of $n^{2/3}( \wt\lambda_i - \lambda_+)$ was proved in \cite{PartI}. For the reader's convenience, we will state it in Theorem \ref{main_thm3}.

The proof in \cite{CCA} depends crucially on the fact that multivariate Gaussian distributions are rotationally invariant under orthogonal transforms, which makes it hard to be extended to the non-Gaussian case. To circumvent this issue, we employ an entirely different approach---a linearization method developed in \cite{PartIII}. More precisely, we define a $(p+q+2n)\times (p+q+2n)$ random matrix $H$ that is linear in $X$ and $Y$ (cf.\;equation  \eqref{linearize_block} below) and call its inverse $G:=H^{-1}$ as \emph{resolvent}. We found that the eigenvalues of the SCC matrix $\cal C_{{\cal X}{\cal Y}}$ are precisely the solutions to a determinant equation in terms of a linear functional of $G$ (cf.\;equation \eqref{detereq temp2} below). Moreover, an (almost) optimal local law for this linear functional was obtained in \cite{PartIII}. 
In \cite{PartI}, we obtained a large deviation estimate on the outlier sample CCCs: if $t_i>t_c$, then \smash{$\wt\lambda_i$} converges to $\theta_i$ with convergence rate $\OO(n^{-1/2+\e})$ (which is slightly larger than the correct order of fluctuation $n^{-1/2}$). 
With the local law and the large deviation estimate as main inputs, we can reduce the problem to proving the CLT for a (different) linear functional of $G$, denoted by $\cal E(X,Y,Z)$ (cf.\;Section \ref{sec mainthm}). 

The main technical part of our proof is to show that $\cal E(X,Y,Z)$ converges weakly to a centered Gaussian random variable. Our basic idea is to use the classical moment method, that is, showing that the moments of $\cal E(X,Y,Z)$ match those of a Gaussian random variable asymptotically. One method to calculate the moments of $\cal E(X,Y,Z)$ is to use the simple identity $1=HG$ and apply a cumulant expansion formula (cf.\;Lemma \ref{lemma_add_cumu} below) to the resulting expression. However, the calculation for this strategy will be rather tedious. Instead, we adopt a strategy in \cite{KY,KY_AOP}, that is, we first prove the CLT in an ``almost Gaussian" case (i.e., a case where most of the entries of $X$ and $Y$ are Gaussian), and then show that the general case is sufficiently close to the almost Gaussian case. This strategy allows us to divide the lengthy calculation into several parts that are more manageable. In particular, the resolvent expansion formula can be replaced by a simpler Gaussian integration by parts formula. {We refer the reader to Section \ref{sec_overview} for a more detailed review of our proof.}

Finally, we remark that the limiting variance of $n^{1/2}(\wt\lambda_i - \theta_i)$ depends on the fourth cumulants of the entries of $X$, $Y$ and $Z$ in an intricate way, which has not been identified in the Gaussian case. We also perform simulations to verify this deviation from the CLT result in \cite{CCA} (cf.\;Figure \ref{fig_simul_gauss}). 


\vspace{5pt}

\noindent{\bf Organizations.} The rest of this paper is organized as follows. In Section \ref{main_result}, we define the model and state the main results, Theorem \ref{main_thm1} and Theorem \ref{main_thm2}, on the limiting distributions of the outlier sample CCCs. {In Section \ref{sec_overview}, we introduce the linearization method, define the resolvent, and give a brief overview of the proof strategy for Theorem \ref{main_thm1} and Theorem \ref{main_thm2}. The proof of Theorem \ref{main_thm1} will be given in Sections \ref{sec pfstart}--\ref{secpfmain1}. In Section \ref{sec pfstart}, we use the linearization method to reduce the problem to showing a CLT for a linear functional of the resolvent. In Section \ref{sec almostGauss}, we establish the CLT of the outlier sample CCCs in an almost Gaussian case, where most of the entries of $X$ and $Y$ are Gaussian. Section \ref{sec Gauss} contains the proof of Lemma \ref{Gauss lemma}, which is a key lemma for the proof in Section \ref{sec almostGauss}, while Section \ref{appd GJG} gives the proof of Theorem \ref{lemmaGHF}, which is used in the proof of Lemma \ref{Gauss lemma}.} In Section \ref{secpfmain1}, we complete the proof of Theorem \ref{main_thm1} by showing that the general setting of Theorem \ref{main_thm1} is close to the almost Gaussian case asymptotically. Finally, utilizing Theorem \ref{main_thm1} and a comparison argument, we complete the proof of Theorem \ref{main_thm2} in Section \ref{pf thm2}. 

\vspace{5pt}

\noindent{\bf Conventions.} 
For two quantities $a_n$ and $b_n$ depending on $n$, the notation $a_n = \OO(b_n)$ means that $|a_n| \le C|b_n|$ for some constant $C>0$, and $a_n=\oo(b_n)$ means that $|a_n| \le c_n |b_n|$ for a positive sequence $c_n\downarrow 0$ as $n\to \infty$. We use the notation $a_n \lesssim b_n$ if $a_n = \OO(b_n)$ and the notation $a_n \sim b_n$ if $a_n = \OO(b_n)$ and $b_n = \OO(a_n)$. Given a matrix $A$, we use $\|A\|:=\|A\|_{l^2 \to l^2}$ to denote the operator norm,  $\|A\|_F $ to denote the Frobenius norm, and $\|A\|_{\max}:=\max_{i,j}|A_{ij}|$ to denote the maximum norm. Given a vector $\mathbf v=(v_i)_{i=1}^n$, $\|\mathbf v\|\equiv \|\mathbf v\|_2$ stands for the Euclidean norm. In this paper, we often write an identity matrix as $I$ or $1$ without causing any confusion. 

\vspace{5pt}

\noindent{\bf Acknowledgements.} I want to thank Zongming Ma for bringing this problem to my attention and for valuable suggestions. I also want to thank Edgar Dobriban, David Hong and Yue Sheng for fruitful discussions. I am grateful to the editor, the associated editor  and an anonymous referee for their helpful comments, which have resulted in a significant improvement. 

\section{The model and main results}\label{main_result}

\subsection{The model}\label{sec_the_model}

In this paper, we consider the model \eqref{assm data}. Here $X$ and $Y$ are two independent real matrices of dimensions $p\times n$ and $q\times n$, respectively, where the entries $X_{ij}$, $1\le i \le p$, $1\le j \le n$, and $Y_{ij}$, $1\le i \le q$, $1\le j \le n$, are i.i.d. random variables satisfying that
\begin{equation}\label{assm1}
	\mathbb{E} X_{11}=\mathbb{E} Y_{11} =0, \ \quad \ \mathbb{E} \vert X_{11} \vert^2=\mathbb{E} \vert Y_{11} \vert^2  =n^{-1}.
\end{equation}
$Z$ is an $r\times n$ random matrix that is independent of $X,Y$ and has i.i.d. entries $Z_{ij}$, $1 \le i\le r$, $1\le j \le n$, satisfying that
\begin{equation}\label{assmZ}
	\mathbb{E} Z_{11} =0, \ \quad \ \mathbb{E} \vert Z_{11} \vert^2 =n^{-1}.
\end{equation}
$A$ and $B$ are $p\times r$ and $q\times r$ deterministic matrices 
with singular value decompositions (SVD) 
\be\label{assm AB}
A = \bU_a \Sigma_a \bV_a^\top = \sum_{i=1}^r a_i \bu_i^a (\bv_i^a)^{\top} , \quad B=\bU_b \Sigma_b \bV_b^\top= \sum_{i=1}^r b_i \bu_i^b (\bv_i^b)^{\top},
\ee
where $\{a_i\}$ and $\{b_i\}$ are the singular values, $\{\bu_i^a\}$ and $\{\bu_i^b\}$ are the left singular vectors, $\{\bv_i^a\}$ and $\{\bv_i^b\}$ are the right singular vectors, and we have used the matrix notations 
\be\label{sigmaab}
\Sigma_a:=\diag \left( a_1, \cdots, a_r\right), \quad \Sigma_b:=\diag \left(b_1, \cdots, b_r\right), 
\ee
\be\label{OabUab} 
\bU_a:=\begin{pmatrix}\mathbf u_1^a, \cdots, \mathbf u_r^a \end{pmatrix}, \   \bV_a: = \begin{pmatrix} \bv_1^a , \cdots ,  \bv_r^a \end{pmatrix},\   \bU_b:= \begin{pmatrix}\mathbf u_1^b, \cdots, \mathbf u_r^b \end{pmatrix}, \   \bV_b: = \begin{pmatrix} \bv_1^b , \cdots ,  \bv_r^b \end{pmatrix} .
\ee
Recall that the PCC matrix $\wt{\bm\Sigma}$ is given by \eqref{def_PCC}. We assume that for some constant $C>0$,
\be\label{assm evalue}
0\le a_r \le \cdots \le a_2\le a_1 \le C, \quad 0\le b_r\le \cdots \le b_2\le b_1\le C.
\ee
In this paper, we focus on the high-dimensional setting, that is, there exist constants $\wt c_1$ and $\wt c_2$ such that as $n\to \infty$,
\begin{equation}
	c_1(n) := \frac{p}{n} \to \wt c_1  , \quad  c_2(n) := \frac{q}{n} \to  \wt c_2  , \quad \text{with} \quad  \wt c_1  +   \wt c_2 \in (0,1).  \label{assm2}
\end{equation}
For simplicity of notations, we will always abbreviate $c_1(n)\equiv c_1$ and $c_2(n)\equiv c_2$ in this paper. Without loss of generality, we assume that $c_1\ge c_2$. We now summarize the above assumptions for future reference. We will also assume a high moment condition on the entries of $X$, $Y$ and $Z$. 


\begin{assumption}\label{main_assm}
	Fix a small constant $\tau>0$ and a large constant $C>0$.
	\begin{itemize}
		\item[(i)] $X=(X_{ij})$ and $Y=(Y_{ij})$ are independent $p\times n$ and $q\times n$ random matrices, whose entries are real i.i.d. random variables satisfying \eqref{assm1} and the following high moment condition: for any fixed $k\in \mathbb N$, there is a constant $\mu_k>0$ such that
		\begin{equation}\label{eq_highmoment} 
			\left(\mathbb E|X_{11}|^k\right)^{1/k} \le \mu_k n^{-1/2},\quad   \left(\mathbb E|Y_{11}|^k\right)^{1/k} \le \mu_k n^{-1/2}. 
		\end{equation}
		
		
		\item[(ii)] $Z=(Z_{ij})$ is an $r\times n$ random matrix independent of $X$ and $Y$, and its entries are real i.i.d. random variables satisfying \eqref{assmZ} and \eqref{eq_highmoment}.
		
		\item[(iii)] We assume that $r\le C$ and $c_1= {p}/{n}$, $c_2 = {q}/{n}$ satisfy that
		\begin{equation}
		\tau \le c_2 \le  c_1, \quad c_1+c_2\le 1-\tau.  \label{assm20}
		\end{equation}
		
		\item[(iv)] We consider the model in \eqref{assm data}, where $A$ and $B$ satisfy \eqref{assm AB} and \eqref{assm evalue}.
		
	\end{itemize}
\end{assumption}

In this paper, we will use the SCC matrix 
\be\label{eq_SCC}\cal C_{{\cal X}{\cal Y}}:= \left({\cal X}{\cal X}^{\top}\right)^{-1/2} \left({\cal X}{\cal Y}^{\top}\right)\left({\cal Y}{\cal Y}^{\top}\right)^{-1}\left({\cal Y}{\cal X}^{\top}\right)  \left({\cal X}{\cal X}^{\top}\right)^{-1/2},\ee
and the null SCC matrix 
\be\label{eq_nullSCC}\cal C_{XY}:= S_{xx}^{-1/2} S_{xy} S_{yy}^{-1}S_{yx}S_{xx}^{-1/2} ,\ee
with
\be\label{def Sxy}S_{xx} := {X}{ X}^{\top}, \quad S_{yy} := {Y}{ Y}^{\top}, \quad S_{xy} = S^{\top}_{yx}:=XY^{\top}.\ee
We will also use the following SCC and null SCC matrices: 
$$\cal C_{{\cal Y}{\cal X}}:= \left({\cal Y}{\cal Y}^{\top}\right)^{-1/2} \left({\cal Y}{\cal X}^{\top}\right)\left({\cal X}{\cal X}^{\top}\right)^{-1}\left({\cal X}{\cal Y}^{\top}\right)  \left({\cal Y}{\cal Y}^{\top}\right)^{-1/2},\quad \cal C_{YX}= S_{yy}^{-1/2} S_{yx} S_{xx}^{-1}S_{xy}S_{yy}^{-1/2}.$$

Our results can be easily extended to a more general model  
\be\label{assm data0}{\mathcal X}: =\mathbf C_1^{1/2}X+ AZ , \quad  {\mathcal Y}: =\mathbf C_2^{1/2}Y+ B Z, \ee
with non-identity population covariance matrices $\mathbf C_1$ and $\mathbf C_2$. In fact, it is easy to see that the eigenvalues of the SCC matrix $\cal C_{{\cal X}{\cal Y}}$ are unchanged under the non-singular transformations  $ {\mathcal X} \to   \mathbf C_1^{-1/2} {\mathcal X} $ and ${\mathcal Y} \to   \mathbf C_2^{-1/2} {\mathcal Y}  ,$ which reduce \eqref{assm data0} to the model \eqref{assm data} with $A$ and $B$ replaced by $ \mathbf C_1^{-1/2} A$ and $  \mathbf C_2^{-1/2} B$.


\subsection{The main results}

We denote the eigenvalues of the null SCC matrix $\cal C_{YX}$ by $ \lambda_1 \ge \lambda_2 \ge \cdots \ge \lambda_q\ge 0$. It is easy to see that $\cal C_{ X Y}$ shares the same eigenvalues with $\cal C_{ Y X}$, besides the $p-q$ more trivial zero eigenvalues $ \lambda_ {q+1} = \cdots = \lambda_{p}=0$. We denote the ESD of $\cal C_{YX}$ by 
$$ F_n(x):= \frac1q\sum_{i=1}^q \mathbf 1_{\lambda_i \le x}.$$
It has been proved in \cite{Wachter,CCA_ESD} that, almost surely, $F_n$ converges weakly to a deterministic probability distribution $F(x)$ with density 
\be\label{LSD}
f(x)= \frac{1}{2\pi c_2} \frac{\sqrt{(\lambda_+ - x)(x-\lambda_-)}}{x(1-x)}, \quad \lambda_- \le x \le \lambda_+,
\ee
where the left edge $\lambda_-$ and the right edge $\lambda_+$ of the density are defined as
\be\label{lambdapm}\lambda_\pm:= \left( \sqrt{c_1(1-c_2)} \pm \sqrt{c_2(1-c_1)}\right)^2.\ee

Under the setting of \eqref{assm data}, we denote the eigenvalues of $\cal C_{\cal X\cal Y}$ by $\wt \lambda_1  \ge \cdots \ge \wt\lambda_q \ge \wt\lambda_{q+1} =\cdots =\wt \lambda_p= 0$, 
and the eigenvalues of the PCC matrix 
$\wt{\bm\Sigma}$ by $t_1 \ge t_2 \ge \cdots \ge t_r \ge t_{r+1} = \cdots = t_p=0.$
Recall the threshold $t_c$ for BBP transition defined in \eqref{tc}. 
Assuming the entries of $X$ and $Y$ are i.i.d.\;Gaussian, it was proved in \cite{CCA} that 
for any $1\le i \le r$, if $t_i\le t_c$, then \smash{$\wt\lambda_i - \lambda_+\to 0$} almost surely, while if $t_i>t_c$, then \smash{$\wt\lambda_i - \theta_i \to 0$} almost surely, where 
\be\label{gammai}
\theta_i : = t_i\left( 1-c_1+c_1t_i^{-1}\right) \left( 1-c_2+c_2t_i^{-1}\right) .
\ee
Moreover, the limiting distributions were also identified in \cite{CCA}: if $t_i<t_c$, \smash{$n^{2/3}(\wt\lambda_i - \lambda_+)$} converges to the Tracy-Widom law; if $t_i>t_c$, \smash{$\sqrt{n}(\wt\lambda_i - \theta_i)$} converges to a centered normal distribution. The main purpose of this paper is to extend the CLT of the outliers to the setting in Section \ref{sec_the_model}, assuming only the moment conditions in \eqref{eq_highmoment} (or the weaker ones in \eqref{eq_8moment} below).


In \cite{CCA}, it was assumed that the population CCCs are either well-separate or exactly degenerate. In this paper, however, we consider the general setting which allows for near-degenerate outliers. For this purpose, we first introduce some new notations following \cite{KY_AOP}. For any $r\times r$ matrix $\cal A=(A_{ij})$ and a subset of indices $\pi \subset \{1,\cdots, r\}$, we define the $|\pi| \times |\pi|$ submatrix 
\be\label{Api}\cal A_{\llbracket \pi \rrbracket}:=(A_{ij})_{i,j\in \pi}.\ee
We arrange the eigenvalues of $\cal A_{\llbracket \pi \rrbracket}$ in descending order as 
\be\label{ev minor}\mu_1\left(\cal A_{\llbracket \pi \rrbracket}\right)\ge \cdots \ge \mu_{|\pi|}\left(\cal A_{\llbracket \pi \rrbracket}\right).\ee
We will group the near-degenerate $t_i$-s according to the following definition.

\begin{definition}\label{Def gammal}
	Fix two small constants $\delta_l,\delta>0$. For $l\in \{1,\cdots,r\}$ satisfying 
	\be\label{tlc} t_c + \delta_l \le t_l \le 1- \delta_l, \ee
	we define the subset $\gamma(l) \ni l$ as the \emph{smallest} subset of $\{1,\cdots, r\}$ such that the following property holds: if $i,j\in \{1,\cdots, r\}$ satisfy $t_i>t_c$ and $ |t_i - t_j| \le n^{-1/2+\delta},$
	then either $i,j \in \gamma(l)$ or $i,j \notin \gamma(l)$. 
\end{definition}
The set $\gamma(l)$ in this definition can be constructed by successively choosing $i\in \{1,\cdots, r\}$ such that $t_i$ is away from the set $\{t_j: j\in \gamma(l)\}$ by a distance $\le n^{-1/2+\delta}$, and then adding $i$ to $\gamma(l)$. Since the number of such indices is at most $r$, we have that $|t_i-t_l|\le rn^{-1/2+\delta}$ for any $i\in \gamma(l)$.
Now, we are ready to state the first main result, which describes the joint limiting distribution of a group of near-degenerate outliers indexed by indices in $\gamma(l)$.  

\begin{theorem}\label{main_thm1}
Fix any $1\le l \le r$. Suppose Assumption \ref{main_assm} holds, and there exists a constant $\delta_l>0$ such that \eqref{tlc} holds. Define the vector of rescaled eigenvalues $\bm \zeta=(\zeta_i)_{i\in \gamma(l)}$, where $\zeta_i := n^{1/2} (\wt \lambda_i - \theta_l)$ for $\theta_l$ defined in \eqref{gammai}. Let $\bm \xi=(\xi_i)_{i\in \gamma(l)}$ be the vector of the eigenvalues (in descending order) of the random $|\gamma(l)|\times |\gamma(l)|$ matrix 
\be\label{Upsilonl}
a(t_l)\left\{n^{1/2}\left[ \diag (t_1, \cdots, t_r) - t_l\right]_{\llbracket\gamma(l)\rrbracket} + \Upsilon_l \right\},
\ee
where $a(t_l)$ is a function of $t_l$ defined as
\be\label{cgtl} a(t_l):= \frac{(1-c_1)(1-c_2)}{t_l^2} (t_l^2 - t_c^2),\ee
$\left[ \cdot\right]_{\llbracket\gamma(l)\rrbracket}$ is defined in \eqref{Api} with $\pi=\gamma(l)$,
and $\Upsilon_l$ is a $|\gamma(l)|\times |\gamma(l)|$ symmetric Gaussian random matrix, whose entries have zero mean and covariance function
\be\label{covariancey}
\mathbb E (\Upsilon_l)_{ij}(\Upsilon_l)_{i'j'}= C_{ij,i'j'}(t_l),\quad \text{for}\quad (i,j), (i', j')\in \gamma(l)\times \gamma(l). \ee
The function $C_{ij,i'j'}(t_l)$ will be defined in equation \eqref{Cijij} below. Then, for any bounded continuous function $f: \R^{|\gamma(l)|}\to \R$, we have that
\be\label{limf} \lim_n \left[ \E f(\bm\zeta)-\E f(\bm\xi)\right]=0. \ee
\end{theorem}



Roughly speaking, the above theorem means that the eigenvalues around $\wt\lambda_l$ converge in distribution to the eigenvalues of a  symmetric Gaussian random matrix. The mean of this Gaussian matrix is a diagonal matrix depending on the rescaled gaps $ n^{1/2} (t_i-t_l)$, $i\in \gamma(l)$. We now give the explicit expressions of the covariance function. 
Using the SVD \eqref{assm AB}, we can rewrite the PCC matrix $\wt{\bm\Sigma}$ in \eqref{def_PCC} as
\begin{align*}
\wt{\bm\Sigma}=\bU_a \left[\frac{\Sigma_a}{(I_r+\Sigma_a^2)^{1/2}}\bV_a^{\top} \bV_b \frac{\Sigma_b^2}{I_r+\Sigma_b^2}\bV_b^{\top} \bV_a  \frac{\Sigma_a}{(I_r+\Sigma_a^2)^{1/2}} \right]\bU_a^\top .
\end{align*}
Hence, the matrix inside brackets has eigenvalues $t_1\ge \cdots \ge t_r$. Now, suppose we have the following SVD
\be\label{Mrab}
\frac{\Sigma_a}{(I_r+\Sigma_a^2)^{1/2}}\bV_a^{\top} \bV_b \frac{\Sigma_b}{(I_r+\Sigma_b^2)^{1/2}} = \cal O \diag (\sqrt{t_1}, \cdots, \sqrt{t_r}) \wO^{\top},
\ee
for two $r\times r$ orthogonal matrices $\cal O$ and $\wt{\cal O}$. 
Then, for $k\in \{1, \cdots n\}$ and $i,j\in \{1,\cdots, r\}$, we define 
\begin{align}
	\cal W_{k,ij}:=&\, t_l \left(W_a\right)_{ki}\left(W_a\right)_{kj} +  t_l  \left(W_b\right)_{ki} \left(W_b\right)_{kj} \nonumber\\
	& -  \sqrt{t_l} \left(W_a\right)_{ki} \left(W_b\right)_{kj} -  \sqrt{t_l}\left(W_b\right)_{ki} \left(W_a\right)_{kj}   ,\label{Wij}
\end{align}
where $W_a$ and $W_b$ are two $n\times r$ matrices defined by 
$$W_a:= \bV_a \frac{\Sigma_a}{(I_r+\Sigma_a^2)^{1/2}}\cal O,\quad W_b:=\bV_b \frac{\Sigma_b}{(I_r+\Sigma_b^2)^{1/2}} \wO .$$
Moreover, we define the $p\times r$ and $q\times r$ matrices
\be\label{defUV}
\begin{split}
	& \cal U :=  \bU_a { (I_{r} + \Sigma_a^2)^{-1/2}} \cal O  , \quad  \cal V:=  \bU_b {(I_r +\Sigma_b^2)^{-1/2}}\wO  .
\end{split}
\ee
Then, the covariance function $C_{ij,i'j'}(t_l)$ for $(i,j), (i', j')\in \gamma(l)\times \gamma(l)$ is defined as 
\begin{align}
 &	C_{ij,i'j'}(t_l) := \frac{(1-t_l)^2 t_l^2}{t_l^2 - t_c^2}\left(2 t_l+ \frac{c_1}{1-c_1}+\frac{c_2}{1-c_2}  \right)\left(\delta_{ii'} \delta_{jj'} +\delta_{ij'} \delta_{ji'} \right) \nonumber\\
& +t_l^2 \kappa_x^{(4)}\sum_{k} \cal U_{ki}\cal U_{ki'}\cal U_{kj}\cal U_{kj'}+ t_l^2 \kappa_y^{(4)} \sum_{k} \cal V_{ki}\cal V_{ki'}\cal V_{kj}\cal V_{kj'}  + \kappa_z^{(4)}  \sum_{k} \cal W_{k,ij}\cal W_{k,i'j'}  , \label{Cijij}
\end{align}
where we have introduced the notations 
\be\label{mux4}
\kappa_x^{(4)}:= n^{2}\mathbb EX_{11}^4 - 3, \quad \kappa_y^{(4)}:= n^{2}\mathbb EY_{11}^4 -3,\quad \kappa_z^{(4)}:= n^{2}\mathbb EZ_{11}^4 -3,\ee
which are the fourth cumulants of $\sqrt{n}X_{11}$, $\sqrt{n}Y_{11}$, and $\sqrt{n}Z_{11}$.


We apply our result to the special case where the entries of $X$, $Y$ and $Z$ are i.i.d.\;Gaussian random variables, and $t_i=t_l$ for all $i\in \gamma(l)$. In this case, the last three terms in \eqref{Cijij} vanish and $\left[ \diag (t_1, \cdots, t_r) - t_l\right]_{\llbracket\gamma(l)\rrbracket}=0$. 
Hence, by Theorem \ref{main_thm1}, $\bm \zeta$ converges weakly to the ordered eigenvalues of a GOE (Gaussian orthogonal ensemble) $\mathbf G=(g_{ij})$, with independent Gaussian entries 
\be\label{Gauss_var} g_{ij}= g_{ji} \sim \cal N(0, (1+\delta_{ij} )\sigma^2(t_l)), \ee 
where
\be\label{defn_cg} \sigma^2(t_l):=  \frac{(1-c_1)^2(1-c_2)^2(1-t_l)^2(t_l^2-t_c^2)}{t_l^2} \left( 2t_l + \frac{c_1}{1-c_1} + \frac{c_2}{1-c_2}\right)  .\ee 
This is in accordance with \cite[Theorem 1.9]{CCA}.

The next theorem shows that if we assume that the population CCCs are either well-separated or exactly degenerate (cf. condition \eqref{well-sep}), then the CLT of the outlier eigenvalues in Theorem \ref{main_thm1} also holds under the relaxed moment assumption \eqref{eq_8moment}.

\begin{theorem}\label{main_thm2}
	Fix any $1\le l \le r$. Suppose Assumption \ref{main_assm} holds except that \eqref{eq_highmoment} is replaced with the following moment assumption: there exist constants $c_0, C_0>0$ such that 
	\begin{equation}\label{eq_8moment} 
		\mathbb E |\sqrt{n}X_{11}|^{8+c_0} \le C_0,\quad   \mathbb E |\sqrt{n}Y_{11}|^{8+c_0} \le C_0,\quad \mathbb E |\sqrt{n}Z_{11}|^{4+c_0} \le C_0. 
	\end{equation}
	Suppose there exists a constant $\delta_l>0$ such that \eqref{tlc} holds, and 
	\be\label{well-sep}
	\text{$t_i = t_l$ \ \ for \ \ $i\in \gamma(l)$},\quad \text{ and }\quad \text{$|t_i - t_l| \ge \delta_l$\ \ for \ \ $i\notin \gamma(l)$.} 
	\ee
	Then, \eqref{limf} holds for $\bm\zeta$ and $\bm\xi$ defined in Theorem \ref{main_thm1}. 
\end{theorem}

On the other hand, the limiting Tracy-Widom distribution of the extreme non-outlier eigenvalues has been proved under a fourth moment tail assumption in \cite{PartI}. 

\begin{theorem}[Theorem 2.14 of \cite{PartI}]\label{main_thm3}
	Suppose Assumption \ref{main_assm} (iii)-(iv) hold. Assume that $x_{ij}=n^{-1/2} \wh x_{ij}$, $y_{ij}= n^{-1/2} \wh y_{ij}$ and $z_{ij}= n^{-1/2} \wh z_{ij}$, where $\{\wh x_{ij}\}$, $\{\wh y_{ij}\}$ and $\{\wh z_{ij}\}$ are three independent families of real i.i.d. random variables of mean zero and variance one. Moreover, we assume the fourth moment tail condition  
	\begin{equation}
		\lim_{t \rightarrow \infty } t^4 \left[\mathbb{P}\left( \vert \wh x_{11} \vert \geq t\right)+ \mathbb{P}\left( \vert \wh y_{11} \vert \geq t\right)\right]=0 .\label{tail_cond}
	\end{equation}
	Assume that for a fixed $0\le r_+ \le r$, the eigenvalues of $\wt{\mathbf \Sigma}$ satisfy that  
	\be\label{well-sep1} \liminf_n t_{r_+} > t_c > \limsup_{n} t_{r_+ + 1}.\ee
	Then, we have that for any fixed $k\in \N$ and $(s_1 , s_2, \ldots, s_k) \in \mathbb R^k$, 
	\be\label{boundinTW}
	\begin{split}
		\lim_{n\to \infty}\mathbb{P}&\left[ \left(n^{{2}/{3}}\frac{\wt\lambda_{r_+ + i} - \lambda_+}{c_{TW}} \leq s_i\right)_{ i =1}^k \right] = \lim_{n\to \infty} \mathbb{P}^{GOE}\left[\left(n^{{2}/{3}}(\lambda_i - 2) \leq s_i\right)_{ i =1}^k \right], 
	\end{split}
	\ee
	where
	$$c_{TW}:= \left[ \frac{\lambda_+^2 (1-\lambda_+)^2}{\sqrt{c_1c_2(1-c_1)(1-c_2)}}\right]^{1/3},$$
	and $\mathbb P^{GOE}$ stands for the law of GOE, referring to an $n\times n$ symmetric matrix with independent Gaussian entries of mean zero and variance $n^{-1}$. 
\end{theorem}

The assumption \eqref{well-sep1} means that $t_i$, $1\le i \le r_+$, are supercritical spikes that lead to outlier eigenvalues, while $t_i$, $r_+ +1 \le i \le r$, are subcritical spikes. Hence, \smash{$\wt\lambda_{r_+ + i}$} is the $i$-th non-outlier eigenvalue of the SCC matrix, and \eqref{boundinTW} gives a complete description of the asymptotic joint distribution of the first $k$ non-outlier eigenvalues of $\cal C_{\cal X\cal Y}$ in terms of the extreme eigenvalues of GOE. 
Taking $k=1$ in \eqref{boundinTW} shows that the first (rescaled) non-outlier eigenvalue \smash{$n^{{2}/{3}}(\wt\lambda_{r_+ + 1} - \lambda_+)/{c_{TW}}$} converges weakly to the type-1 Tracy-Widom distribution \cite{TW1,TW}. For a general $k\in \N$, the joint distribution of the largest $k$ eigenvalues of GOE can be written in terms of the Airy kernel \cite{Forr}. 

Combining Theorems \ref{main_thm1}, \ref{main_thm2} and \ref{main_thm3}, we complete the story of BBP transition for high-dimensional CCA with finite rank correlations. 


\subsection{Simulations}
In this subsection, we verify Theorem \ref{main_thm1} with some numerical simulations. In particular, we will show that the last three terms in \eqref{Cijij}, which depend on the fourth cumulants $\kappa_x^{(4)}$, $\kappa_y^{(4)}$ and $\kappa_z^{(4)}$, are necessary to match the variance of the simulated sample CCC. For our simulations, we take the entries of $X$, $Y$ and $Z$ to be i.i.d.\;Rademacher random variables (with an extra scaling $n^{-1/2}$). In this setting, we have \smash{$\kappa_x^{(4)}=\kappa_y^{(4)}=\kappa_z^{(4)}=-2$}. Moreover, we take $n=2000$ and $c_1=c_2=0.2$, i.e. $p=q=400$, which gives $t_c=0.25$ by \eqref{tc}. We consider the rank-one case with $r=1$ and take the matrices $A$ and $B$ as $A= a_1 \mathbf u^a $ and $B= b_1 \mathbf u^b $ with $a_1=b_1=2$, which gives a supercritical spike $t_1=0.64$. We consider the following two scenarios for the unit vectors $\mathbf u^a$ and $\mathbf u^b$.

\vspace{3pt}
\noindent {\bf Scenario (a)}: $\mathbf u^a$ and $\mathbf u^b$ are standard unit vectors along the first coordinate axis in $\R^p$ and $\R^q$, respectively. In this case, the limiting variance of \smash{$\zeta_1=n^{1/2}(\wt\lambda_1-\theta_1)$} is given by $\sigma_a^2:= a^2(t_1)C_{11,11}(t_1)$, where $C_{11,11}(t_1)$ is defined in \eqref{Cijij}: 
\begin{align*}
	C_{11,11}(t_1) =&\ 2\frac{(1-t_1)^2 t_1^2}{t_1^2 - t_c^2}\left(2 t_1+ \frac{c_1}{1-c_1}+\frac{c_2}{1-c_2}  \right)-2t_1^2 \left[\frac{1}{(1+a_1^2)^2}  + \frac{1}{(1+b_1^2)^2} \right] \\
	&- 2\left[t_1 \frac{a_1^2}{1+a_1^2} + t_1 \frac{b_1^2}{1+b_1^2}-2\sqrt{t_1} \frac{a_1b_1}{(1+a_1^2)^{1/2}(1+b_1^2)^{1/2}}   \right]^2.  
\end{align*}

\vspace{3pt}
\noindent {\bf Scenario (b)}: $\mathbf u^a$ and $\mathbf u^b$ are random unit vectors on the unit spheres $\mathbb S^p$ and $\mathbb S^q$, respectively. Then we have $\|\mathbf u^a\|_\infty \le n^{-1/2+\e}$ and $\|\mathbf u^b\|_\infty \le n^{-1/2+\e}$ with probability $1-\oo(1)$ for any constant $\e>0$, with which we can easily check that the {$\kappa_x^{(4)}$ and $\kappa_y^{(4)}$} terms in \eqref{Cijij} are both of order $\OO(n^{-1+2\e})$ with probability $1-\oo(1)$. Hence the limiting variance of $\zeta_1 $ is given by $\sigma_b^2:= a^2(t_1)C_{11,11}(t_1)$, where 
\begin{align*}
	C_{11,11}(t_1)=&\ 2\frac{(1-t_1)^2 t_1^2}{t_1^2 - t_c^2}\left(2 t_1+ \frac{c_1}{1-c_1}+\frac{c_2}{1-c_2}  \right) \\
	&-2\left[t_1 \frac{a_1^2}{1+a_1^2} + t_1 \frac{b_1^2}{1+b_1^2}-2\sqrt{t_1} \frac{a_1b_1}{(1+a_1^2)^{1/2}(1+b_1^2)^{1/2}}   \right]^2 + \OO(n^{-1+2\e}),
\end{align*}
 with probability $1-\oo(1)$. 
 
\vspace{3pt}

 In Figure \ref{fig_simul_gauss}, we report the simulation results based on $10^5$ replications. We find that the histograms match our result in Theorem \ref{main_thm1} pretty well. Furthermore, it is not surprising that the {prediction} \eqref{Gauss_var} in the Gaussian setting deviates from the simulations, which shows that the last three terms in \eqref{Cijij} are necessary for non-Gaussian settings.

\begin{figure}[!ht]
	\includegraphics[width=\textwidth]{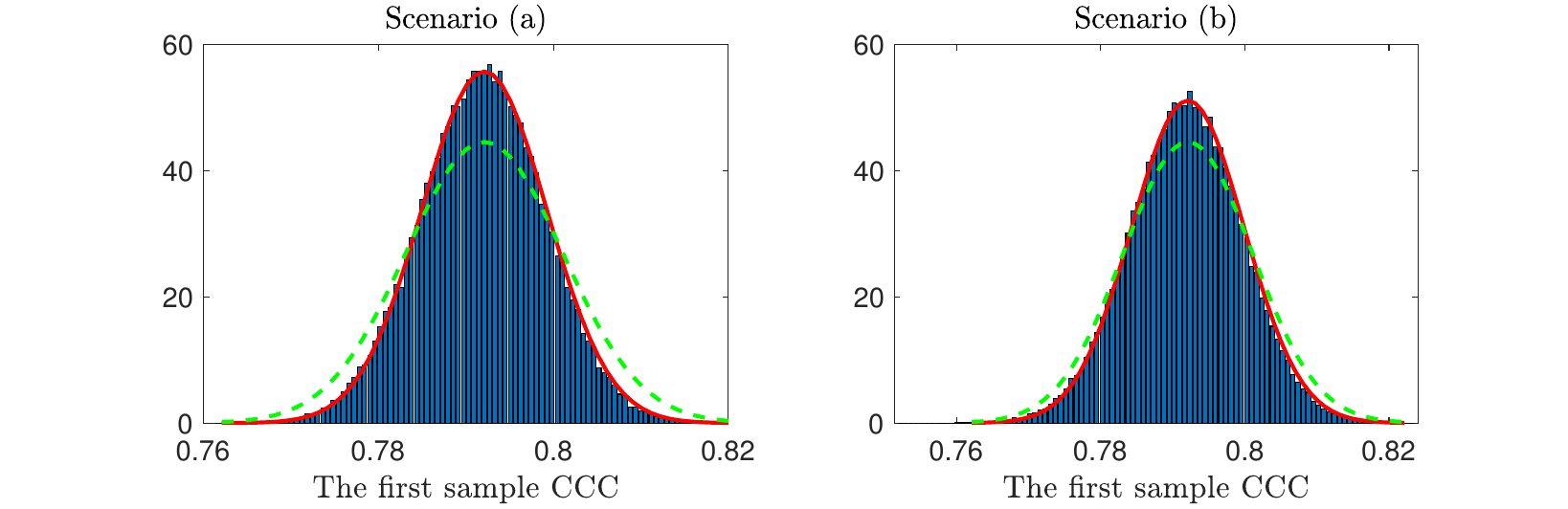}
	\caption{The histograms give the simulated first sample CCC based on {$10^5$ replications}. The red solid curves give the probability density functions (PDF) of the normal distributions $\cal N(\theta_1, \sigma_a^2/n)$ and $\cal N(\theta_1, \sigma_b^2/n)$ in scenarios (a) and (b), respectively. The green dashed curves represent the PDF of the normal distribution $\cal N(\theta_1, 2\sigma^2(t_1)/n)$, where $\sigma^2(t_1) $ is defined in \eqref{defn_cg}.}
	\label{fig_simul_gauss}
\end{figure}

\subsection{Relation with \cite{PartI} and \cite{PartIII}}
This paper is the third part of a series of papers with \cite{PartIII} and \cite{PartI} being the first two parts. The main goal of this series is to establish the BBP transition of sample CCCs in the setting of high-dimensional CCA with finite rank correlations and without Gaussian assumptions. 

In the first part \cite{PartIII}, we considered the null case with $r=0$ and developed a new linearization method for the study of sample CCCs. More precisely, we introduce a $(p+q+2n)\times (p+q+2n)$ linearized matrix $H(z)$ in terms of $X$, $Y$ and a spectral parameter $z\in \mathbb C$ (cf.\;equation \eqref{linearize_block}), so that the eigenvalues of the SCC matrix are exactly the solutions to the equation $\det H (z) =0$. In \cite{PartIII}, we studied this equation through its inverse $G(z):=H(z)^{-1}$, called the \emph{resolvent}. The main result of \cite{PartIII} is an optimal large deviation estimate, called the \emph{anisotropic local law}, on $G(z)$ (cf. Theorem \ref{thm_local} below).
As consequences of the anisotropic local law, we also proved a sharp eigenvalue rigidity estimate for the null SCC matrix $\cal C_{ X Y}$ (cf. Lemma \ref{lem null} below) and the Tracy-Widom law of the largest eigenvalue of $\cal C_{XY}$, which is a special case of Theorem \ref{main_thm3} with $r=0$. 

In the second part \cite{PartI}, we considered the model \eqref{assm data} with $r>0$. In particular, we showed that the eigenvalues \smash{$\wt\lambda_i$, $1\le i \le p\wedge q$}, of $\cal C_{{\cal X}{\cal Y}}$ are precisely the solutions to a determinant equation in terms of a linear functional of $G(z)$ and the matrices in the SVD \eqref{assm AB}, see equation \eqref{detereq temp2} below. Then, based on the anisotropic local law and the eigenvalue rigidity estimate obtained in \cite{PartIII}, we proved Theorem \ref{main_thm3} regarding the Tracy-Widom law of the extreme non-outlier eigenvalues. In addition, we also proved in \cite{PartI} that the outlier sample CCC \smash{$\wt\lambda_i$} corresponding to a supercritical spike $t_i>t_c$ converges to $\theta_i$ with a sharp convergence rate $\OO(n^{-1/2+\e})$ (cf. Lemma \ref{main_thm}).  

Finally, in this paper, we complete the theory of BBP transition for high-dimensional non-Gaussian CCA by showing the CLT of the outlier eigenvalues, that is, Theorem \ref{main_thm1} and Theorem \ref{main_thm2}. In the proof of these results, we first reduce the problem to proving the CLT for a linear functional of $G$ (cf. Proposition \ref{redGthm} and equation \eqref{wtMx}) by using the anisotropic local law, Theorem \ref{thm_local}, obtained in \cite{PartIII} and the convergence estimate of outlier eigenvalues, Lemma \ref{main_thm}, obtained in \cite{PartI}. Then, the main part of our proof is to show that the linear functional of $G$ converges weakly to a centered Gaussian random matrix. Again, the anisotropic local law, Theorem \ref{thm_local}, is the key tool for this proof. We refer the reader to Section \ref{sec_overview} for a brief overview of the proof and to Sections \ref{sec pfstart}--\ref{pf thm2} for complete details.

\section{Overview of the proof} \label{sec_overview}

In this section, we give a brief overview of the proof for Theorem \ref{main_thm1}. 
The starting point of our proof is the following self-adjoint linearization trick developed in \cite{PartI,PartIII}, that is, a $ \lambda \in (0,1)$ is an eigenvalue of $\cal C_{\cal X\cal Y}$ if and only if the following equation holds:
\be\label{deteq}\det \begin{bmatrix} 0 & \begin{pmatrix} \cal X & 0\\ 0 & \cal Y\end{pmatrix}\\ \begin{pmatrix} \cal X^{\top} & 0\\ 0 & \cal Y^{\top}\end{pmatrix}  & \begin{pmatrix}  \lambda  I_n & \lambda^{1/2}I_n\\ \lambda^{1/2} I_n &  \lambda I_n\end{pmatrix}^{-1}\end{bmatrix} = 0 . \ee
Inspired by this equation, we define the following $(p+q+2n)\times (p+q+2n)$ self-adjoint block matrix 
\begin{equation}\label{linearize_block}
	H(\lambda)\equiv H(X,Y,\lambda) : = \begin{bmatrix} 0 & \begin{pmatrix}  X & 0\\ 0 &  Y\end{pmatrix}\\ \begin{pmatrix} X^{\top} & 0\\ 0 &  Y^{\top}\end{pmatrix}  & \begin{pmatrix}  \lambda  I_n & \lambda^{1/2}I_n\\ \lambda^{1/2} I_n &  \lambda I_n\end{pmatrix}^{-1}\end{bmatrix} ,
\end{equation}
and call its inverse the \emph{resolvent}:
\be\label{def_resolvent}
G(\lambda)\equiv G(X,Y,\lambda) := \left[H(X,Y,\lambda)\right]^{-1}.
\ee
In this paper, we extend the argument $\lambda$ to $z\in \mathbb C_+:=\{z\in \mathbb C: \im z>0\}$ with $z^{1/2}$ being the branch with positive imaginary part. Similar to equation \eqref{deteq}, it is not hard to see that $\lambda$ is \emph{not} an eigenvalue of the null SCC matrix if and only if $\det \left[H(\lambda)\right] \ne 0$. Hence, for $\lambda \notin \text{Spec}(\cal C_{XY})$, using \eqref{assm data}, \eqref{assm AB}, \eqref{sigmaab} and \eqref{OabUab}, we can rewrite \eqref{deteq} as
\begin{align}
	0 & = \det \left[ 1+\begin{pmatrix} \bU & 0 \\ 0 & \bV\end{pmatrix}\begin{pmatrix} 0 & \cal D\\ \cal D  & 0\end{pmatrix}\begin{pmatrix} \bU^{\top} & 0 \\ 0 & \bV^{\top}\end{pmatrix} G(\lambda) \right] \nonumber\\
	&=\det \left[ 1+\begin{pmatrix} 0 & \cal D\\ \cal D  & 0\end{pmatrix}\begin{pmatrix} \bU^{\top} & 0 \\ 0 & \bV^{\top}\end{pmatrix} G(\lambda) \begin{pmatrix} \bU & 0 \\ 0 & \bV\end{pmatrix}\right],\label{detereq temp2}
\end{align} 
where we have used the identity $\det(1+M_1M_2)=\det(1+M_2M_1)$ for any two matrices $M_1$ and $M_2$ of conformable dimensions. Here, $\cal D$, $\bU$ and $\bV$ are $2r\times 2r$, $(p+q)\times 2r$ and $2n\times 2r$ matrices defined as
$$\cal D:=\begin{pmatrix} \Sigma_a & 0 \\ 0 & \Sigma_b\end{pmatrix}, \quad \bU := \begin{pmatrix} \bU_a  & 0 \\ 0 & \bU_b  \end{pmatrix}, \quad \bV:= \begin{pmatrix}  Z^{\top} \bV_a  & 0 \\ 0 & Z^{\top} \bV_b \end{pmatrix} .$$

By the anisotropic local law in Theorem \ref{thm_local}, $G(\lambda)$ in equation \eqref{detereq temp2} can be replaced by a deterministic matrix, denoted by $\Pi(\lambda)$, up to a small error:
\be\label{deteq_2} \det \left\{ 1+\begin{pmatrix} 0 & \cal D\\ \cal D  & 0\end{pmatrix}\left[\begin{pmatrix} \bU^{\top} & 0 \\ 0 & \bV^{\top}\end{pmatrix} \Pi(\lambda) \begin{pmatrix} \bU & 0 \\ 0 & \bV\end{pmatrix} + \cal E(\lambda)\right]\right\}=0, \ee
where 
$$ \cal E(\lambda):= \begin{pmatrix} \bU^{\top} & 0 \\ 0 & \bV^{\top}\end{pmatrix} \left[G(\lambda)-\Pi(\lambda)\right] \begin{pmatrix} \bU & 0 \\ 0 & \bV\end{pmatrix}.$$
Using the definition of $\Pi$ in equation \eqref{defn_pi} below, we can check that if we set $\cal E(\lambda)=0$ in \eqref{deteq_2}, then the resulting deterministic equation has a solution $\lambda=\theta_l$ if $t_l$ is supercritical. Moreover, Theorem \ref{thm_local} shows that $\|\cal E(\lambda)\| \le n^{-1/2+\e}$ with high probability (cf. Definition \ref{stoch_domination} (iv)) for any constant $\e>0$. With this fact, we proved in \cite{PartI} that \smash{$|\wt\lambda_l-\theta_l|\le n^{-1/2+\e}$} with high probability. Thus, performing a Taylor expansion of equation \eqref{deteq_2} around $\theta_l$, we obtain that with high probability,
\be\nonumber
\begin{split}
\det \left\{ 1+\begin{pmatrix} 0 & \cal D\\ \cal D  & 0\end{pmatrix}\left[\begin{pmatrix} \bU^{\top} & 0 \\ 0 & \bV^{\top}\end{pmatrix} \left(\Pi(\theta_l)  + (\wt\lambda_l-\theta_l) \Pi'(\theta_l) \right)\begin{pmatrix} \bU & 0 \\ 0 & \bV\end{pmatrix} + \cal E(\theta_l)\right]\right\}\\
=\OO(n^{-1+2\e}).
\end{split}
\ee
This equation suggests that the limiting distribution of \smash{$n^{1/2}(\wt\lambda_l-\theta_l)$} should be determined by that of $n^{1/2}\cal E(\theta_l)$. In fact, through calculations in Section \ref{sec mainthm}, we find that \smash{$n^{1/2}(\wt\lambda_l-\theta_l)$} is related to a more complicated linear function of $G(\theta_l)-\Pi(\theta_l)$ given in \eqref{master_er}. We refer the reader to Proposition \ref{redGthm} below for a precise statement. 

Now, roughly speaking, our problem has been reduced to showing the CLT for a linear function of $G(\theta_l)-\Pi(\theta_l)$. Through a direct calculation, we can further reduce the problem to showing the CLT of a matrix of the form (cf. equation \eqref{wtMx})  
\be\label{eq_upsilon0} 
\Upsilon_0:=n^{1/2} \bW^\top \left[G(\theta_l)-\Pi(\theta_l)\right]\bW,
\ee
where $\bW$ is a $4r\times(p+q+n)$ matrix independent of $X$ and $Y$. To illustrate the basic idea, we describe the strategy of the proof for the following quantity:
\be\label{eq_upsilon} 
\Upsilon:=n^{1/2} \bw^\top \left[G(\theta_l)-\Pi(\theta_l)\right]\bw,
\ee
where $\bw$ is a $(p+q+n)$-dimensional vector independent of $X$ and $Y$.
In general, to show that $\Upsilon_0$ in \eqref{eq_upsilon0} converges weakly to a Gaussian matrix, we can adopt the Cram{\'e}r-Wold device, that is, we will show that
$$ n^{1/2}\sum_{ 1\le i\le j \le 4r} \lambda_{ij} (\Upsilon_0)_{ij} $$ 
is asymptotically Gaussian for any fixed vector of parameters $(\lambda_{ij})_{1\le i\le j\le 4r}$. This can be proved using the same strategy as the proof of the CLT for $\Upsilon$, which we will discuss now.

In order to prove that $\Upsilon$ is asymptotically Gaussian, we will show that its moments match those of a Gaussian random variable as $n\to \infty$. It suffices to prove the zero mean condition $\mathbb E \Upsilon \to 0$ and the induction relation: for any fixed integer $k\ge 2$, 
\be\label{idea_induction}
\mathbb E\Upsilon^k = (k-1) \sigma^2 \mathbb E\Upsilon^{k-2} + \oo(1)
\ee
for some deterministic parameter $\sigma^2$, which determines the variance of the limiting Gaussian distribution. We will describe some basic ideas for the proof of \eqref{idea_induction}, while the mean condition can be regarded as a special case with $k=1$. 
Using the definition of $G$, we can write that $G-\Pi=\Pi \left( \Pi^{-1}  - H\right)G$, and hence 
$$\mathbb E\Upsilon^k  = n^{1/2}\mathbb E\Upsilon^{k-1}  \bw^\top \Pi(\theta_l) \left[ \Pi^{-1}(\theta_l)  - H(\theta_l)\right]G(\theta_l) \bw .$$
Using the definitions of $\Pi$ (cf. equation \eqref{useful}), we can write $\bw^\top \Pi \left( \Pi^{-1}  - H\right)G \bw$ into a sum of terms of three types (cf. equation \eqref{simpleid} below) 
$$ \text{\bf Type A}:\ \bw_1^\top G(\theta_l) \bw_2 ,\quad \text{\bf Type B}:\ \bw_3^\top J_1 H J_3 G(\theta_l) \bw_2 ,\quad \text{\bf Type C}:\ \bw_5^\top J_2 H J_4 G(\theta_l) \bw_6 , $$
where $\bw_k$, $1\le k \le 6$, are vectors that are independent of $G$ (and whose forms are irrelevant for our discussion below), and the matrices $J_\al$ are $(p+q+n)\times(p+q+n)$ block identity matrices defined as
\be\label{defJal}
J_{\al}:=\begin{pmatrix} \mathbf 1_{\al=1}I_p & 0 & 0 & 0 \\ 0 & \mathbf 1_{\al=2}I_q & 0 & 0 \\ 0 & 0 & \mathbf 1_{\al=3}I_n & 0 \\ 0 & 0 & 0 & \mathbf 1_{\al=4}I_n \end{pmatrix} ,\quad \al=1,2,3,4. 
\ee

We only consider type B terms, while type C terms can be handled in exactly the same way. We need to calculate terms of the form 
\be\label{eq_typeb} 
n^{1/2} \mathbb E \sum_{1\le \fa \le p+q+2n}\sum_{1\le i \le p, p+q+1\le \mu \le p+q+n}\bw_3(i)\bw_4(\fa) X_{i\mu} G_{\mu\fa}\Upsilon^{k-1} .
\ee
Assume for now that the entries of $X$ are Gaussian. Then, applying Gaussian integration by parts to $X_{i\mu}$, we obtain that 
\begin{align*} 
\eqref{eq_typeb}&=n^{1/2} \E \sum_{\fa }\sum_{i,\mu}\bw_3(i)\bw_4(\fa) \frac{\partial G_{\mu\fa}}{\partial X_{i\mu}}\Upsilon^{k-1} + (k-1) n^{1/2} \E \sum_{\fa }\sum_{i,\mu}\bw_3(i)\bw_4(\fa) G_{\mu\fa}\Upsilon^{k-2} \frac{\partial \Upsilon}{\partial X_{i\mu}}\\
&=:\text{I} + \text{II}.
\end{align*}
By the definition of $G$, its derivative with respect to $X_{i\mu}$ can be evaluated as
\be\nonumber
\frac{\partial G_{\fa\fb}}{\partial X_{i\mu}} = - G_{\fa i}G_{\mu \fb} - G_{\fa \mu}G_{i \fb}.
\ee
We can calculate the terms I and II using this identity. Then, the resulting expressions can be estimated using the anisotropic local law,  Theorem \ref{thm_local}, on $G$ and the anisotropic local laws on $GJ_\al G$, $\al=1,2,3,4$, which will be provided by Theorem \ref{lemmaGHF} below. Through our calculations, we find that the term I will cancel certain type A terms up to an $\oo(1)$ error, while the term II will contribute to the first term on the right-hand side of \eqref{idea_induction}. 

In general, when the entries of $X$ are not Gaussian, we can replace Gaussian integration by parts by a cumulant expansion formula in Lemma \ref{lemma_add_cumu}, with which we get an expansion of \eqref{eq_typeb} with higher order derivatives of $G_{\mu\fa}\Upsilon^{k-1}$. Then, we need to estimate them using anisotropic local laws on $G$ and $GJ_\al G$. However, due to the intricate form of $G$ as an inverse of a $4\times 4$ block matrix, the estimation of first order derivative terms is already quite complicated. The estimation of higher order derivative terms will be even more tedious. In particular, to get the fourth cumulant terms in \eqref{Cijij}, we need to study terms coming from the third order derivative of $G_{\mu\fa}\Upsilon^{k-1}$, which leads to a much lengthier calculation than that in the Gaussian case. To have a more tractable proof, we will adopt a strategy in \cite{KY,KY_AOP}: we first consider an \emph{almost Gaussian case} where most of the entries of $X$ and $Y$ are Gaussian, and then show that the general case is sufficiently close to the almost Gaussian case in the sense of the limiting CLT of $\Upsilon_0$ in \eqref{eq_upsilon0}. The merit of this strategy is that we can divide the proof into several parts that are relatively easier to handle, as we will explain now.

First, given the matrix $\bW$ appearing in $\Upsilon_0$, we will construct almost Gaussian matrices $X^g$ and $Y^g$ by changing most entries of $X^g$ and $Y^g$ to i.i.d.\;Gaussian random variables, while keeping the rest entries unchanged. The locations of Gaussian entries depend on the indices of ``small" entries in $\bW$ (see Proposition \ref{main_prop1} for more details). Then, we can define $H^g$, $G^g$ and $\Upsilon^g_0$ by replacing $X$ and $Y$ with $X^g$ and $Y^g_0$ in definitions \eqref{linearize_block}, \eqref{def_resolvent} and \eqref{eq_upsilon0}. Under this construction, we can show that $\Upsilon_0$ has the same asymptotic distribution as $\Upsilon_0^g$ through a resolvent comparison argument developed in \cite[Section 7]{KY}. Since this is a relatively standard argument in the random matrix theory literature, we will not discuss it here and refer the reader to Section \ref{secpfmain1} for more details.

Now, to conclude the proof, it remains to prove the CLT of $\Upsilon^g_0$. We first decompose each of $X^g$ and $Y^g$ into several different blocks---a large block consisting of Gaussian entries only and several small blocks that also contain non-Gaussian entries. Using the Schur complement formula and concentration estimates for large random vectors, after some calculations, we can rewrite $\Upsilon^g_0$ into two parts, where one part is of the form \eqref{eq_upsilon0} with a resolvent consisting of the large Gaussian blocks in $X^g$ and $Y^g$, and the other part is a quadratic form of the small blocks in $X^g$ and $Y^g$ (see equation \eqref{Mthetad2} below). We have discussed the proof for the former part using Gaussian integration by parts and local laws. On the other hand, the latter part can be handled directly using the classical CLT. This completes the proof for the almost Gaussian case in principle, but the calculations of the limiting covariance functions of the two parts (cf. Sections \ref{sec_var} and \ref{sec_var2}) are rather tedious. However, these calculations are straightforward algebraic calculations, and the reader can use a computer algebra system to check them. 

Our main result for the almost Gaussian case is summarized in Proposition \ref{main_prop1}, and its proof in Sections \ref{sec almostGauss}--\ref{appd GJG} constitutes the main theoretical contribution of this paper. More precisely, Section \ref{sec almostGauss}  constructs the almost Gaussian setting and calculates the limiting covariance function; Section \ref{sec Gauss} proves the CLT of \eqref{eq_upsilon0} in the Gaussian case; Section \ref{appd GJG} proves an sharp anisotropic local law on $GJ_\al G$.

Finally, Theorem \ref{main_thm2} follows from Theorem \ref{main_thm1} combined with a comparison argument. More precisely, suppose we have two ensembles of random matrices $(X, Y)$ and $(\wt X,\wt Y)$, where $X$ and $Y$ satisfy the moment assumption \eqref{eq_8moment} and \smash{$\wt X$ and $\wt Y$} satisfy \eqref{eq_highmoment}. Then, using the resolvent comparison method developed in \cite{Anisotropic}, we can show that the asymptotic distributions of $ \Upsilon_0(X,Y)$ and \smash{$ \Upsilon_0(\wt X,\wt Y)$} are the same as long as the \emph{first four moments} of the $X$ entries and $Y$ entries match those of the $\wt X$ entries and $\wt Y$ entries. In the proof of Theorem \ref{main_thm1}, we have shown the CLT of $ \Upsilon_0(\wt X,\wt Y)$. Together with the comparison result, it implies that $ \Upsilon_0(X,Y)$ satisfies the same CLT, and thus concludes Theorem \ref{main_thm2}. Both the construction of $(\wt X,\wt Y)$ according to the moment matching conditions and the resolvent comparison method have been well-understood in the random matrix theory literature. We refer the reader to Section \ref{pf thm2} for more details.

\section{Linearization method and resolvents}\label{sec pfstart}

In this section, we reduce the study of the limiting distribution of the outliers to proving the CLT for a matrix of the form \eqref{eq_upsilon0}. We first recall some (almost) sharp convergence estimates on the sample CCCs that have been proved in \cite{PartI,PartIII}. They will serve as important a priori estimates for our proof. 

\subsection{Convergence of sample CCCs}

To simplify notations, it is helpful to use the following notion of stochastic domination introduced in \cite{Average_fluc}. 
It greatly simplifies the presentation by systematizing statements of the form ``$\xi$ is bounded by $\zeta$ with high probability up to a small power of $n$".

\begin{definition}[Stochastic domination and high probability event]\label{stoch_domination}
	(i) Let
	\[\xi=\left(\xi^{(n)}(u):n\in\bbN, u\in U^{(n)}\right),\hskip 10pt \zeta=\left(\zeta^{(n)}(u):n\in\bbN, u\in U^{(n)}\right)\]
	be two families of nonnegative random variables, where $U^{(n)}$ is a possibly $n$-dependent parameter set. We say $\xi$ is stochastically dominated by $\zeta$, uniformly in $u$, if for any small constant $\epsilon>0$ and large constant $D>0$, we have that
	\[\sup_{u\in U^{(n)}}\bbP\left[\xi^{(n)}(u)>n^\epsilon\zeta^{(n)}(u)\right]\le n^{-D}\]
	for large enough $n\ge n_0(\epsilon, D)$, and we will use the notation $\xi\prec\zeta$ to denote it. 
	If a family of complex random variables $\xi$ satisfy $|\xi|\prec\zeta$, then we will also write $\xi \prec \zeta$ or $\xi=\OO_\prec(\zeta)$.
	
	\vspace{5pt}	
	\noindent (ii) We extend $\OO_\prec(\cdot)$ to matrices in the operator norm sense as follows. Let $A$ be a family of random matrices and $\zeta$ be a family of nonnegative random variables. Then $A=\OO_\prec(\zeta)$ means that $\|A\|\prec \zeta$.
	
	\vspace{5pt}	
	\noindent (iii) As a convention, for two \emph{deterministic} nonnegative quantities $\xi$ and $\zeta$, we write $\xi\prec\zeta$ if and only if $\xi\le n^\tau \zeta$ for any constant $\tau>0$. 
	
	\vspace{5pt}	
	\noindent (iv) We say an event $\Xi$ holds with high probability (w.h.p.) if for any constant $D>0$, $\mathbb P(\Xi)\ge 1- n^{-D}$ for large enough $n$. Moreover, we say $\Xi$ holds with high probability on an event $\Omega$ if for any constant $D>0$, $\mathbb P(\Omega\setminus \Xi)\le n^{-D}$ for large enough $n$.
\end{definition}

The following lemma collects some basic properties of stochastic domination $\prec$, which will be used tacitly in the proof.

\begin{lemma}[Lemma 3.2 in \cite{isotropic}]\label{lem_stodomin}
	Let $\xi$ and $\zeta$ be two families of nonnegative random variables, $U^{(n)}$ and $V^{(n)}$ be two parameter sets, and $C>0$ be a large constant.
	\begin{enumerate}
		\item[(i)] Suppose that $\xi (u,v)\prec \zeta(u,v)$ uniformly in $u\in U^{(n)}$ and $v\in V^{(n)}$. If $|V^{(n)}|\le n^C$, then $\sum_{v\in V^{(n)}} \xi(u,v) \prec \sum_{v\in V^{(n)}} \zeta(u,v)$ uniformly in $u \in U^{(n)}$.
		
		\item[(ii)] If $\xi_1 (u)\prec \zeta_1(u)$ and $\xi_2 (u)\prec \zeta_2(u)$ uniformly in $u\in U^{(n)}$, then $\xi_1(u)\xi_2(u) \prec \zeta_1(u)\zeta_2(u)$ uniformly in $u\in U^{(n)}$.
		
		\item[(iii)] Suppose that $\Psi(u)\ge n^{-C}$ is deterministic and $\xi(u)$ satisfies $\mathbb E|\xi(u)|^2 \le n^C$ for all $u\in U^{(n)}$. Then if $\xi(u)\prec \Psi(u)$ uniformly in $u\in U^{(n)}$, we have that $\mathbb E\xi(u) \prec \Psi(u)$ uniformly in $u\in U^{(n)}$.
	\end{enumerate}
\end{lemma}

The following large deviation bounds on the outliers of $\cal C_{\cal X\cal Y}$ were proved in \cite{PartI}.

\begin{lemma}[Theorem 2.9 of \cite{PartI}]\label{main_thm}
	Suppose Assumption \ref{main_assm} holds. If $t_i \ge t_c + n^{-1/3}$, then we have that
	\be\label{boundout}
	|\wt\lambda_i -\theta_i | \prec  n^{-1/2}|t_i - t_c|^{1/2}.
	\ee
	On the other hand, for any $i=\OO(1)$ with $t_i < t_c + n^{-1/3}$, we have that
	\be\label{boundedge}
	|\wt\lambda_i - \lambda_+| \prec  n^{-2/3}.
	\ee
\end{lemma}

The quantiles of the density \eqref{LSD} correspond to the classical locations of the eigenvalues of $\cal C_{YX}$. 
\begin{definition} 
	The classical location $\gamma_j$ of the $j$-th eigenvalue of $\cal C_{YX}$ is defined as
	\begin{equation}\label{gammaj}
		\gamma_j:=\sup_{x}\left\{\int_{x}^{+\infty} f(t)\dd t > \frac{j-1}{q}\right\},
	\end{equation}
	where $f$ is defined in \eqref{LSD}. Note that we have $\gamma_1 = \lambda_+$ and $\lambda_+ - \gamma_j \sim (j/n)^{2/3}$ for $j>1$.
\end{definition}

In \cite{PartIII}, we have proved the following eigenvalue rigidity estimate for $\cal C_{YX}$.
 
\begin{lemma}[Theorem 2.5 of \cite{PartIII}]\label{lem null}
	Suppose Assumption \ref{main_assm} holds. The eigenvalues of the null SCC  matrix $\cal C_{YX}$ satisfy the following eigenvalue rigidity estimate: 
	\be\label{rigidity}
	|\lambda_i - \gamma_i | \prec i^{-1/3} n^{-2/3},\quad 1 \le i \le (1-\delta)q,
	\ee
	where $\delta>0$ is any small constant.
\end{lemma}

\subsection{Local laws}\label{sec_maintools}

In this section, we state some local laws on the resolvent that have been proved in \cite{PartI,PartIII}. These local laws will be important tools for our proof. We first introduce some new notations.

\begin{definition}[Index sets]\label{def_index}
	For simplicity of notations, we define the index sets
	\begin{align*}
		\cal I_1:=\{1,\cdots,p\} , \quad & \cal I_2:=\{p+1,\cdots,p+q\}, \\
		\cal I_3:=\{ p+q+1,\cdots,p+q+n\}, \quad & \cal I_4:=\{ p+q+n+1,\cdots,p+q+2n\}.
	\end{align*} 
	We will consistently use latin letters $i,j\in\sI_{1}\cup \sI_2$ and greek letters $\mu,\nu\in\sI_{3}\cup \sI_{4}$. Moreover, we will use the notations $\fa,\fb\in \cal I:=\cup_{i=1}^4 \cal I_i$. 
\end{definition}

Denote the averaged partial traces of the resolvent by
\be\label{defmal} 
m_\al(z) : = \frac{1}{n}\sum_{\fa \in \cal I_\al}  G_{\fa\fa}(z) ,\quad \al=1,2,3,4. \ee
In \cite{PartIII}, we have shown that they converge to the deterministic limits given by 
\begin{align}
	&m_{1c}(z) 
	= \frac{ - z +c_1+c_2+\sqrt{(z-\lambda_-)(z-\lambda_+)} }{2(1-c_1)z(1-z)} - \frac{c_1}{(1-c_1)z}, \label{m1c}\\
	&m_{2c}(z) = \frac{ -z +c_1 + c_2+ \sqrt{(z-\lambda_-)(z-\lambda_+)}}{2(1-c_2)z(1-z)} -\frac{c_2}{(1-c_2)z}, \label{m2c}\\
	&m_{3c}(z) 
	= \frac{1}{2}\left[ (1-2c_1) z + c_1 - c_2 + \sqrt{(z-\lambda_-)(z-\lambda_+)}\right] , \label{m3c}\\
	&m_{4c}(z)= \frac{1}{2}\left[ (1-2c_2) z +  c_2 - c_1 + \sqrt{(z-\lambda_-)(z-\lambda_+)}\right] ,\label{m4c}
\end{align}
where $\lambda_\pm$ are defined in \eqref{lambdapm}. In \cite{PartIII}, we also verified the following equations for $m_{\al c}$:
\begin{align}
	& m_{1c}= - \frac{c_1}{m_{3c}} , \quad m_{2c} = -\frac{ c_2}{m_{4c}}, \quad  m_{3c}(z) - m_{4c} (z)= (1-z)(c_1-c_2) ,\label{selfm12}\\
	& m_{3c}(z) 
	=\frac{ 1-(z-1)m_{2c}(z)}{z^{-1} - [m_{1c}(z)+m_{2c}(z)] + (z-1)m_{1c}(z)m_{2c}(z)}, \label{selfm3}  \\
	& m_{4c}(z) = \frac{ 1-(z-1)m_{1c}(z)}{z^{-1} - [m_{1c}(z)+m_{2c}(z)] + (z-1)m_{1c}(z)m_{2c}(z)} \label{selfm32} .
\end{align}
One can also check them through direct calculations with \eqref{m1c}--\eqref{m4c}. We also define the function 
\be
\begin{split}\label{hz}
	h(z):&=  \frac{z^{-1/2}m_{3c}(z)}{1+(1-z)m_{2c}(z)} =  \frac{z^{-1/2}m_{4c}(z)}{1+(1-z)m_{1c}(z)}\\
	&= \frac{z^{1/2}}{2} \left[ - z + (2-c_1-c_2) + \sqrt{(z-\lambda_-)(z-\lambda_+)}\right].
\end{split}
\ee
With the above definitions, we define the matrix limit of $G(z)$ as
\be \label{defn_pi}
\Pi(z) := \begin{bmatrix} \begin{pmatrix} c_1^{-1}m_{1c}(z)I_p & 0\\ 0 & c_2^{-1}m_{2c}(z)I_q\end{pmatrix} & 0 \\ 0  & \begin{pmatrix}  m_{3c}(z)I_n  & h(z)I_n\\  h(z)I_n &  m_{4c}(z)  I_n\end{pmatrix}\end{bmatrix} .\ee
Using \eqref{selfm12}--\eqref{hz}, one can check that 
\be \label{useful}
\Pi= \begin{bmatrix} \begin{pmatrix} -m_{3c}I_p & 0\\ 0 & -m_{4c}I_q\end{pmatrix} & 0 \\ 0  &  \begin{pmatrix} z  I_n & z^{1/2}I_n\\ z^{1/2} I_n &  z I_n\end{pmatrix}^{-1} -\begin{pmatrix} m_{1c}I_n& 0 \\ 0 & m_{2c}I_n\end{pmatrix}\end{bmatrix} ^{-1}.\ee

We define two different spectral domains of $z$ for the local laws. 
\begin{definition}
	Given a constant $\e >0$, we define a spectral domain around the bulk spectrum $[\lambda_-,\lambda_+]$ as
	\begin{equation}
		S(\e):= \left\{z=E+ \ii \eta: \e \leq E \leq 1-\e, n^{-1+\e} \leq \eta \leq \e^{-1} \right\}, \label{SSET1}
	\end{equation}
	and a spectral domain outside the bulk spectrum as 
	\begin{equation}\label{eq_paraout}
		S_{out}(\epsilon):=\left\{z= E+ \ii\eta: \lambda_+ + n^{-2/3+\e}\le E \le 1-\e,  0\le \eta \le \e^{-1}\right\}.
	\end{equation}
\end{definition} 

The following theorem gives the \emph{anisotropic local law} of $G(z)$ on the above two spectral domains. 
\begin{theorem} [Anisotropic local law]\label{thm_local}  
	Suppose Assumption \ref{main_assm} holds. For any fixed $\e>0$ and deterministic unit vectors $\mathbf u, \mathbf v \in \mathbb C^{\mathcal I}$, the following anisotropic local laws hold. 
	
	\begin{enumerate}
		\item  {\bf (Theorem 2.13 of \cite{PartIII})}. For any $z= E+\ii \eta\in S(\epsilon)$, we have that
		\begin{equation}\label{aniso_law}
			\left| \langle \mathbf u, G(z) \mathbf v\rangle - \langle \mathbf u, \Pi (z)\mathbf v\rangle \right| \prec  \sqrt {\frac{\Im \, m_{3c}(z)}{{n\eta }} } + \frac{1}{n\eta},
		\end{equation}
		where the inner product is defined as $\langle \mathbf v, \mathbf w\rangle:= \bv^* \bw$ with $\bv^*$ denoting the conjugate transpose. 
		
		\item{\bf (Theorem 3.9 of \cite{PartI})}. For any $z= E+\ii \eta\in S_{out}(\epsilon)$, we have that
		\begin{equation}\label{aniso_outstrong}
			\left| \langle \mathbf u, G(z) \mathbf v\rangle - \langle \mathbf u, \Pi (z)\mathbf v\rangle \right|  \prec   \frac{1}{ n^{1/2}(\vert E -\lambda_+\vert+\eta)^{1/4}}.
		\end{equation}
	\end{enumerate}
	The above estimates \eqref{aniso_law} and \eqref{aniso_outstrong} hold uniformly in the spectral parameter $z$. Moreover, for these estimates to hold, it is not necessary to assume that the entries of $X$, $Y$ and $Z$ are identically distributed---only independence and moment conditions are needed. 
\end{theorem}

The averaged partial traces in \eqref{defmal} satisfy stronger \emph{averaged local laws}.  
\begin{theorem}[Averaged local law, Theorem 2.14 of \cite{PartIII}]\label{thm_largerigidity}
	Suppose Assumption \ref{main_assm} holds. For any fixed $\epsilon>0$, we have that	
	\begin{equation}
	\max_{\al=1,2,3,4}	\vert m_\al(z)-m_{\al c}(z) \vert \prec (n \eta)^{-1}, \label{aver_in}
	\end{equation}
	uniformly in $z \in S(\epsilon )$. Moreover, outside of the spectrum we have the stronger estimate 
	\begin{equation}\label{aver_out0}
	\max_{\al=1,2,3,4}	\vert m_\al(z)-m_{\al c}(z) \vert \prec \frac{1}{n(\vert E -\lambda_+\vert +\eta)} + \frac{1}{(n\eta)^2\sqrt{\vert E -\lambda_+\vert +\eta}},
	\end{equation}
	uniformly in $z\in S(\e) \cap S_{out}(\epsilon)$. 
\end{theorem}

\subsection{Reduction to the law of resolvent}\label{sec mainthm}

In this subsection, we relate the limiting law of $\bm\zeta$ in Theorem \ref{main_thm1} to that of a matrix taking the form \eqref{eq_upsilon0}. Without loss of generality, we assume a slightly stronger condition than \eqref{assm evalue} so that $A$ and $B$ are both of rank $r$:
\be\label{assm evalue2}
0< a_r \le \cdots \le a_2\le a_1 \le C, \quad 0< b_r\le \cdots \le b_2\le b_1\le C.
\ee
This can be achieved by adding a small $0<\e_n<e^{-n}$ to each zero $a_i$ or $b_i$. Since the proof does not depend on the lower bounds of $a_r$ and $b_r$, we can easily extend it to the case with zero $a_i$'s or $b_i$'s by taking $\e_n\to 0$.

Recall that if $\lambda\in (0,1)$ is not in the spectrum of $\cal C_{XY}$, then it is an eigenvalue of $\cal C_{\cal X\cal Y}$ if and only if \eqref{detereq temp2} holds. Throughout the following discussion, we always assume that $\lambda\in S_{out}(\e)$ and $\lambda\ge \lambda_+ +\e$ for a small constant $\e>0$. We write \eqref{detereq temp2} as
\be\label{masterx2}
\begin{split}
	0&=\det \left[ \begin{pmatrix} 0 & \cal D^{-1}\\ \cal D^{-1}  & 0\end{pmatrix}+ \Pi_{4r}(\lambda) + \cal E_{4r} \right]  = \det \left[ \begin{pmatrix} \Pi_{2r}^{(1)} & \cal D^{-1}  \\ \cal D^{-1}  & \Pi_{2r}^{(2)}  \end{pmatrix} + \cal E_{4r}  \right],
\end{split}
\ee
where $\Pi^{(1)}_{2r}$ and $\Pi^{(2)}_{2r}$ are $2r\times 2r$ deterministic matrices defined as 
$$\Pi^{(1)}_{2r}(\lambda):= \begin{pmatrix} c_1^{-1}m_{1c}(\lambda) I_r & 0 \\ 0 & c_2^{-1}m_{2c}(\lambda) I_r\end{pmatrix}, \quad \Pi^{(2)}_{2r}(\lambda):= \begin{pmatrix}  m_{3c}(\lambda) I_r & h(\lambda)  \bV_a^{\top} \bV_b \\ h(\lambda) \bV_b^{\top} \bV_a &  m_{4c}(\lambda) I_r\end{pmatrix},$$
$\Pi_{4r}$ is a $4r\times 4r$ deterministic matrix defined as
\be\label{defpi} \Pi_{4r}(\lambda):= \begin{pmatrix} \Pi^{(1)}_{2r}(\lambda) & 0\\0 & \Pi^{(2)}_{2r}(\lambda) \end{pmatrix},   \ee
and  $\cal E_{4r}$ is a $4r\times 4r$ random matrix defined as
\be\label{e4r} 
\begin{split}
	\cal E_{4r}&\equiv \begin{pmatrix} \cal E_{2r}^{(1)} & \cal E_{2r}^{(3)} \\ \cal E_{2r}^{(4)} & \cal E_{2r}^{(2)} \end{pmatrix} := \begin{pmatrix} \bU^{\top} & 0 \\ 0 & \bV^{\top}\end{pmatrix} \left(G -\Pi \right) \begin{pmatrix} \bU & 0 \\ 0 & \bV\end{pmatrix} + \begin{pmatrix} 0 &0 \\ 0  & \bV^{\top}\Pi^{(2)} \bV  - \Pi_{2r}^{(2)}\end{pmatrix}.
\end{split}
\ee
Here, $\cal E_{2r}^{(1)}$, $\cal E_{2r}^{(2)}$, $\cal E_{2r}^{(3)}$ and $\cal E_{2r}^{(4)}$ are the upper-left, lower-right, upper-right, and lower-left $2r\times 2r$ blocks of $\cal E_{4r}$, and 
$$\Pi^{(2)}(\lambda):=  \begin{pmatrix}  m_{3c}(\lambda) I_n & h(\lambda) I_n \\ h(\lambda) I_n &  m_{4c}(\lambda) I_n\end{pmatrix} $$
is the lower-right $2n\times 2n$ block of $\Pi$. Note $\Pi_{2r}^{(2)}$ is defined such that $\Pi_{2r}^{(2)}= \E (\bV^{\top}\Pi^{(2)} \bV)$. 

Using the large deviation bounds in Lemma \ref{largedeviation} below, we can obtain the following approximate isotropic conditions for $Z$:
\begin{equation}
	\| ZZ^{\top} - I_{r}\| \prec n^{-1/2}, \quad \text{and} \quad \|Z\mathbf v \|_{2} \prec n^{-1/2}\|\mathbf v\|_2, \label{eq_iso}
\end{equation}
for any deterministic vector $\bv \in \mathbb C^n$. Using Theorem \ref{thm_local} and equation \eqref{eq_iso}, we can bound $\cal E_{4r}$ as 
\be\label{e4r2}
\left\|\cal E_{4r}\right\| \prec n^{-1/2} . 
\ee 
Now, using the Schur complement formula, we find that \eqref{masterx2} is equivalent to 
\be\nonumber 
\begin{split}
	\det\left[\Pi_{2r}^{(2)} + \cal E_{2r}^{(2)}  -  \left(\cal D^{-1}+\cal E_{2r}^{(4)}\right)\left(\Pi_{2r}^{(1)} + \cal E_{2r}^{(1)} \right)^{-1} \left(\cal D^{-1}+\cal E_{2r}^{(3)}\right) \right] = 0 .
\end{split}
\ee
Using \eqref{e4r2} and the first two equations in \eqref{selfm12}, we can reduce this equation to 
\be\label{masterx2.5}
\begin{split}
	& \det \left[\begin{pmatrix} m_{3c}(\lambda) (I_r+\Sigma_a^{2})  & h(\lambda) \Sigma_a\bV_a^{\top} \bV_b\Sigma_b \\ h(\lambda) \Sigma_b \cal \bV_b^{\top} \bV_a \Sigma_a & m_{4c}(\lambda) (I_r+\Sigma_b^2)  \end{pmatrix} +  \cal E_{2r}+\OO_\prec (n^{-1} ) \right]=0,
\end{split}
\ee
where $\cal E_{2r}$ is a $2r\times 2r$ random matrix defined as
\begin{align*}
	\cal E_{2r} & =  \cal D\cal E_{2r}^{(2)}\cal D +   (\Pi_{2r}^{(1)})^{-1} \cal E_{2r}^{(1)} (\Pi_{2r}^{(1)})^{-1}   - (\Pi_{2r}^{(1)})^{-1} \cal E_{2r}^{(3)} \cal D - \cal D \cal E_{2r}^{(4)}(\Pi_{2r}^{(1)})^{-1} \\
	&= \begin{pmatrix}  m_{3c}\cal E^{(1)}_r  & h  \cal E^{(3)}_r  \\ h  \cal E^{(4)}_r &  m_{4c} \cal E^{(2)}_r  \end{pmatrix}  ,
\end{align*}
with $\cal E^{(\al)}_r$, $\al=1,2,3,4$, being four $r\times r$ random matrices defined as
\begin{align*}
	\cal E^{(1)}_r &=  m_{3c}^{-1}\Sigma_a\bV_a^{\top} Z\left( \cal G_{(33)} - m_{3c}\right)Z \bV_a\Sigma_a +  \Sigma_a \bV_a^{\top} \left(Z Z^\top -I_r\right)\bV_a \Sigma_a \\
	& + m_{3c}  \bU_a^{\top} (\cal G_{(11)} - c_1^{-1}m_{1c} )\bU_a + \left[ \bU_a^{\top} \cal G_{(13)} Z^\top \bV_a \Sigma_a+  \Sigma_a \bV_a^{\top} Z\cal G_{(31)} \bU_a \right] , \\
	\cal E^{(2)}_r & = m_{4c}^{-1}\Sigma_b  \bV_b^{\top} Z\left(\cal G_{(44)} - m_{4c}\right) Z^\top\bV_b\Sigma_b +  \Sigma_b \bV_b^{\top}\left( ZZ^\top-I_r\right) \bV_b \Sigma_b  \\
	&+ m_{4c}   \bU_b^{\top} (\cal G_{(22)} - c_2^{-1}m_{2c} )\bU_b +  \left[  \bU_b^{\top} \cal G_{(24)}Z^\top \bV_b\Sigma_b+ \Sigma_b \bV_b^{\top} Z\cal G_{(42)} \bU_b \right] \\
	\cal E^{(3)}_r &=  (\cal E^{(4)}_r)^{\top} = h^{-1}\Sigma_a\bV_a^{\top}Z \left(\cal G_{(34)} - h\right) Z^\top \bV_b\Sigma_b +  \Sigma_a\bV_a^{\top}\left( ZZ^\top  - I_r\right)\bV_b\Sigma_b  \\
	&+ \frac{m_{3c}m_{4c}}{h}  \bU_a^{\top} \cal G_{(12)} \bU_b +\frac{m_{3c}}{h} \bU_a^{\top}\cal G_{(14)}Z^\top \bV_b \Sigma_b + \frac{m_{4c}}{h}\Sigma_a \bV_a^{\top}Z \cal G_{(32)} \bU_b.
\end{align*}
In the above expressions, we abbreviated the $\cal I_\al \times \cal I_\beta$ block of $ G$ by $ \cal G_{(\al\beta)}$ for $\al,\beta=1,2,3,4$. 
Applying the Schur complement formula once again, we obtain that \eqref{masterx2.5} is equivalent to 
\begin{align*}
	\det & \left[ f_{c}(\lambda) \left(I_r+\Sigma_a^{2}\right)  + f_{c}(\lambda) \cal E_r^{(1)} \right.\\
	& \left. - \left(  \Sigma_a \bV_a^{\top} \bV_b \Sigma_b+ \cal E_r^{(3)}\right) \frac{1}{  I_r+\Sigma_b^{2} + \cal E_r^{(2)}} \left( \Sigma_b\bV_b^{\top} \bV_a\Sigma_a+ \cal E_r^{(4)}\right) +\OO_\prec (n^{-1})\right]=0,
\end{align*}
where the function $f_c$ is defined by
\be\label{fcz} f_c(z):=\frac{m_{3c}(z)m_{4c}(z)}{h^2(z)}=\frac{z- (c_1+c_2-2c_1c_2) + \sqrt{(z-\lambda_-)(z-\lambda_+)}}{2(1-c_1)(1-c_2)}.\ee 
Using \eqref{e4r2}, we can check that $\|\cal E_r^{(\al)}(\lambda)\|\prec n^{-1/2}$, $\al=1,2,3,4$, 
with which we can further reduce the above equation to
\be\label{masterx3}
\begin{split}
	&  \det \left[ f_c(\lambda) I_r -  \wh\Sigma_a\bV_a^{\top} \bV_b\wh\Sigma_b^2 \bV_b^{\top} \bV_a\wh\Sigma_a+ \cal E_r (\lambda)+\OO_\prec (n^{-1}) \right] =0,
\end{split}
\ee
where we have abbreviated that
\be\label{Mrab0}\wh\Sigma_a:=\frac{\Sigma_a}{(I_r+\Sigma_a^2)^{1/2}},\quad \wh\Sigma_b:=\frac{\Sigma_b}{(I_r+\Sigma_b^2)^{1/2}},\ee
and $\cal E_{r}$ is a $r\times r$ random matrix defined by
\be \label{master_er}
\begin{split}
	\cal E_r: = & f_c \frac{1}{ (I_{r} + \Sigma_a^2)^{1/2}}  \cal E^{(1)}_r \frac{1}{ (I_{r} + \Sigma_a^2)^{1/2}}   \\
	&+ \wh\Sigma_a \bV_a^{\top} \bV_b \wh\Sigma_b \frac{1}{(I_r+\Sigma_b^2)^{1/2}} \cal E_r^{(2)}  \frac{1}{(I_r+\Sigma_b^2)^{1/2}}\wh\Sigma_b \bV_b^{\top} \bV_a \wh\Sigma_a \\
	& - \frac{ 1}{ (I_{r} + \Sigma_a^2)^{1/2}}  \cal E^{(3)}_r  \frac{1}{(I_r+\Sigma_b^2)^{1/2}}\wh\Sigma_b \bV_b^{\top} \bV_a\wh\Sigma_a \\
	&  -  \wh\Sigma_a \bV_a^{\top} \bV_b \wh\Sigma_b \frac{1}{(I_r+\Sigma_b^2)^{1/2}} \cal E^{(4)}_r  \frac{1}{ (I_{r} + \Sigma_a^2)^{1/2}}   .
\end{split}
\ee
Finally, with the SVD \eqref{Mrab}, we can rewrite the equation \eqref{masterx3} as
\be\label{masterx4}
\begin{split}
	&  \det \left[ f_c(\lambda) I_r -   \diag (t_1, \cdots, t_r) + \cal O^{\top} \cal E_r (\lambda) \cal O+\OO_\prec (n^{-1}) \right] =0.
\end{split}
\ee

One can easily check that the following function is the inverse of $f_c$ in \eqref{fcz} when $ z\notin [\lambda_-,\lambda_+]$:
$$g_c(\xi) := \xi \left( 1- c_1 + c_1\xi^{-1}\right)\left( 1- c_2 + c_2\xi^{-1}\right).$$
Moreover, it is easy to check that $f_c(\lambda_+)=t_c$ (recall \eqref{tc}). Since $f_c(\lambda)$ is monotonically increasing when $\lambda>\lambda_+$, the function $f_c(\lambda) - t_i=0$ has a solution in $(\lambda_+,1)$ if and only if 
\be\label{t>tc} 
t_c=f_c(\lambda_+)< t_i . 
\ee
If \eqref{t>tc} holds, then $t_i$ gives rise to an outlier lying around $\theta_i = g_c(t_i)$, which explains \eqref{gammai}. With a direct calculation, we can verify the following deterministic estimates on $f_c$ and $g_c$. 

\begin{lemma}[Lemma 4.1 of \cite{PartI}] \label{lem_complexderivative}
	Fix a  large constant $C>0$. For any $z \in \mathbb D:=\{z\in \mathbb C: \lambda_+ < \re z < C\}$ and $\xi \in f_c( \mathbb D)$,
	the following estimates hold:
	\begin{align}
		|f_{c}(z) - f_{c}(\lambda_+)| \sim |z-\lambda_+|^{1/2}, \quad & |f_{c}'(z) | \sim |z-\lambda_+|^{-1/2}, \label{eq_mcomplex0}\\
		|g_{c}(\xi) - \lambda_+| \sim |\xi-t_c|^2, \quad  & |g_{c}'(\xi) | \sim |\xi-t_c| .\label{eq_gcomplex0}
	\end{align}
\end{lemma}

Now, with equation \eqref{masterx4}, we can prove the following proposition, which shows that the limiting law of $\bm\zeta$ in Theorem \ref{main_thm1} is determined by the limiting law of $n^{1/2}\cal O^{\top} \cal E_r (\theta_l) \cal O$. {Let $\al:\{1,\cdots, \gamma(l)\}\to \{1,\cdots, r\}$ be a labeling function so that $\wt\lambda_{\al(i)}$ is the $i$-th largest value in the set $\{\wt\lambda_i: i\in \gamma(l)\}$.}

\begin{proposition}[Reduction to the law of $G$]\label{redGthm}
	Under the assumptions of Theorem \ref{main_thm1}, 
	there exists a constant $\e>0$ depending on $\delta$ only such that for $1\le i \le |\gamma(l)|$,
	\be\label{redG}
	\left| \left(\wt\lambda_{\al(i)} - \theta_l \right)- \mu_i \left\{ a(t_l) \left[  \diag (t_1, \cdots, t_r) - t_l - \cal O^{\top} \cal E_r (\theta_l)\cal O\right]_{\llbracket\gamma(l)\rrbracket}  \right\} \right| \prec   n^{-1/2-\e} ,
	\ee
	where $\mu_i$ is the $i$-th eigenvalue of the $|\gamma(l)|\times |\gamma(l)|$ matrix 
	$$ a(t_l)\left[ \diag (t_1, \cdots, t_r) - t_l - \cal O^{\top} \cal E_r (\theta_l)\cal O\right]_{\llbracket\gamma(l)\rrbracket} $$
	in the sense of \eqref{ev minor}. 
\end{proposition}
\begin{proof}
	By Lemma \ref{main_thm} and the condition \eqref{tlc}, we have that for $i\in \gamma(l)$, $\wt\lambda_i \in S_{out}(\e)$ and $\wt\lambda_i\ge \lambda_+ + \e$ with high probability for a sufficiently small constant $\e>0$. Thus the above discussion starting at \eqref{masterx2} will finally lead to the equation \eqref{masterx4}. Armed with \eqref{boundout}, equation \eqref{masterx4} and the estimates in Lemma \ref{lem_complexderivative}, we can conclude the proof using the same argument as the one for \cite[Proposition 4.5]{KY_AOP}. We omit the details. In fact, one can easily see why \eqref{redG} holds by performing a Taylor expansion of {$f_c(\wt\lambda_{\al(i)})$} around $\theta_l$ in \eqref{masterx4}, and noticing that $1/{f_c'(\theta_l)} = g_c'(t_l)=a(t_l)$.
\end{proof}


By Proposition \ref{redGthm}, to prove Theorem \ref{main_thm1}, it suffices to study the CLT of $n^{1/2}\cal O^{\top} \cal E_r (\theta_l)\cal O$. With a straightforward algebraic calculation, we get that 
\begin{align}\label{defPPP+QQQ}
	& \cal E_r(\theta_l)   =\cal E_r^{(z)}(\theta_l) + \cal E_r^{(g)}(\theta_l),
\end{align}
where 
\be\label{defPPP}
\begin{split}
	\cal E_r^{(z)}(\theta_l)   :=&\ f_c(\theta_l)  \wh\Sigma_a \bV_a^{\top}  \left(ZZ^{\top} - I_r\right) \bV_a\wh\Sigma_a  \\
	& + \wh\Sigma_a  \bV_a^{\top} \bV_b \wh\Sigma_b^2\bV_b^{\top} \left(ZZ^{\top} - I_r\right)\bV_b \wh\Sigma_b^2 \bV_b^{\top} \bV_a \wh\Sigma_a  \\
	& -   \wh\Sigma_a  \bV_a^{\top}  \left(ZZ^{\top} - I_r\right) \bV_b \wh\Sigma_b^2 \bV_b^{\top} \bV_a \wh\Sigma_a \\
	& -  \wh\Sigma_a  \bV_a^{\top} \bV_b \wh\Sigma_b^2 \bV_b^{\top} \left(ZZ^{\top} - I_r\right)\bV_a\wh\Sigma_a  ,
\end{split}
\ee
and
\be\label{defQQQ}
\begin{split}
	\cal E_r^{(g)}(\theta_l)  :=&\,   f_c(\theta_l) m_{3c}(\theta_l) \mathfrak W^\top(\theta_l)   \begin{pmatrix} \bU_a^{\top} & 0 & 0 & 0 \\ 0 & \bU_b^{\top} & 0 & 0 \\ 0 & 0 & Z & 0 \\ 0 & 0 & 0 & Z \end{pmatrix} [G(\theta_l)  -\Pi(\theta_l)  ]  \\
	& \times \begin{pmatrix} \bU_a & 0 & 0 & 0 \\ 0 & \bU_b & 0 & 0 \\ 0 & 0 & Z^{\top}& 0 \\ 0 & 0 & 0 & Z^{\top}  \end{pmatrix}  \mathfrak W(\theta_l)  ,
\end{split}
\ee
with $\mathfrak W$ being a $4r\times r$ matrix defined by
$$ \mathfrak W(\theta_l)  :=\begin{bmatrix}   (I_{r} + \Sigma_a^2)^{-1/2}    \\   -  h(\theta_l) {m^{-1}_{3c}(\theta_l) }(1+\Sigma_b^2)^{-1/2}\wh\Sigma_b \bV_b^{\top} \bV_a \wh\Sigma_a   \\   m_{3c}^{-1}(\theta_l)\bV_a\wh\Sigma_a \\   - h(\theta_l){m^{-1}_{3c}(\theta_l) m^{-1}_{4c}(\theta_l)}\bV_b \wh\Sigma_b^2 \bV_b^{\top} \bV_a \wh\Sigma_a  \end{bmatrix}.$$
{Here, the superscripts $(z)$ and $(g)$ indicate that we will make use of the CLT of $ZZ^\top -I_r$ and $G-\Pi$, respectively, when dealing with these two terms \eqref{defPPP} and \eqref{defQQQ}. }

By classical CLT, we know that
\be\label{CLTe1}\sqrt{n} \left(ZZ^{\top} - I_r\right) \Rightarrow \mathbf{G},\ee
where $ \mathbf{G}$ is an $r\times r$ symmetric Gaussian matrix whose entries are independent up to symmetry and have mean zero and variances (recall \eqref{mux4})
$$\mathbb E \mathbf{G}_{ij}^2 = 1, \quad i\ne j, \quad \text{and} \quad \mathbb E \mathbf{G}_{ii}^2 = \kappa_z^{(4)} + 2.$$
With this result, we immediately derive the CLT for $n^{1/2}\cal O^{\top}\cal E_r^{(z)}\cal O $. Therefore, to conclude Theorem \ref{main_thm1}, it remains to prove the CLT for the matrix 
\be\label{wtMx}{\cal M}_0(\theta_l)  :=\sqrt{n}\begin{pmatrix} \bU_a^{\top} & 0 & 0 & 0 \\ 0 & \bU_b^{\top} & 0 & 0 \\ 0 & 0 & Z & 0 \\ 0 & 0 & 0 &Z \end{pmatrix} [G(\theta_l)  -\Pi(\theta_l)  ]\begin{pmatrix} \bU_a & 0 & 0 & 0 \\ 0 & \bU_b & 0 & 0 \\ 0 & 0 & Z^{\top}& 0 \\ 0 & 0 & 0 & Z^{\top}  \end{pmatrix} .\ee
As discussed in Section \ref{sec_overview}, we first prove the CLT for ${\cal M}_0(\theta_l) $ in an almost Gaussian case, where most of the $X$ and $Y$ entries are Gaussian. Then, in Section \ref{secpfmain1}, we show that the general case in the setting of Theorem \ref{main_thm1} is sufficiently close to the almost Gaussian case, thereby completing the proof of Theorem \ref{main_thm1}.

\section{The almost Gaussian case}\label{sec almostGauss}


In this section, we calculate the limiting distribution of $\cal M_0(\theta_l)$ in the almost Gaussian case. The extension to the general setting in Theorem \ref{main_thm1} will be postponed to Section \ref{secpfmain1}. We fix a small constant $\tau_0>0$ in this section, and use $n^{-\tau_0}$ as a cutoff scale in the entries of $\bU_a$ and $\bU_b$, below which the corresponding entries of $X$ and $Y$ are Gaussian. Our goal is to prove the following proposition.

\begin{proposition}\label{main_prop1}
Fix any $1\le l \le r$ and a sufficiently small constant $\tau_0>0$. Suppose Assumption \ref{main_assm} and \eqref{tlc} hold. Suppose $X$ and $Y$ satisfy that for $k\in \cal I_1$,
	\be\label{almost GuassianX}
	\max_{1\le i \le r} |\bu_i^a(k)| \le n^{-\tau_0}\ \Rightarrow \ X_{k\mu} \text{ is Gaussian,} \ \ \mu\in \cal I_3, 
	\ee
	and for $k\in \cal I_2$,
	\be\label{almost GuassianY}
	\max_{1\le i \le r} |\bu_i^b(k)| \le n^{-\tau_0}\ \Rightarrow \ Y_{k\mu} \text{ is Gaussian,} \ \ \mu\in \cal I_4. 
	\ee
	Then, for any bounded continuous function $f:\R^{|\gamma(l)|\times |\gamma(l)|}\to \R$, we have that 
	\be\label{OEO}
	\lim_{n}\left[\mathbb E f\left( \left( \sqrt{n} \cal O^{\top} \cal E_r (\theta_l)\cal O\right)_{\llbracket \gamma(l)\rrbracket} \right) - \E f(\Upsilon_l)\right] =0,
	\ee
	where $\Upsilon_l$ is the Gaussian random matrix defined in Theorem \ref{main_thm1}. 
\end{proposition}

For simplicity, in the proof below we often drop the spectral parameter $z=\theta_l$ from our notations.  
Using \eqref{eq_iso} and the SVD of $Z$, we can find an $r\times n$ partial orthogonal matrix $\wt Z$ such that
\be\label{V-F} \wt Z \wt Z^{\top} = I_r, \quad  \|\wt Z - Z \|_F\prec n^{-1/2} .\ee
From \eqref{eq_iso} and \eqref{V-F}, we also obtain the following estimate: 
\be\label{Vmax}\|\wt Z\|_{\max} \le \|Z^\top\|_{\max} + n^{-1/2 +\e/2}  \le n^{-1/2+\e},\ee
with high probability for any fixed $\e>0$. Now, using \eqref{V-F} and \eqref{aniso_outstrong}, we get that
\be\label{M-M0}\|{\cal M}(\theta_l) - {\cal M}_0(\theta_l)\|\prec n^{-1/2},\ee
where $\cal M$ is a $4r\times 4r$ random matrix defined by
\be\label{defM} {\cal M}(\theta_l):=\sqrt{n}\begin{pmatrix} \bU_a^{\top} & 0 & 0 & 0 \\ 0 & \bU_b^{\top} & 0 & 0 \\ 0 & 0 & \wt Z & 0 \\ 0 & 0 & 0 & \wt Z \end{pmatrix} [G(\theta_l)-\Pi(\theta_l)] \begin{pmatrix} \bU_a & 0 & 0 & 0 \\ 0 & \bU_b & 0 & 0 \\ 0 & 0 & \wt Z^{\top} & 0 \\ 0 & 0 & 0 & \wt Z^{\top}  \end{pmatrix} .\ee
Hence, to obtain the CLT of $\cal M_0(\theta_l)$, it suffices to study $ {\cal M}(\theta_l)$.  For this purpose, we first introduce the concept of \emph{minors} of $H$ and $G$.  

\begin{definition}[Minors] \label{defminor}
Let $\cal J$ and $\mathbb T\subset \cal J$ be some index sets. Given any $ \cal J \times \cal J$ matrix $\cal A$, we define the minor $\cal A^{(\mathbb T)}:=(\cal A_{ab}:a,b \in \mathcal J\setminus \mathbb T)$ as the $ (\cal J\setminus \mathbb T)\times (\cal J\setminus \mathbb T)$ matrix obtained by removing all rows and columns indexed by $\mathbb T$. Note that we keep the names of indices when defining $\cal A^{(\mathbb T)}$, i.e. $(\cal A^{(\mathbb{T})})_{ab}= \cal A_{ab}$ for $a,b \notin \mathbb{{T}}$. Correspondingly, we define the resolvent minor as \smash{$	G^{(\mathbb T)}(z):=[  H^{(\mathbb T)}(z) ]^{-1}.$}
For convenience, we will adopt the convention that $\cal A^{(T)}_{\fa\fb} = 0$ when $\fa \in \mathbb T$ or $\fb \in \mathbb T$. We will abbreviate that $(\{\fa\})\equiv (\fa)$ and $(\{\fa, \fb\})\equiv (\fa\fb)$. 
\end{definition}

The following large deviation bounds for linear and quadratic forms of independent random variables were proved in proved in \cite{Delocal}.  

\begin{lemma}[Theorem B.1 of \cite{Delocal}]\label{largedeviation}
	Let $(x_i)$, $(y_j)$ be independent families of centered independent random variables, and $(\cal A_i)$, $(\cal B_{ij})$ be families of deterministic complex numbers. Suppose the entries $x_i$, $y_j$ have variances at most $n^{-1}$ and satisfy  (\ref{eq_highmoment}). Then, the following large deviation bounds hold:
	\be\nonumber
	\begin{split}
	\Big\vert \sum_i \cal A_i x_i \Big\vert \prec   \frac{1}{\sqrt{n}} &\Big(\sum_i |\cal A_i|^2 \Big)^{1/2} ,  \quad
	\Big\vert \sum_{i,j} x_i \cal B_{ij} y_j \Big\vert \prec  \frac{1}{n}\Big(\sum_{i, j} |\cal B_{ij}|^2\Big)^{{1}/{2}} , \\
	& \Big\vert \sum_{i\ne j}  x_i \cal B_{ij} x_j \Big\vert  \prec  \frac{1}{n}\Big(\sum_{i\ne j} |\cal B_{ij}|^2\Big)^{{1}/{2}} .
	\end{split}
	\ee
\end{lemma} 

For convenience, we introduce the following shorthand for the equivalence relation between two random vectors of fixed size in the sense of asymptotic distributions.
\begin{definition}\label{def_simD}
	Given two sequences of random vectors $\cal A_n$ and $\cal B_n$ in $\R^k$, where $k\in \N$ is a fixed integer,  we write $\cal A_n \stackrel{d}{\sim} \cal B_n$ if
	$$\lim_{n\to \infty}\left[ \E f(\cal A_n)-\E f(\cal B_n)\right] =0$$
	for any bounded continuous function $f$.
\end{definition}  

In the proof, we will frequently use the following simple fact, which can be proved using characteristic functions. Given two sequences of random vectors $\cal A_n$ and $\cal B_n$, suppose that conditioning on $\cal A_n$, we have $\cal B_n \stackrel{d}{\sim} \cal D_n$, where $\cal D_n$ has an asymptotic distribution that does not depend on $\cal A_n$. Then, we have that
\be\label{simple} \cal A_n + \cal B_n \stackrel{d}{\sim}\cal A_n + \cal D_n ,\ee
where on the right-hand side $\cal D_n$ is independent of $\cal A_n$. One immediate use of this fact is to decouple the randomness of $\cal M (\theta_l)$ from that of $Z$ (and hence $\wt Z$) as long as we can show that conditioning on $Z$, the limiting distribution of $\cal M (\theta_l)$ does not depend on $Z$. 

\subsection{Step 1: Rewriting $\cal M(x)$}

We start with some linear algebra to write $\cal M(x)$ into a form that is more amenable to our analysis. Our main tool is the rotational invariance of multivariate Gaussian distributions. 

First, notice that since $\|\mathbf u_i^a\|_2=1$ and $\|\mathbf u_i^b\|_2=1$ for $1\le i \le r$, we have 
\be\label{sets_large_entries}\Big|\Big\{k: \max_{1\le i \le r}|u_i^a(k)| > n^{-\tau_0}\Big\}\Big| \le r n^{2\tau_0}, \quad \Big|\Big\{k: \max_{1\le i \le r} |u_i^b(k)| > n^{-\tau_0}\Big\}\Big| \le rn^{2\tau_0}.\ee
We permute the rows of $\bU_a$, $\bU_b$, $X$ and $Y$ using $p\times p$ and $q\times q$ permutation matrices $P_1$ and $P_2$:
\begin{align*} {\cal M}(\theta_l)=&\ \sqrt{n}\begin{pmatrix} \bU_a^{\top}P_1^{\top} & 0 & 0 & 0 \\ 0 & \bU_b^{\top} P_2^{\top}& 0 & 0 \\ 0 & 0 & \wt Z & 0 \\ 0 & 0 & 0 & \wt Z \end{pmatrix}  \begin{pmatrix} P_1 & 0 & 0 & 0 \\ 0 &  P_2& 0 & 0 \\ 0 & 0 & I_n & 0 \\ 0 & 0 & 0 & I_n \end{pmatrix} \\
	&\times [G(\theta_l)-\Pi(\theta_l)] \begin{pmatrix} P_1^{\top} & 0 & 0 & 0 \\ 0 &  P_2^{\top} & 0 & 0 \\ 0 & 0 & I_n & 0 \\ 0 & 0 & 0 & I_n \end{pmatrix} \begin{pmatrix} P_1\bU_a & 0 & 0 & 0 \\ 0 & P_2\bU_b & 0 & 0 \\ 0 & 0 & \wt Z^\top & 0 \\ 0 & 0 & 0 & \wt Z^\top  \end{pmatrix} .\end{align*}
We can choose $P_1$ and $P_2$ such that all the ``large" entries of $\bU_a$ and $\bU_b$ in the two sets of \eqref{sets_large_entries} are now in the first $\rho$ rows of $ P_1\bU_a$ and $ P_2\bU_a$ for some integer $\rho\le r n^{2\tau_0}$. Without loss of generality, we rename $ P_1\bU_a$ and $ P_2\bU_a$ as $\bU_a$ and $\bU_b$. Then, 
we can assume that $\bU_a$ and $\bU_b$ take the forms
\be\label{U1U2} 
\bU_a =\begin{pmatrix} \mathbf O_1 \\ \bO'_1\end{pmatrix}, \quad \bU_b =\begin{pmatrix} \mathbf O_2 \\ \bO'_2\end{pmatrix},
\ee
where $\mathbf O_1$, $\mathbf O_2$ are $\rho\times r$ matrices, $\bO'_1$ is a $(p-\rho)\times r$ matrix, $\bO'_2$ is a $(q-\rho)\times r$ matrix, and $\|\bO'_1\|_{\max} \le n^{-\tau_0}$, $\|\bO'_2\|_{\max} \le n^{-\tau_0}$.
On the other hand, we have 
\begin{align*}  
	& \begin{pmatrix} P_1 & 0 & 0 & 0 \\ 0 &  P_2& 0 & 0 \\ 0 & 0 & I_n & 0 \\ 0 & 0 & 0 & I_n \end{pmatrix}   [G(\theta_l)-\Pi(\theta_l)] \begin{pmatrix} P_1^{\top} & 0 & 0 & 0 \\ 0 &  P_2^{\top} & 0 & 0 \\ 0 & 0 & I_n & 0 \\ 0 & 0 & 0 & I_n \end{pmatrix} \\
	& =\begin{bmatrix} 0 & \begin{pmatrix} P_1 X & 0\\ 0 & P_2 Y\end{pmatrix}\\ \begin{pmatrix} X^{\top}P_1^\top & 0\\ 0 &  Y^{\top}P_2^\top\end{pmatrix}  & \begin{pmatrix}  \theta_l  I_n & \theta_l^{1/2}I_n\\ \theta_l^{1/2} I_n &  \theta_l I_n\end{pmatrix}^{-1}\end{bmatrix}^{-1} -\Pi(\theta_l)	 .
\end{align*}
Again, without loss of generality, we rename the permuted matrices $P_1 X$ and $P_2 Y$ as $X$ and $Y$. Then, because of \eqref{almost GuassianX} and \eqref{almost GuassianY}, $X$ and $Y$ take the forms
$$X= \begin{pmatrix} X_1 \\ X_2\end{pmatrix}, \quad Y= \begin{pmatrix} Y_1 \\ Y_2\end{pmatrix},$$
where $X_1$, $Y_1$ are $\rho \times n$ matrices, $X_2$ is a $(p-\rho)\times n$ Gaussian matrix and $Y_2$ is a $(q-\rho)\times n$ Gaussian matrix. Next, we rotate $\bO'_1$ and $\bO'_2$ using orthogonal $(p-\rho)\times (p-\rho)$ and $(q-\rho)\times (q-\rho)$ matrices $\wt S_{1}$ and $\wt S_2$ so that 
$$ \wt S_1^\top \bO'_1 =\begin{pmatrix} \wt{\mathbf O}'_1\\ 0\end{pmatrix}, \quad  \wt S_2^\top {\mathbf O}'_2 =\begin{pmatrix} \wt{\mathbf O}'_2\\ 0\end{pmatrix},$$ 
where $\wt{\mathbf O}'_{1}$ and $\wt{\mathbf O}'_{2}$ are $r\times r$ matrices satisfying that 
\be\label{OOT}
\mathbf O_\al^{\top}\mathbf O_\al+  (\bO'_\al)^{\top} \bO'_\al= \mathbf O_\al^{\top}\mathbf O_\al+  (\wt{\mathbf O}'_\al)^{\top} \wt{\mathbf O}'_\al = I_r,\quad \al=1,2.
\ee
Similarly, we rotate $\wt Z^\top$ using an orthogonal $ n \times n$ matrix $\wt S=(\wt Z^\top , S)$, where $S$ is an $n\times (n-r)$ matrix satisfying $S^\top S=I_{n-r}$ and $ S^\top \wt Z^\top =0$.  

With the above notations, we can rewrite $\cal M$ in \eqref{defM} as 
\begin{align}
	{\cal M}
	  {=}&\ \sqrt{n}\begin{pmatrix} \wt\bU_a^{\top} & 0 & 0 & 0 \\ 0 & \wt\bU_b^{\top} & 0 & 0 \\ 0 & 0 & \mathbf I^{\top} & 0 \\ 0 & 0 & 0 & \mathbf I^{\top} \end{pmatrix} \begin{bmatrix} 0 & \begin{pmatrix} \wt X & 0\\ 0 &  \wt Y\end{pmatrix}\\ \begin{pmatrix} \wt X^{\top} & 0\\ 0 &  \wt Y^{\top}\end{pmatrix}  & \begin{pmatrix}  \theta_l  I_n & \theta_l^{1/2}I_n\\ \theta_l^{1/2} I_n &  \theta_l I_n\end{pmatrix}^{-1}\end{bmatrix}^{-1} \begin{pmatrix} \wt\bU_a & 0 & 0 & 0 \\ 0 & \wt\bU_b & 0 & 0 \\ 0 & 0 & \mathbf I & 0 \\ 0 & 0 & 0 & \mathbf I  \end{pmatrix} \nonumber\\
	&- \sqrt{n}\Pi_{2r,2r}(\theta_l),\label{M_another}
\end{align}
where $\Pi_{2r,2r}$ is a $4r\times 4r$ matrix defined as 
\be\label{defwtpi2} \Pi_{2r,2r} := \begin{bmatrix}  \begin{pmatrix} c_1^{-1}m_{1c} I_{r} & 0 \\ 0 & c_2^{-1}m_{2c} I_{r}\end{pmatrix} & 0\\0 & \begin{pmatrix}  m_{3c} I_r & h I_r \\ h I_r &  m_{4c} I_r\end{pmatrix} \end{bmatrix},\ee
and we have abbreviated that
$$\wt \bU_a:= \begin{pmatrix} \mathbf O_1 \\ \wt{\mathbf O}'_1 \\ 0\end{pmatrix},\ \ \wt \bU_b:= \begin{pmatrix} \mathbf O_2 \\ \wt{\mathbf O}'_2 \\ 0\end{pmatrix}, \ \ \mathbf I:= \begin{pmatrix} I_r  \\ 0\end{pmatrix}, \ \ \wt X:=\begin{pmatrix} I_\rho & 0 \\ 0 & \wt S^{\top}_1\end{pmatrix} X \wt S,\ \  \wt Y:=\begin{pmatrix} I_\rho & 0 \\ 0 & \wt S^{\top}_2\end{pmatrix} Y \wt S. $$
Using the rotational invariance of $X_2$, we can write $\wt X$ as
$$ \wt X \stackrel{d}{=}\begin{pmatrix} X_1 \wt Z^\top , X_1 S \\ X_2\end{pmatrix}\equiv \begin{pmatrix} X_1 \wt Z^\top & X_1 S  \\ X^{(1)}_{L} & X^{(1)}_{R}\\ X^{(2)}_{L} & X^{(2)}_{R}\end{pmatrix} ,$$
where ``$\stackrel{d}{=}$" means ``equal in distribution", and $X_{L}^{(1)}$, $X_{L}^{(2)}$, $X_R^{(1)}$ and $X_R^{(2)}$ are respectively $r\times r$, $(p-\rho-r)\times r$, $r\times (n-r)$ and $(p-\rho-r)\times (n-r)$ Gaussian matrices.  We have a similar decomposition for $Y$:
$$ \wt Y \stackrel{d}{=}\begin{pmatrix} Y_1 \wt Z^\top , Y_1 S \\ Y_2\end{pmatrix}\equiv \begin{pmatrix} Y_1 \wt Z^\top & Y_1 S  \\ Y^{(1)}_{L} & Y^{(1)}_{R}\\ Y^{(2)}_{L} & Y^{(2)}_{R}\end{pmatrix}.$$

For simplicity, we introduce the notations $\wt r= r+ \rho$ and 
\begin{align*}
	\mathbb T :=& \left\{1,\cdots, \wt r\right\} \cup \left\{p+1,\cdots, p+\wt r\right\}  \cup \left\{ p+q+1, \cdots, p+q+ r\right\} \\
	&\cup \left\{ p+q+n+1,\cdots, p+q+n+ r\right\}.
\end{align*} 
Then, applying the Schur complement formula to \eqref{M_another}, we obtain that
\be\label{Mthetad}\cal M\stackrel{d}{=} \sqrt{n}\left[\bO^\top  \cal H_{2\wt r, 2r}^{-1} \bO- \Pi_{2r,2r}(\theta_l)\right],\ee
where 
$\bO$ and $\cal H_{2\wt r, 2r}$ are $(2\wt r+ 2r)\times 4r$ and $(2\wt r + 2r)\times (2\wt r + 2r)$ matrices defined as (recall Definition \ref{defminor})
\begin{align*} 
& \bO:= \begin{bmatrix} \begin{pmatrix}\bO_1 \\ \wt{\mathbf O}'_1 \end{pmatrix} & 0 & 0 & 0 \\ 0 &  \begin{pmatrix}\bO_2\\ \wt{\mathbf O}'_2 \end{pmatrix} & 0 & 0 \\ 0 & 0 & I_r & 0 \\ 0 & 0 & 0 & I_r \end{bmatrix},\\
&	\cal H_{2\wt r, 2r}: = \begin{bmatrix} 0 \cdot I_{2\wt r} & 0 \\  0 & \begin{pmatrix}\theta_l I_{r} & \theta_l^{1/2}I_{r} \\ \theta_l^{1/2}I_{r} & \theta_l I_{r} \end{pmatrix}^{-1}\end{bmatrix} +  H_1 - F^{\top}G^{(\mathbb T)}(\theta_l) F ,
\end{align*}
and $H_1$ and $F$ are $(2\wt r+ 2r)\times (2\wt r+ 2r)$ and $(p+q+2n-2\wt r - 2r)\times (2\wt r + 2r)$ matrices defined as 
\begin{align*}
&H_1:=  \begin{bmatrix} 0  \cdot I_{2\wt r} & \begin{pmatrix} \begin{pmatrix} X_1\wt Z^\top   \\ X_{L}^{(1)} \end{pmatrix}  & 0 \\ 0 &  \begin{pmatrix} Y_1 \wt Z^\top   \\ Y_{L}^{(1)} \end{pmatrix} \end{pmatrix} \\  0 & 0 \cdot I_{2r}\end{bmatrix} + c.t. , \\ 
&F:=\begin{bmatrix} 0 & 0 & X^{(2)}_{L}  & 0 \\ 0 &  0 & 0 &   Y^{(2)}_{L} \\  \begin{pmatrix}S^{\top} X_1^{\top}, (X_R^{(1)})^{\top}\end{pmatrix}  & 0 & 0 & 0 \\ 0 &  \begin{pmatrix}S^{\top} Y_{1}^{\top}, (Y_R^{(1)})^{\top}\end{pmatrix} &0 &0 \end{bmatrix}.
\end{align*}
Here, ``$c.t.$" means the (conjugate) transpose of the preceding term. 
Using \eqref{useful}, we can rewrite $\cal H_{2\wt r, 2r}$ as 
\begin{align} 
\cal H_{2\wt r, 2r}: =&\, \Pi_{2\wt r,2r}^{-1} +  H_1 + \begin{pmatrix}  m_{3c}I_{\wt r}  & 0 & 0 & 0 \\ 0 & m_{4c}I_{\wt r}  & 0 & 0  \\  0 & 0 & m_{1c}I_{r}  &  0 \\ 0 & 0 & 0 & m_{2c}I_{r} \end{pmatrix}  \nonumber\\
& - F^{\top} \Pi^{(\mathbb T)} F - F^{\top}(G^{(\mathbb T)}-\Pi^{(\mathbb T)}) F,\label{BrH}
\end{align}
where $\Pi^{(\mathbb T)}$ is the minor of $\Pi$ as defined in Definition \ref{defminor} and $\Pi_{2\wt r,2r} $ is defined in a similar way as \eqref{defwtpi2}: 
\be\label{defwtpi} \Pi_{2\wt r,2r} := \begin{bmatrix}  \begin{pmatrix} c_1^{-1}m_{1c} I_{\wt r} & 0 \\ 0 & c_2^{-1}m_{2c} I_{\wt r}\end{pmatrix} & 0\\0 & \begin{pmatrix}  m_{3c} I_r & h I_r \\ h I_r &  m_{4c} I_r\end{pmatrix} \end{bmatrix}.\ee

\subsection{Step 2: Concentration estimates}
In this step, we establish some (almost) sharp concentration estimates on the terms in \eqref{BrH}. More precisely, we claim that
\be\label{claimF1}
F^{\top} F -\begin{pmatrix}   I_{\wt r}  & 0 & 0 & 0 \\ 0 &  I_{\wt r}  & 0 & 0  \\  0 & 0 & c_1I_{r}  &  0 \\ 0 & 0 & 0 & c_2I_{r} \end{pmatrix}= \OO_{\prec} (n^{-1/2+2\tau_0}), 
\ee
and 
\be\label{claimF2} 
F^{\top} \Pi^{(\mathbb T)}F - \mathbb E_F ( F^{\top} \Pi^{(\mathbb T)}F ) =\OO_{\prec} (n^{-1/2+2\tau_0}), 
\ee
where $\mathbb E_F$ denotes the partial expectation over the randomness in $F$ and conditioning on $Z$. (To avoid confusion, we emphasize that the matrix $S$ is deterministic conditioning on $Z$.) Using the facts $S^\top S=I_{n-r}$ and $\wt r =\OO( n^{2\tau_0})$, we get that
\be\label{claimF3}\mathbb E_F\left( F^{\top} \Pi^{(\mathbb T)}F\right)=\begin{pmatrix}  m_{3c} I_{\wt r}  & 0 & 0 & 0 \\ 0 & m_{4c} I_{\wt r}  & 0 & 0  \\  0 & 0 & m_{1c} I_{r}  &  0 \\ 0 & 0 & 0 & m_{2c} I_{r} \end{pmatrix} + \OO \left(n^{-1+2\tau_0}\right).\ee

Both the estimates \eqref{claimF1} and \eqref{claimF2} follow from Lemma \ref{largedeviation}. We consider the terms $(X_{L}^{(2)})^{\top} X_{L}^{(2)}$, $X_1  S S^{\top} X_1^{\top}$ and $X_1  S (X_{R}^{(1)})^{\top}$ as examples, {where recall that $X_{L}^{(2)}$, $X_1$ and $X_{R}^{(1)}$ are $(p-\wt r)\times r$, $\rho\times n$ and $r\times (n-r)$ random matrices with i.i.d. entries of mean zero and variance $n^{-1}$.} For $p+q+1\le \mu , \nu\le p+q+r$,  we have that
\begin{align*}
	\Big|\left[(X_{L}^{(2)})^{\top} X_{L}^{(2)}\right]_{\mu\nu} - \frac{p-\wt r}{n}\delta_{\mu\nu}\Big| &= \Big| \sum_{ \wt r+1 \le i \le p} \left(X_{i\mu}X_{i\nu}-  n^{-1}\delta_{\mu\nu}\right)\Big|\prec \OO( n^{-1/2}).
\end{align*}
For $1\le i \le \rho$, we have that
\begin{align*}
& \left|\left(X_1  S S^{\top} X_1^{\top}\right)_{ii} -  {n}^{-1}\tr \left(SS^{\top}\right) \right|   = \Big| \sum_{  \mu\ne \nu  \in \cal I_3}  X_{i\mu}X_{i\nu} (SS^{\top})_{\mu\nu}\Big| + \Big| \sum_{  \mu  \in \cal I_3}  (X_{i\mu}^2 - n^{-1})(SS^{\top})_{\mu\mu}\Big| \\
& \prec \frac{1}{n}  \Big( \sum_{  \mu\ne \nu  \in \cal I_3}  [(SS^{\top})_{\mu\nu}]^2 \Big)^{1/2} + \frac{1}{n}  \Big( \sum_{  \mu  \in \cal I_3}  [(SS^{\top})_{\mu\mu}]^2 \Big)^{1/2}  \le \frac{2}{n}  \left\{ \tr \left[ (SS^{\top})^2\right] \right\}^{1/2}  = \OO(n^{-1/2}),
\end{align*}
while for $1\le i < j\le \rho$, we have that
\begin{align*}
	(X_1  S S^{\top} X_1^{\top})_{ij} &=  \sum_{  \mu,\nu  \in \cal I_3}  X_{i\mu}X_{j\nu} (SS^{\top})_{\mu\nu} \prec \frac{1}{n}  \Big( \sum_{  \mu,\nu  \in \cal I_3}  [(SS^{\top})_{\mu\nu}]^2 \Big)^{1/2}  \\
	&=\frac{1}{n}  \left\{ \tr \left[ (SS^{\top})^2\right] \right\}^{1/2} = \OO(n^{-1/2}).
\end{align*}
Using the fact $\tr (SS^{\top}) = n - r$, the above two estimates actually give the estimate
$$\left|(X_1  S S^{\top} X_1^{\top})_{ij}-\delta_{ij} \right| \prec n^{-1/2},\quad 1\le i ,j \le \rho.$$
Finally, for $1\le i \le \rho$ and $\rho+1\le j  \le \rho+r$, we have that
\begin{align*}
	\left[X_1  S (X_{R}^{(1)})^{\top}\right]_{ij} &=  \sum_{  \mu, \nu  \in \cal I_3}  X_{i\mu}X_{j\nu} S_{\mu\nu} \prec \frac{1}{n}  \Big( \sum_{  \mu, \nu  \in \cal I_3}  S_{\mu\nu}^2 \Big)^{1/2} = \frac{1}{n}  \left[ \tr  (SS^{\top}) \right]^{1/2}  = \OO(n^{-1/2}).
\end{align*}
With similar arguments as above, using Lemma \ref{largedeviation}, we can obtain the following concentration estimates: for any constant $\e>0$, with high probability, 
\be\label{FFT}
\begin{split}
	 \left\|(X_{L}^{(2)})^{\top} X_{L}^{(2)} - c_1 I_{ r}\right\|_{\max}\le  n^{-1/2+\e}, \quad  & \left\|(Y_{L}^{(2)})^{\top} Y_{L}^{(2)} - c_2 I_{ r}\right\|_{\max}\le  n^{-1/2+\e}, \\ 
	 \left\|X_1  S S^{\top} X_1^{\top} -  I_{\rho}\right\|_{\max}\le  n^{-1/2+ \e}, \quad &\left\|Y_1  S S^{\top} Y_1^{\top} -  I_{\rho}\right\|_{\max} \le  n^{-1/2+ \e}, \\\
	  \left\|X_1  S S^{\top} Y_1\right\|_{\max}\le  n^{-1/2+ \e}, \quad    &\left\| X_R^{(1)}(X_R^{(1)})^{\top} - I_r\right\|_{\max}\le n^{-1/2+\e},\\ 
	 \left\| Y_R^{(1)}(Y_R^{(1)})^{\top} - I_r\right\|_{\max}\le n^{-1/2+\e} ,\quad &\left\| X_R^{(1)}(Y_R^{(1)})^{\top}\right\|_{\max}\le n^{-1/2+\e},\\
	  \left\|X_1  S (X_{R}^{(1)})^{\top}\right\|_{\max} \le  n^{-1/2+ \e}, \quad  &\left\|X_1  S (Y_R^{(1)})^{\top}\right\|_{\max}\le  n^{-1/2 +\e}, \\
	\left\|Y_1  S (X_{R}^{(1)})^{\top}\right\|_{\max} \le  n^{-1/2 +\e} , \quad &\left\|Y_1  S (Y_{R}^{(1)})^{\top}\right\|_{\max} \le  n^{-1/2+\e}.
\end{split}
\ee
 These estimates immediately imply \eqref{claimF1} and \eqref{claimF2} by bounding the operator norms of error matrices by their Frobenius norms.


By \eqref{claimF1}, we have that $\|F\|=\OO(1)$ with high probability. Then, using the local law \eqref{aniso_outstrong} and the fact that $F$ is independent of $G^{(\mathbb T)}$, we get that
$$\left\| F^{\top}(G^{(\mathbb T)}-\Pi^{(\mathbb T)}) F\right\|  \le (2\wt r + 2r) \left\| F^{\top}(G^{(\mathbb T)}-\Pi^{(\mathbb T)}) F\right\|_{\max}\prec n^{-1/2+2\tau_0}. $$
Under the moment assumption \eqref{eq_highmoment}, every entry of $H_1$ is of order $\OO_\prec(n^{-1/2})$ by Markov's inequality, so we have that
$$\|H_1\|  \le  (2\wt r + 2r) \|H_1\|_{\max}\prec n^{-1/2+2\tau_0}.$$ 
Finally, by \eqref{claimF2} and \eqref{claimF3}, we have that
$$\left\| \begin{pmatrix}  m_{3c}I_{\wt r}  & 0 & 0 & 0 \\ 0 & m_{4c}I_{\wt r}  & 0 & 0  \\  0 & 0 & m_{1c}I_{r}  &  0 \\ 0 & 0 & 0 & m_{2c}I_{r} \end{pmatrix}- F^{\top} \Pi^{(\mathbb T)} F\right\| \prec n^{-1/2+2\tau_0} .$$
Hence, for $\cal M$ in \eqref{Mthetad}, taking the inverse of \eqref{BrH} and performing a simple Taylor expansion, we get that 
\be\label{Mthetad2}
\begin{split}
	\cal M\stackrel{d}{=} & \sqrt{n}\bO^{\top}   \Pi_{2\wt r,2r} \left[ - H_1  + (1-\E_F)  (F^{\top} \Pi^{(\mathbb T)} F ) + F^{\top}(G^{(\mathbb T)}-\Pi^{(\mathbb T)}) F \right] \Pi_{2\wt r,2r} \bO  \\
	&+ \OO_\prec(n^{-1/2+4\tau_0}),
\end{split}
\ee
where we used \eqref{claimF3} and $\bO^{\top}   \Pi_{2\wt r,2r} \bO=\Pi_{2r, 2r} $. Since $\tau_0$ can be taken as small as possible, it suffices to study the CLT of the first term in \eqref{Mthetad2}.

\subsection{Step 3: CLT of the resolvent}
In this step, we establish the CLT of the resolvent term $\sqrt{n}  F^{\top} (G^{(\mathbb T)}-\Pi^{(\mathbb T)}) F $ in \eqref{Mthetad2}. 
Conditioning on $F$, we have the following lemma, whose proof will be given in Section \ref{sec Gauss}.

\begin{lemma}\label{Gauss lemma}
	Fix any $F$ such that the estimates in \eqref{FFT} hold for a small enough constant $\e>0$. Then, we have that (recall Definition \ref{def_simD}) 
	\begin{align}\label{eq_FGF}
		&\sqrt{n} \bO^{\top}  F^{\top} (G^{(\mathbb T)}-\Pi^{(\mathbb T)}) F \bO \stackrel{d}{\sim}  \begin{pmatrix} a_{11}g_{11} & a_{12}g_{12} & a_{13}g_{13} & a_{14}g_{14} \\  
			a_{21}g_{21} & a_{22}g_{22}  & a_{23}g_{23} & a_{24}g_{24} \\  
			a_{31}g_{31} & a_{32}g_{32} & a_{33}g_{33} & a_{34}g_{34} \\  
			a_{41}g_{41} & a_{42}g_{42} & a_{43}g_{43} & a_{44}g_{44} \\  
		\end{pmatrix}  ,
	\end{align}
	where $g_{\al\beta}$, $1\le \al \le \beta \le 4$, are independent Gaussian matrices satisfying the following properties: $g_{\al\beta}=g_{ \beta\al}^\top$, $1\le \al< \beta \le 4$, are $r\times r$ random matrices with i.i.d. Gaussian entries $(g_{\al \beta})_{ij}\sim \cal N(0,1)$; 
	$g_{\al\al}$, $1\le \al \le 4$, are $r\times r$ symmetric GOE (Gaussian orthogonal ensemble) with entries $(g_{\al\al})_{ij}\sim \cal N(0,1+\delta_{ij})$. Moreover, the coefficients are given by 
	\be\label{defQa}
	\begin{split}
		& a_{11}:= m_{3c}\sqrt{ \frac{a_c^2 + c_1}{1-c_1} + \frac{a_c^2}{c_1}} ,\quad  a_{12}=a_{21}:=h \sqrt{  \frac{a_c^2 }{c_2} t_l^2+ \frac{a_c^2 + c_2}{1-c_2} } ,\\
		& a_{13}=a_{31}:=\sqrt{\frac{a_c^2 + c_1}{ 1-c_1}}, \quad  a_{14}=a_{41}:= \frac{a_c}{\sqrt{c_1}}\frac{m_{3c}}{h},\quad  a_{22}:= m_{4c}\sqrt{ \frac{a_c^2 + c_2}{1-c_2} + \frac{a_c^2}{c_2}} ,\\
		& a_{23}=a_{32}:=\frac{a_c}{\sqrt{c_2}}\frac{m_{4c}}{h},\quad a_{24}=a_{42}:=\sqrt{ \frac{ a_c^2 + c_2}{  1-c_2} }, \quad a_{33}:=m_{3c}^{-1}\sqrt{c_1\frac{a_c^2 + c_1}{ 1-c_1 }}, \\ 
		& a_{34}=a_{43}:=\frac{ a_c} { h},\quad a_{44}:=m^{-1}_{4c}\sqrt{ c_2\frac{ a_c^2 + c_2}{ 1-c_2} },
	\end{split} \ee
	where we have introduced the notation 
	\be\label{ac2}
	a_c^2:=\frac{ t_c^2 }{ t_l^2 - t_c^2}.\ee
\end{lemma}

With Lemma \ref{Gauss lemma}, we get the weak convergence
\be\label{CLTEr2(1)}
\begin{split}
	 &\sqrt{n}\bO^{\top}   \Pi_{2\wt r,2r}  F^{\top}(G^{(\mathbb T)}-\Pi^{(\mathbb T)}) F \Pi_{2\wt r,2r} \bO \\ 
	 & \Rightarrow  \ \ \Pi_{2r,2r} \begin{pmatrix} a_{11}g_{11} & a_{12}g_{12} & a_{13}g_{13} & a_{14}g_{14} \\  
		a_{21}g_{21} & a_{22}g_{22}  & a_{23}g_{23} & a_{24}g_{24} \\  
		a_{31}g_{31} & a_{32}g_{32} & a_{33}g_{33} & a_{34}g_{34} \\  
		a_{41}g_{41} & a_{42}g_{42} & a_{43}g_{43} & a_{44}g_{44} \\  
	\end{pmatrix}  \Pi_{2r,2r}, 
\end{split}
\ee
using the simple identity $\Pi_{2\wt r,2r} \bO=\bO\Pi_{2r, 2r}$, with $\Pi_{2r, 2r}$ defined in \eqref{defwtpi2}.

\subsection{Step 4: Calculating the limiting covariances} \label{sec_var}

In this step, we expand \eqref{Mthetad2} and show a CLT for each term. The main technical work is to calculate the limiting covariance functions. Lemma \ref{Gauss lemma} already gives the CLT for $\sqrt{n}\bO^{\top}   \Pi_{2\wt r,2r}  F^{\top}(G^{(\mathbb T)}-\Pi^{(\mathbb T)}) F \Pi_{2\wt r,2r} \bO$. 
We still need to study the term 
\begin{align*}
	&  \sqrt{n}\bO^{\top}   \Pi_{2\wt r,2r} \left[ - H_1  + (1-\E_F) \left(F^{\top} \Pi^{(\mathbb T)} F\right) \right] \Pi_{2\wt r,2r} \bO \\
&= \Pi_{2r,2r} Q_{4r}\Pi_{2r,2r} = \Pi_{2r,2r}\begin{pmatrix}Q_1 & Q_{3}\\ Q_{4} & Q_2\end{pmatrix} \Pi_{2r,2r},
\end{align*}
where $Q_{4r}$ is a $4r\times 4r$ symmetric matrix, with $Q_1$, $ Q_{2}$, $Q_{3}$ and $Q_4$ being the $2r\times 2r$ blocks defined by 
\begin{align*}
	&Q_1:= \begin{pmatrix} Q_1^{(1)} & Q_1^{(3)} \\ Q_1^{(4)} & Q_1^{(2)} \end{pmatrix},\quad  Q_2:=\sqrt{n}\begin{pmatrix} -m^{-1}_{3c}\IE(X_{L}^{(2)})^{\top}X_{L}^{(2)} & 0 \\ 0 & -m^{-1}_{4c}\IE(Y_{L}^{(2)})^{\top} Y_{L}^{(2)} \end{pmatrix},\\
	& Q_{3}=Q_{4}^{\top}:=\sqrt{n}\begin{pmatrix}  - \bO_1^{\top} X_1\wt Z^\top  - (\wt{\mathbf O}'_1)^{\top} X_{L}^{(1)} &  0 \\ 0 & - \bO_2^{\top} Y_1 \wt Z^\top   - (\wt{\mathbf O}'_2)^{\top} Y_{L}^{(1)} \end{pmatrix}.
\end{align*}
Here, we have abbreviated $\IE:= 1- \E_F$, and the four $r\times r$ blocks of $Q_L$ are defined as
\begin{align*}
	&Q_1^{(1)}:=\sqrt{n}\IE \left[m_{3c}\left( \bO_1^{\top} X_1S + (\wt{\mathbf O}'_1)^{\top} X_R^{(1)}\right)\left(S^{\top}X_1^{\top} \bO_1  + (X_R^{(1)})^{\top}\wt{\mathbf O}'_1 \right)\right],\\
	&Q_1^{(2)}:= \sqrt{n}\IE\left[m_{4c} \left( \bO_2^{\top} Y_1S + (\wt{\mathbf O}'_2)^{\top} Y_R^{(1)}\right) \left(S^{\top}Y_{1}^{\top}\bO_2 +  (Y_R^{(1)})^{\top}\wt {\mathbf O}'_2\right) \right],\\
	&Q_1^{(3)}=(Q_1^{(4)})^\top:= \sqrt{n}\IE\left[ h  \left( \bO_1^{\top} X_1S + (\wt{\mathbf O}'_1)^{\top} X_R^{(1)}\right) \left(S^{\top}Y_{1}^{\top}\bO_2 +  (Y_R^{(1)})^{\top}\wt {\mathbf O}'_2\right)\right].
\end{align*}

Now, using \eqref{defQQQ}, \eqref{M-M0}, \eqref{Mthetad2}, \eqref{CLTEr2(1)} and the simple fact \eqref{simple}, we obtain that
\begin{align}
	\sqrt{n}\cal O^{\top} \cal E^{(g)}_r \cal O \stackrel{d}{\sim} &  t_l m_{3c} \bW^{\top}  \begin{pmatrix} a_{11}g_{11} & a_{12}g_{12} & a_{13}g_{13} & a_{14}g_{14} \\  
		a_{21}g_{21} & a_{22}g_{22}  & a_{23}g_{23} & a_{24}g_{24} \\  
		a_{31}g_{31} & a_{32}g_{32} & a_{33}g_{33} & a_{34}g_{34} \\  
		a_{41}g_{41} & a_{42}g_{42} & a_{43}g_{43} & a_{44}g_{44} \\  
	\end{pmatrix} \bW + t_l m_{3c}  \bW^{\top} Q_{4r}\bW .\label{midEr2(1)}
\end{align}
Here, the $4r\times r$ matrix $\bW$ is defined as
\begin{align*}
	\bW:= \Pi_{2r,2r}\mathfrak W\cal O=  \begin{pmatrix} -m_{3c}^{-1}\bW_1   \\     {h}{m_{3c}^{-1} m_{4c}^{-1}}\mathbf W_2  \\  t_l^{-1}\mathbf W_3 \\   {h}{m_{3c}^{-1}}\mathbf W_4 \end{pmatrix},
\end{align*}
where we have abbreviated that
\be\label{defn_AB}
\begin{split}
		\bW_1:= \left(I_{r} + \Sigma_a^2\right)^{-1/2} \cal O ,\quad & \bW_2:= \left(1+\Sigma_b^2\right)^{-1/2}\wh\Sigma_b \bV_b^\top \bV_a \wh\Sigma_a \cal O, \\
	\mathbf W_3:= t_l \bV_a\wh \Sigma_a \cal O -\bV_b \wh\Sigma_b^2  \bV_b^\top \bV_a \wh \Sigma_a \cal O,\quad & \mathbf W_4:= \bV_a\wh \Sigma_a \cal O- \bV_b \wh\Sigma_b^2  \bV_b^\top \bV_a \wh \Sigma_a \cal O.
\end{split}
\ee
In the derivation, we also used that $f_c(\theta_l)= {m_{3c}(\theta_l)m_{4c}(\theta_l)}/{h^2(\theta_l)} = t_l $.
Expanding \eqref{midEr2(1)}, we get that
\begin{align}
	&\sqrt{n}\cal O^{\top} \cal E^{(g)}_r \cal O \nonumber\\
	&\stackrel{d}{\sim}  t_l \bW_1^{\top}\left[  \frac{a_{11}}{m_{3c}}g_{11} +\sqrt{n} \IE\left( \bO_1^{\top} X_1S + (\wt{\mathbf O}'_1)^{\top} X_R^{(1)}\right)\left(S^{\top}X_1^{\top} \bO_1  + (X_R^{(1)})^{\top}\wt{\mathbf O}'_1 \right)\right]\bW_1 \nonumber\\ 
	& - \left\{ \bW_1^{\top}  \left[  \frac{a_{12}}{h}g_{12} + \sqrt{n} \left( \bO_1^{\top} X_1S + (\wt{\mathbf O}'_1)^{\top} X_R^{(1)}\right) \left(S^{\top}Y_{1}^{\top}\bO_2 +  (Y_R^{(1)})^{\top}\wt{\mathbf O}'_2\right)\right]\bW_2 + c.t.\right\}  \nonumber\\
	& + \bW_2^{\top}  \left[ \frac{a_{22}}{m_{4c}}g_{22} + \sqrt{n} \IE\left( \bO_2^{\top} Y_1S + (\wt{\mathbf O}'_2)^{\top} Y_R^{(1)}\right) \left(S^{\top}Y_{1}^{\top}\bO_2 +  (Y_R^{(1)})^{\top}\wt{\mathbf O}'_2\right)\right] \bW_2 \nonumber\\
	& -  \left[\bW_1^{\top} \left(a_{13}g_{13} - \sqrt{n} \bO_1^{\top} X_1\wt Z^\top  - \sqrt{n} (\wt{\mathbf O}'_1)^{\top} X_{L}^{(1)} \right)  \mathbf W_3 + c.t.\right]  \nonumber\\
	&- \left[  \bW_1^{\top} \left(\frac{m_{4c}}{h} a_{14}g_{14}\right)\mathbf W_4 +c.t.\right] +\left[ \bW_2^{\top} \left( \frac{h}{m_{4c}}  a_{23}g_{23}  \right) \mathbf W_3 + c.t. \right]\nonumber\\
	& + \left[\bW_2^{\top} \left(a_{24} g_{24}-\sqrt{n}\bO_2^{\top} Y_1 \wt Z^\top   -\sqrt{n} (\wt{\mathbf O}'_2)^{\top} Y_{L}^{(1)}\right) \mathbf W_4+c.t. \right] \nonumber\\
	& +t_l^{-1}\bW_3^{\top}  \left( m_{3c}a_{33}g_{33} -  \sqrt{n} \IE (X_{L}^{(2)})^{\top}X_{L}^{(2)} \right)\bW_3   \nonumber\\
	& +  \bW_4^{\top}\left(m_{4c}a_{44}g_{44} - \sqrt{n} \IE(Y_{L}^{(2)})^{\top} Y_{L}^{(2)} \right)\bW_4+\left[ \bW_3^{\top}  \left(h  a_{34}g_{34}  \right) \bW_4 +c.t. \right] ,\label{OEO1}
\end{align}
where recall that ``$c.t.$" denotes the (conjugate) transpose of the preceding term. Note $\sqrt{n}X_{L}^{(1)}$ and $\sqrt{n}Y_{L}^{(1)}$ are $r\times r$ matrices with i.i.d.\;Gaussian entries of mean 0 and variance $1$, and they are independent of all the other terms. So we rename them as two new Gaussian matrices 
\be\label{OEO2}
\wt g_{13}:=-\sqrt{n}X_{L}^{(1)}, \quad \wt g_{24}:=-\sqrt{n}Y_{L}^{(1)}.
\ee
Moreover, the matrices $ \sqrt{n} \IE (X_{L}^{(2)})^{\top}X_{L}^{(2)} $ and $ \sqrt{n} \IE (Y_{L}^{(2)})^{\top}Y_{L}^{(2)} $ are also independent of all the other terms. With classical CLT, we obtain that
\be\label{OEO4}
-  \sqrt{n} \IE (X_{L}^{(2)})^{\top}X_{L}^{(2)}  \stackrel{d}{\sim} \sqrt{c_1}\wt g_{33}, \quad -  \sqrt{n} \IE (Y_{L}^{(2)})^{\top}Y_{L}^{(2)}  \stackrel{d}{\sim}  \sqrt{c_2}\wt g_{44},
\ee
where $\wt g_{33}$ and $\wt g_{44}$ are $r\times r$ symmetric GOE with entries $(\wt g_{33})_{ij}\sim \cal N(0,1+\delta_{ij})$ and $(\wt g_{44})_{ij}\sim \cal N(0,1+\delta_{ij})$. 

Now, to conclude the CLT for \eqref{OEO1}, it remains to show the CLT for the matrix
\be\label{OEO50}
\begin{split}
	\Theta:=&\ t_l \bW_1^{\top}\left[  \sqrt{n} \IE\left( \bO_1^{\top} X_1S + \wt \bW_1^{\top} X_R^{(1)}\right)\left(S^{\top}X_1^{\top} \bO_1  + (X_R^{(1)})^{\top}\wt \bW_1 \right)\right]\bW_1\\
	& - \left\{\bW_1^\top \left[ \sqrt{n} \left( \bO_1^{\top} X_1S + \wt \bW_1^{\top} X_R^{(1)}\right) \left(S^{\top}Y_{1}^{\top}\bO_2 +  (Y_R^{(1)})^{\top}\wt \bW_2\right)\right]\bW_2  + c.t. \right\} \\
	& + \bW_2^{\top}   \left[  \sqrt{n} \IE\left( \bO_2^{\top} Y_1S + \wt \bW_2^{\top} Y_R^{(1)}\right) \left(S^{\top}Y_{1}^{\top}\bO_2 +  (Y_R^{(1)})^{\top}\wt \bW_2\right)\right] \bW_2 \\
	&+\left[\bW_1^{\top} \left( \sqrt{n} \bO_1^{\top} X_1\wt Z^\top  \right) \bW_3 + c.t.\right] - \left[\bW_2^{\top} \left( \sqrt{n}\bO_2^{\top} Y_1 \wt Z^\top   \right) \bW_4 + c.t.\right]  .
\end{split}
\ee
We decompose $\Theta$ into the sum of four matrices, $\Theta:= \Theta_1 + \Theta_2 + \Theta_3 + \Theta_4,$ as follows. We first group all terms depending on \smash{$Y_R^{(1)}$} into $\Theta_1$,
\begin{align*}
	\Theta_1:=&\ \bW_2^{\top} \left[\sqrt{n}\IE (\wt{\mathbf O}'_2)^{\top} Y_R^{(1)}(Y_R^{(1)})^{\top}\wt{\mathbf O}'_2 +\sqrt{n} \left(\bO_2^{\top} Y_1S (Y_R^{(1)})^{\top}\wt{\mathbf O}'_2 + c.t.\right)\right] \bW_2  \\
	& -\left[\bW_1^{\top}\left( \sqrt{n} \left( \bO_1^{\top} X_1S + (\wt{\mathbf O}'_1)^{\top} X_R^{(1)}\right) (Y_R^{(1)})^{\top}\wt{\mathbf O}'_2 \right)\bW_2  + c.t.\right]  , 
\end{align*}
all the remaining terms depending on \smash{$X_R^{(1)}$} into $\Theta_2$,
\begin{align*}
	\Theta_2:= &\ t_l  \bW_1^{\top}\left[  \sqrt{n} \IE (\wt{\mathbf O}'_1)^{\top} X_R^{(1)}(X_R^{(1)})^{\top}\wt{\mathbf O}'_1  +  \sqrt{n}\left(\bO_1^{\top} X_1S (X_R^{(1)})^{\top}\wt{\mathbf O}'_1 + c.t.\right)\right]\bW_1 \\
	&- \left[\bW_2^{\top} \left( \sqrt{n}   \bO_2^{\top}Y_{1}S (X_R^{(1)})^{\top} \wt{\mathbf O}'_1\right)\bW_1 + c.t. \right] ,
\end{align*}
all the remaining terms depending on $X_1$ into $\Theta_3$,
\begin{align*}
	\Theta_3:=& \left[\bW_1^{\top} \left( \sqrt{n} \bO_1^{\top} X_1\wt Z^\top  \right)  \bW_3 + c.t.\right]  -\left[\bW_1^{\top}\left( \sqrt{n}  \bO_1^{\top} X_1 SS^\top Y_{1}^{\top}\bO_2  \right)\bW_2  + c.t.\right]  \\
	& + t_l\bW_1^{\top}\left[  \sqrt{n} \IE( \bO_1^{\top} X_1 SS^\top  X_1^{\top} \bO_1    )\right]\bW_1,
\end{align*}
and finally all the remaining terms depending on $Y_1$ into $\Theta_4$,
\begin{align*}
	\Theta_4:=   - \left[ \bW_2^{\top} \left( \sqrt{n}\bO_2^{\top} Y_1 \wt Z^\top   \right)  \bW_4+ c.t. \right]  +  \bW_2^{\top}  \left[  \sqrt{n} \IE\left( \bO_2^{\top} Y_1 SS^\top Y_{1}^{\top}\bO_2 \right)\right] \bW_2.
\end{align*}
Using \eqref{eq_highmoment} and Lemma \ref{largedeviation}, we can obtain the following large deviation estimates as in \eqref{FFT}: for any small constant $\e>0$, with high probability,
\begin{align}
	&   \| X_1 \wt Z^\top \|_{\max} +\|X_1\|_{\max} \le n^{-1/2+\e},  \quad \|X_1 X_1^{\top} -  I_{\rho}\|_{\max}\le  n^{-1/2 +\e}, \label{FFT3} \\
	& \|  Y_1 \wt Z^\top \|_{\max} +\|Y_1\|_{\max} \le n^{-1/2+\e}, \quad \|Y_1  Y_1^{\top} -  I_{\rho}\|_{\max} \le  n^{-1/2 +\e}.\label{FFT2} 
\end{align}
Combining \eqref{FFT3} and \eqref{FFT2} with the facts $SS^\top = I_n-VV^\top$ and $\rho =\OO(n^{2\tau_0})$, we can simplify $\Theta_3$ and $\Theta_4$ as
$$\Theta_\al=\Theta_\al' + \OO_\prec(n^{-1/2+4\tau_0}),\quad \al =3,4,$$
where
\begin{align*}
	\Theta_3'&:=  \left[\bW_1^{\top} \left( \sqrt{n} \bO_1^{\top} X_1\wt Z^\top  \right)  \bW_3 + c.t.\right]  -\left[\bW_1^{\top}\left( \sqrt{n}  \bO_1^{\top} X_1  Y_{1}^{\top}\bO_2  \right)\bW_2  + c.t.\right]  \\
	& \quad \ + t_l\bW_1^{\top}\left[  \sqrt{n} \IE\left( \bO_1^{\top} X_1  X_1^{\top} \bO_1    \right)\right]\bW_1,\\
	\Theta_4'&:=   - \left[ \bW_2^{\top} \left( \sqrt{n}\bO_2^{\top} Y_1 \wt Z^\top   \right)  \bW_4+ c.t. \right]  +  \bW_2^{\top}  \left[  \sqrt{n} \IE\left( \bO_2^{\top} Y_1   Y_{1}^{\top}\bO_2 \right)\right] \bW_2.
\end{align*}

The next lemma shows that $\Theta_1$, $\Theta_2$, $\Theta_3'$, and $\Theta_4'$ are all asymptotically Gaussian. It has several different proofs using some classical techniques for CLT. For the reader's convenience, we give a proof based on Stein's method in Appendix \ref{appd asymG}. 

\begin{lemma}\label{asymp_Gauss}
	We have the following results conditioning on $\wt Z$ satisfying \eqref{Vmax}:
	\begin{itemize}
		\item[(i)] conditioning on $X_1$, $Y_1$ and $X_R^{(1)}$ satisfying \eqref{FFT}, $\Theta_1$ is asymptotically Gaussian with zero mean;  
		
		\item[(ii)] conditioning on $X_1$ and $Y_1$ satisfying \eqref{FFT}, $\Theta_2$ is asymptotically Gaussian with zero mean;  
		
		\item[(iii)] conditioning on $Y_1$ satisfying \eqref{FFT2}, $\Theta_3'$ is asymptotically Gaussian with zero mean;  
		
		\item[(iv)] $\Theta_4'$ is asymptotically Gaussian with zero mean.
	\end{itemize}
\end{lemma}

With Lemma \ref{asymp_Gauss}, we obtain that $\Theta$ converges in distribution to a centered Gaussian matrix. It remains to determine the covariance of this matrix. First, we calculate the covariance for $\Theta_1$. Conditioning on $X_1$, $Y_1$ and \smash{$X_R^{(1)}$} satisfying \eqref{FFT} and using $\wt r =\OO(n^{2\tau_0})$, we have that 
$$(\bW_2^{\top} \bO_2^{\top} Y_1S S^{\top} Y_1^{\top}\bO_2 \bW_2)_{ij} =( \bW_2^{\top} \bO_2^{\top}  \bO_2 \bW_2)_{ij} + \OO (n^{-1/2+2\tau_0+\e}),$$
and
\begin{align*}
&	\left[\bW_1^{\top} \left( \bO_1^{\top} X_1S + (\wt{\mathbf O}'_1)^{\top} X_R^{(1)}\right) \left( S^{\top}X_1^{\top}\bO_1  + (X_R^{(1)})^{\top}\wt{\mathbf O}'_1 \right) \bW_1\right]_{ij}\\
& = (\bW_1^{\top}  \bW_1)_{ij}+ \OO (n^{-1/2+2\tau_0+\e}).
\end{align*}
With these two identities, we can calculate that  
\begin{align*}
 &\E_{Y_R^{(1)}} (\Theta_1)_{ij}(\Theta_1)_{i'j'} \\
 &=  \left(\bW_2^{\top}  \bW_2 \right) _{ii'} (\bW_2^{\top} (\wt{\mathbf O}'_2)^{\top} \wt{\mathbf O}'_2\bW_2 )_{jj'}+   (\bW_2^{\top} (\wt{\mathbf O}'_2)^{\top} \wt{\mathbf O}'_2\bW_2)_{ii'}(\bW_2^{\top} \bO_2^{\top} \bO_2\bW_2)_{jj'}\\
	&+(\bW_1^{\top}\bW_1)_{ii'}  (\bW_2^{\top} (\wt{\mathbf O}'_2)^{\top} \wt{\mathbf O}'_2\bW_2 )_{jj'} +(\bW_2^{\top} (\wt{\mathbf O}'_2)^{\top} \wt{\mathbf O}'_2\bW_2 )_{ii'} (\bW_1^{\top}\bW_1)_{jj'} \\
	&+ (i' \leftrightarrow j') +  \OO (n^{-1/2+2\tau_0+\e}),
\end{align*}
where {$\E_{Y_R^{(1)}}$} denotes the partial expectation over $Y_R^{(1)}$ and $(i' \leftrightarrow j')$ means an expression obtained by exchanging $i' $ and $j'$ in \emph{all the preceding terms} (i.e., the first four terms on the right-hand side).
Similarly, conditioning on $X_1$ and $Y_1$ satisfying \eqref{FFT}, we can calculate that
\begin{align*}
	&\E_{X_R^{(1)}} (\Theta_2)_{ij}(\Theta_2)_{i'j'} \\
	&=  t_l^2 (\bW_1^{\top}\bW_1)_{ii'}  (\bW_1^{\top} (\wt{\mathbf O}'_1)^{\top} \wt{\mathbf O}'_1\bW_1 )_{jj'} +  t_l^2 (\bW_1^{\top} (\wt{\mathbf O}'_1)^{\top} \wt{\mathbf O}'_1\bW_1 )_{ii'} (\bW_1^{\top}\bO_1^{\top}\bO_1\bW_1)_{jj'}\\
	&+ (\bW_1^{\top} (\wt{\mathbf O}'_1)^{\top} \wt{\mathbf O}'_1\bW_1 )_{ii'}(\bW_2^{\top} \bO_2^{\top} \bO_2\bW_2)_{jj'} + (\bW_2^{\top} \bO_2^{\top} \bO_2\bW_2 )_{ii'} (\bW_1^{\top} (\wt{\mathbf O}'_1)^{\top} \wt{\mathbf O}'_1\bW_1)_{jj'}  \\
	&+ (i' \leftrightarrow j') +  \OO (n^{-1/2+2\tau_0+\e}).
\end{align*}
For $\Theta_3'$ and $\Theta_4'$, the entries of $X_1$ and $Y_1$ are not Gaussian anymore. Hence, the covariances of $\Theta_3'$ and $\Theta_4'$ will depend on the third and fourth moments of $X_1$ and $Y_1$. First, we can calculate the covariance for $\Theta_3'$: 
\be
\begin{split}\label{theta33}
	 &\E_{X_1} (\Theta_3')_{ij}(\Theta_3')_{i'j'} \\
	 &=\,  (\bW_1^{\top}\bO_1^{\top}\bO_1\bW_1)_{ii'} \left[( \bW_3^{\top} \wt Z - \bW_2^{\top} \bO_2^{\top} Y_1)(\wt Z ^\top\bW_3 - Y_1^{\top}\bO_2\bW_2 )\right]_{jj'} \\
	&+  \left[( \bW_3^{\top} \wt Z - \bW_2^{\top} \bO_2^{\top} Y_1)(\wt Z^\top \bW_3 - Y_1^{\top}\bO_2\bW_2 )\right]_{ii'} (\bW_1^{\top}\bO_1^{\top}\bO_1\bW_1)_{jj'}  \\
	&+ t_l^2 (\bW_1^{\top}\bO_1^{\top}\bO_1\bW_1)_{ii'}(\bW_1^{\top}\bO_1^{\top}\bO_1\bW_1)_{jj'}+ (i' \leftrightarrow j') +K_3+K_4,
\end{split}
\ee
where $K_3$ is a third moment term defined as
\begin{align*}
	K_3:=	&\left(n^{3/2}\mathbb EX_{11}^3\right)\cdot   \frac{t_l}{\sqrt{n}}\Big[\sum_{1\le k \le \rho, \mu \in \cal I_3} \left(\bO_1\bW_1\right)_{ki}\left(\bO_1\bW_1\right)_{ki'}\left(\bO_1\bW_1\right)_{kj'} \\
	&\qquad \qquad \qquad \qquad \qquad \qquad \qquad \times (\wt Z^\top \bW_3 - Y_1^{\top}\bO_2\bW_2 )_{\mu j}  +(i\leftrightarrow j)\Big] \\
	+ & \left(n^{3/2}\mathbb EX_{11}^3\right)\cdot \frac{t_l}{\sqrt{n}} \Big[ \sum_{1\le k \le \rho, \mu \in \cal I_3} \left(\bO_1\bW_1\right)_{ki}\left(\bO_1\bW_1\right)_{kj} \left(\bO_1\bW_1\right)_{ki'}\\
	&\qquad \qquad \qquad \qquad \qquad \qquad \qquad \times  (\wt Z^\top  \bW_3 - Y_1^{\top}\bO_2\bW_2 )_{\mu j'} + (i' \leftrightarrow j')\Big],
\end{align*}
and $K_4$ is a fourth cumulant term defined as (recall \eqref{mux4})
$$K_4:=t_l^2 \kappa_x^{(4)}\sum_{1\le k\le\rho} \left(\bO_1 \bW_1\right)_{ki}\left(\bO_1 \bW_1\right)_{ki'}\left(\bO_1 \bW_1\right)_{kj}\left(\bO_1 \bW_1\right)_{kj'} .$$

Using Lemma \ref{largedeviation}, we can check that 
$$\|Y_1\mathbf e\|_{\max} \prec n^{-1/2}, \quad \text{for}\quad \mathbf e:= n^{-1/2}(1,1,\cdots, 1)^{\top} \in \R^n.$$
Applying this estiamate and \eqref{FFT2}, we obtain that
\begin{align*} 
&( \bW_3^{\top} \wt Z - \bW_2^{\top} \bO_2^{\top} Y_1 ) (\wt Z^\top \bW_3 - Y_1^{\top}\bO_2\bW_2  ) \\
& = \bW_3^{\top} \bW_3 +  \bW_2^{\top} \bO_2^{\top}\bO_2\bW_2 +\OO_\prec(n^{-1/2+2\tau_0}),
\end{align*}
and for any $1\le i\le r$,
$$\frac1{\sqrt{n}} \sum_{ \mu \in \cal I_3} \left( Y_1^{\top}\bO_2\bW_2 \right)_{\mu i} = ( \mathbf e^{\top}Y_1^{\top}\bO_2\bW_2 )_{i} =\OO_\prec(n^{-1/2+2\tau_0}).$$
On the other hand, using 
\eqref{eq_iso} and \eqref{V-F}, we obtain that
\be\label{rowsum_V} \|\wt Z\mathbf e\|_{\max} \le \|Z\mathbf e\|_{\max} + \|(\wt Z-Z)\mathbf e\|_{\max} \prec n^{-1/2},\ee
which implies that for any $1\le i\le r$,
$$\frac1{\sqrt{n}} \sum_{ \mu \in \cal I_3} ( \wt Z^\top \bW_3)_{\mu i} = (\mathbf e^\top \wt Z^\top \bW_3)_i = \OO_\prec(n^{-1/2}).$$
The above calculations show that $K_3$ is negligible. For $K_4$, by the assumption of Proposition \ref{main_prop1}, we have that $\|{\mathbf O}'_1\|_{\max}\le n^{-\tau_0}$, which gives  
$ ({\mathbf O}'_1\bW_1)_{ki} \lesssim n^{-\tau_0}$
for any $k$. With this fact, we obtain that 
\begin{align*}
	& \sum_{\rho+1\le k \le p} ({\mathbf O}'_1 \bW_1)_{ki}({\mathbf O}'_1 \bW_1)_{ki'}({\mathbf O}'_1 \bW_1)_{kj}({\mathbf O}'_1 \bW_1)_{kj'}  \\
	&\lesssim n^{-2\tau_0}\sum_{\rho+1\le k \le p} ({\mathbf O}'_1 \bW_1)_{ki}({\mathbf O}'_1 \bW_1)_{ki'} \lesssim n^{-2\tau_0},
\end{align*} 
where ${\mathbf O}'_1$ is defined in \eqref{U1U2}. 
Thus, we can replace $\bO_1 \bW_1$ with $\bU_a \bW_1$ in $K_4$ up to a negligible error. Collecting the above estimates, we can simplify \eqref{theta33} as
\begin{align*}
	  & \E_{X_1} (\Theta_3')_{ij}(\Theta_3')_{i'j'}   \\
	  &=\,   t_l^2 (\mathbf W_1^{\top}\mathbf O_1^{\top}\mathbf O_1\mathbf W_1)_{ii'}(\mathbf W_1^{\top}\mathbf O_1^{\top}\mathbf O_1\mathbf W_1)_{jj'}+ (\mathbf W_1^{\top}\mathbf O_1^{\top}\mathbf O_1\mathbf W_1)_{ii'} ( \mathbf W_3^{\top} \mathbf W_3 +  \mathbf W_2^{\top} \mathbf O_2^{\top}\mathbf O_2\mathbf W_2 )_{jj'} \\
	  &  +( \mathbf W_3^{\top}\mathbf W_3 + \mathbf W_2^{\top}\mathbf O_2^{\top}\mathbf O_2\mathbf W_2 )_{ii'}(\mathbf W_1^{\top}\mathbf O_1^{\top}\mathbf O_1\mathbf W_1)_{jj'} +  (i' \leftrightarrow j')  \\
	  & +t_l^2 \kappa_x^{(4)} \sum_{k\in \cal I_1} \cal U_{ki}\cal U_{ki'}\cal U_{kj}\cal U_{kj'} + \OO(n^{- 2\tau_0 })
\end{align*}
with high probability, where we recall the notations in \eqref{defUV} and \eqref{defn_AB}. With similar calculations, we can obtain the covariance for $\Theta_4'$: with high probability,
\begin{align*}
	\E_{Y_1} (\Theta_4')_{ij}(\Theta_4')_{i'j'}  = & \ (  \bW_2^{\top} \bO_2^{\top}\bO_2\bW_2)_{ii'} ( \bW_4^{\top} \bW_4 )_{jj'} +  (\bW_4^{\top} \bW_4 )_{ii'} ( \bW_2^{\top} \bO_2^{\top}\bO_2\bW_2)_{jj'}\\
	& + (  \bW_2^{\top} \bO_2^{\top}\bO_2\bW_2)_{ii'}(  \bW_2^{\top} \bO_2^{\top}\bO_2\bW_2)_{jj'} +  (i' \leftrightarrow j')  \\
	& + \kappa_y^{(4)} \sum_{k\in \cal I_2} \left(\bU_b \bW_2\right)_{ki}\left(\bU_b \bW_2\right)_{ki'}\left(\bU_b \bW_2\right)_{kj}\left(\bU_b \bW_2\right)_{kj'}+ \OO(n^{- 2\tau_0 }).
\end{align*}


Combining all the above calculations, 
we have shown that $\Theta= \Theta_1 + \Theta_2 + \Theta_3 + \Theta_4$ converges weakly to a centered Gaussian random matrix, denoted by $g_\Theta$, with covariance 
\begin{align}
	&\E (g_\Theta)_{ij} (g_\Theta)_{i'j'}  \label{OEO5}\\
	& =  t_l^2 (\bW_1^{\top}\bW_1)_{ii'}  (\bW_1^{\top} \bW_1 )_{jj'} +  (\bW_2^{\top}  \bW_2 ) _{ii'} (\bW_2^{\top} \bW_2 )_{jj'} \nonumber\\
	&+(\bW_1^{\top}\bW_1)_{ii'}  (\bW_2^{\top} \bW_2 )_{jj'}   +  (\bW_2^{\top} \bW_2 )_{ii'} (\bW_1^{\top}\bW_1)_{jj'}  \nonumber\\
	&+(\bW_1^{\top}\bO_1^{\top}\bO_1\bW_1)_{ii'} \left( \bW_3^{\top} \bW_3 \right)_{jj'}  +  \left( \bW_3^{\top} \bW_3   \right)_{ii'} (\bW_1^{\top}\bO_1^{\top}\bO_1\bW_1)_{jj'} \nonumber\\
	&+(  \bW_2^{\top} \bO_2^{\top}\bO_2\bW_2)_{ii'}( \bW_4^{\top} \bW_4 )_{jj'} + (\bW_4^{\top} \bW_4 )_{ii'}( \bW_2^{\top} \bO_2^{\top}\bO_2\bW_2)_{jj'}+  (i' \leftrightarrow j') \nonumber \\
	&+t_l^2 \kappa_x^{(4)} \sum_{k\in \cal I_1} \cal U_{ki}\cal U_{ki'}\cal U_{kj}\cal U_{kj'} + \kappa_y^{(4)} \sum_{k\in \cal I_2} \left(\bU_b \bW_2\right)_{ki}\left(\bU_b \bW_2\right)_{ki'}\left(\bU_b \bW_2\right)_{kj}\left(\bU_b \bW_2\right)_{kj'}.\nonumber
\end{align}
Notice that for any $i\in \gamma(l)$, we have that (recall \eqref{defUV} and \eqref{defn_AB}),
\be\label{SVD_use}
(\bU_b \bW_2)_{ki}=\left[ \sqrt{t_l} + \OO(n^{-1/2+\delta})\right]\cal V_{ki} ,
\ee
where we used the SVD \eqref{Mrab} and the fact  $t_i=t_l +\OO(n^{-1/2+\delta})$ for any $i \in \gamma(l)$ by Definition \ref{Def gammal}. Hence, up to a negligible error, the last term in \eqref{OEO5} can be replaced by 
$$\kappa_y^{(4)} \sum_{k\in \cal I_2} \cal V_{ki}\cal V_{ki'}\cal V_{kj}\cal V_{kj'},\quad \text{for \ \ $i,j,i',j'\in \gamma(l)$}.$$

\subsection{Step 5: Concluding the proof}\label{sec_var2}
Finally, combing \eqref{OEO1}, \eqref{OEO2}, \eqref{OEO4} and \eqref{OEO5}, after a straightforward algebraic calculation (where a computer algebra system may help), we obtain that \smash{$(\sqrt{n}\cal O^{\top} \cal E^{(g)}_r \cal O)_{\llbracket\gamma(l)\rrbracket}$} converges weakly to an $r\times r$ centered Gaussian matrix \smash{$\Upsilon_l^{(g)}$} with
\begin{align*}
	&\E (\Upsilon^{(g)}_l)_{ij}(\Upsilon^{(g)}_l)_{i'j'}  \\
	&=  t^2_l \left( \frac{a_c^2 + c_1}{1-c_1}  + \frac{a_c^2}{c_1}+1\right) (1-\cal A)_{ii'} (1-\cal A)_{jj'}  +  \left(  \frac{a_c^2 + c_2}{1-c_2} + \frac{a_c^2}{c_2}+1\right)  \cal B_{ii'} \cal B_{jj'}  \\
	&+ \left( \frac{a_c^2}{c_2} t_l^2 + \frac{a_c^2 + c_2}{1-c_2}+1\right)\left[ (1-\cal A)_{ii'} \cal B_{jj'} +\cal B_{ii'} (1-\cal A)_{jj'} \right] \\
	&+  \left( \frac{a_c^2+c_1}{1-c_1} +1\right) \left[ (1-\cal A)_{ii'} (\cal C_1)_{jj'} +(\cal C_1)_{ii'}(1-\cal A)_{jj'}  \right] \\
	&+  \frac{a^2_c }{c_1} t_l^2   \left[ (1-\cal A)_{ii'} (\cal C_2)_{jj'} +(\cal C_2)_{ii'}(1-\cal A)_{jj'}  \right] \\
	& +  \frac{ a^2_c}{c_2}   \left[ \cal B_{ii'} (\cal C_1)_{jj'} +(\cal C_1)_{ii'}\cal B_{jj'}\right]  + \left( \frac{a_c^2+c_2}{1-c_2} +1\right)  \left[ \cal B_{ii'} (\cal C_2)_{jj'} +(\cal C_2)_{ii'}\cal B_{jj'} \right] \\
	& +t_l^{-2}\left(c_1\frac{a_c^2 + c_1}{ 1-c_1 } + c_1\right) (\cal C_1)_{ii'} (\cal C_1)_{jj'} + \left(c_2\frac{ a_c^2 + c_2}{ 1-c_2} + c_2\right)  (\cal C_2)_{ii'} (\cal C_2)_{jj'}  \\
	& +  {a_c^2} \left[ (\cal C_1)_{ii'} (\cal C_2)_{jj'} +(\cal C_2)_{ii'}(\cal C_1)_{jj'} \right] \\
	&+ (i' \leftrightarrow j') +t_l^2 \kappa_x^{(4)} \sum_k \cal U_{ki}\cal U_{ki'}\cal U_{kj}\cal U_{kj'}+ \kappa_y^{(4)} \sum_k \cal V_{ki}\cal V_{ki'}\cal V_{kj}\cal V_{kj'} ,
\end{align*}
where we recall that $(i' \leftrightarrow j')$ means an expression obtained by exchanging $i'$ and $j'$ in all the preceding terms (i.e., the terms in the first seven lines), and we have introduced the following notations:
\begin{align} 
	& \cal A:=1- \bW_1^{\top}\bW_1=\cal O^{\top}\wh\Sigma_a^2 \cal O, \quad   \cal B:=\bW_2^{\top}\bW_2= \cal O^{\top}\wh\Sigma_a  \bV_a^\top \bV_b \wh\Sigma_b \left( 1+\Sigma_b^2 \right)^{-1}\wh\Sigma_b \bV_b^\top \bV_a\wh\Sigma_a\cal O, \nonumber \\
	&	\cal C_1 : = \bW_3^\top \bW_3=t_l^2 \cal A + (1- 2t_l) \cal C  - \cal B , \quad
	\cal C_2 :  =  \bW_4^\top \bW_4=  \cal A -   \cal C   - \cal B, \label{eq_F1F2}
\end{align}
with $\cal C$ defined as
$$  \cal C:=\cal O^{\top}\wh \Sigma_a \bV_a^\top \bV_b \wh\Sigma_b^2  \bV_b^\top \bV_a\wh \Sigma_a \cal O=\diag(t_1, \cdots, t_r) .$$
Then, we plug \eqref{eq_F1F2} into $\E (\Upsilon^{(g)}_l)_{ij}(\Upsilon^{(g)}_l)_{i'j'} $ and simplify the resulting expression. After a straightforward algebraic calculation, we can show that 
\begin{align*}
	&\E (\Upsilon^{(g)}_l)_{ij}(\Upsilon^{(g)}_l)_{i'j'} =  \delta_{ii'} \left[ t_l^2  \frac{ a_c^2 + c_1}{ c_1(1-c_1)}   +\left(  \frac{a_c^2+1}{1-c_1} (1- 2t_l) - \frac{t_l^2 a^2_c}{c_1}\right) \cal C   \right]_{jj'} \\
	& +  \cal C_{ii'} \left[  \left(\frac{a_c^2+1}{1-c_1}(1- 2t_l)  - \frac{t_l^2 a^2_c}{c_1}\right)  + \left(\frac{(1-c_2)(1- 2t_l)^2}{c_2} + \frac{(1-c_1)t_l^2}{c_1}- 2(1- 2t_l)   \right)a_c^2\cal C \right]_{jj'} \\
	&- (1- 2t_l) \left(\cal A_{ii'}  \cal C_{jj'} + \cal C_{ii'}\cal A_{jj'}\right)  - \left(\cal B_{ii'} \cal C_{jj'}+ \cal C_{ii'}\cal B_{jj'}\right) -t_l^2 \cal A _{ii'} \cal A_{jj'}- \cal B_{ii'}\cal B_{jj'}\\
	&+ \left(\cal A _{ii'} \cal B_{jj'}+ \cal B_{ii'} \cal A_{jj'} \right) + (i' \leftrightarrow j') +t_l^2 \kappa_x^{(4)} \sum_{k} \cal U_{ki}\cal U_{ki'}\cal U_{kj}\cal U_{kj'}+ t_l^2 \kappa_y^{(4)} \sum_{k} \cal V_{ki} \cal V_{ki'} \cal V_{kj} \cal V_{kj'} .
\end{align*}

\vspace{-5pt}
On the other hand, using \eqref{defPPP} and \eqref{CLTe1}, we can check that $(\sqrt{n}\cal O^{\top} \cal E^{(z)}_r \cal O)_{\llbracket\gamma(l)\rrbracket}$ converges weakly to an $r\times r$ centered Gaussian matrix \smash{$\Upsilon_l^{(z)}$} with (recall \eqref{Wij})
\begin{align*}
	\E (\Upsilon^{(z)}_l)_{ij}(\Upsilon^{(z)}_l)_{i'j'} =   (2t_l-1)\cal C_{ii'}  \cal C_{jj'}  + t_l^2 \cal A_{ii'} \cal A_{jj'} + \cal B_{ii'}\cal B_{jj'} + (1-2t_l) \left(\cal A_{ii'}  \cal C_{jj'} + \cal C_{ii'}\cal A_{jj'}\right) \\
	- \left(\cal A _{ii'} \cal B_{jj'}+ \cal B_{ii'} \cal A_{jj'} \right)+ \left(\cal B_{ii'} \cal C_{jj'}+ \cal C_{ii'}\cal B_{jj'}\right) + (i' \leftrightarrow j') + \kappa_z^{(4)}\sum_{k} \cal W_{k,ij}\cal W_{k,i'j'}.  
\end{align*}
Then, by \eqref{simple}, we know that 
$$\left(\sqrt{n} \cal O^{\top} \cal E_r (\theta_l)\cal O\right)_{\llbracket\gamma(l)\rrbracket} = \big(\sqrt{n}\cal O^{\top} \cal E^{(z)}_r \cal O\big)_{\llbracket\gamma(l)\rrbracket} + \big(\sqrt{n}\cal O^{\top} \cal E^{(g)}_r \cal O\big)_{\llbracket\gamma(l)\rrbracket}$$ 
converges weakly to a centered Gaussian matrix $\wt\Upsilon_l$ with covariance  
\be\nonumber \E (\wt\Upsilon_l)_{ij}(\wt\Upsilon_l)_{i'j'}=\E (\Upsilon^{(z)}_l)_{ij}(\Upsilon^{(z)}_l)_{i'j'}  + \E (\Upsilon^{(g)}_l)_{ij}(\Upsilon^{(g)}_l)_{i'j'} .\ee
Finally, using $\cal C_{jj'}=t_l\delta_{jj'} + \OO(n^{-1/2+\delta})$ for $j,j'\in \gamma(l)$, we can check that the covariance functions of $\wt\Upsilon_l$ are asymptotically equal to \eqref{Cijij}. This concludes Proposition \ref{main_prop1}, which gives Theorem \ref{main_thm1} in the almost Gaussian case by Proposition \ref{redGthm}. 

\section{Proof of Lemma \ref{Gauss lemma}}\label{sec Gauss}

In this section, we give the proof of Lemma \ref{Gauss lemma}, which, as we have seen, is a key step in the proof of Proposition \ref{main_prop1}. Under the setting of Lemma \ref{Gauss lemma}, we need to study the CLT of the matrix
	\begin{align*} 
		\cal Q_0(\theta_l):&=\sqrt{n} \mathscr V_0^{\top} \left( G^{(\mathbb T)}(\theta_l) - \Pi^{(\mathbb T)}(\theta_l) \right)  \mathscr V_0,\quad \text{where}\quad \mathscr V_0\equiv \begin{pmatrix} 0 & 0 & \bV_1 & 0 \\ 0 & 0 & 0 & \bV_2 \\ \bV_3 & 0 & 0 & 0 \\ 0 & \bV_4 & 0 & 0 \end{pmatrix} := F\bO.
	\end{align*}
	It is easy to check that the matrices $\bV_{1}$, $\bV_2$, $\bV_3$ and $\bV_4$ are respectively $(p-\wt r)\times r$, $(q-\wt r)\times r$, $(n-r)\times r$ and $(n-r)\times r$ random matrices independent of $G^{(\mathbb T)}$, and they satisfy that 
	with high probability,
	\be\label{V12}\bV_1^{\top} \bV_1= c_1 I_{r} + \OO_\prec (n^{-1/2 }), \  \bV_2^{\top} \bV_2 = c_2I_{r} +\OO _\prec(n^{-1/2 }),\ee
	\be\label{V34}\begin{split}
		\bV_3^{\top} \bV_3 =  I_{ r}  +\OO_\prec & (n^{-\frac12+2\tau_0 }),\   \bV_4^{\top} \bV_4 =  I_{ r}  +\OO_\prec (n^{-\frac12+2\tau_0 }),  \  \bV_3^{\top}\bV_4   =\OO_\prec (n^{-\frac12+2\tau_0  }).
	\end{split}
	\ee
	These conditions all follow from \eqref{FFT} and \eqref{OOT}. For simplicity of notations, we permute the columns of $\mathscr V_0$ and study the CLT of 
	\be\label{permu}\begin{pmatrix} 0 & I_{2r} \\ I_{2r} & 0\end{pmatrix}\sqrt{n}  \mathscr V_0^{\top} (G^{(\mathbb T)}-\Pi^{(\mathbb T)})  \mathscr V_0\begin{pmatrix} 0 & I_{2r} \\ I_{2r} & 0\end{pmatrix}.\ee
Moreover, with a slight abuse of notation, we rename $(X^{(\mathbb T)},Y^{(\mathbb T)},G^{(\mathbb T)})$ as $(X,Y,G)$ and study the CLT of the following matrix under the conditions \eqref{V12} and \eqref{V34}:
	\begin{align} \label{defn_4rQ}
		\cal Q(\theta_l):&=\sqrt{n} \mathscr V^{\top} \left[ G(X,Y,\theta_l) - \Pi (\theta_l)\right]  \mathscr V,
	\end{align}
	where
	$$\ \ \mathscr V:= \mathscr V_0 \begin{pmatrix} 0 & I_{2r} \\ I_{2r} & 0\end{pmatrix}=\begin{pmatrix} \bV_1 & 0 & 0 & 0 \\ 0 & \bV_2 & 0 & 0 \\ 0 & 0 & \bV_3 & 0 \\ 0 & 0 & 0 & \bV_4 \end{pmatrix}.$$
Since $|\mathbb T| \lesssim n^{2\tau_0}$, we have $(n-|\mathbb T|)/n = 1+ \OO(n^{-1+2\tau_0}),$ where $\OO(n^{-1+2\tau_0})$ is a negligible error. Hence, without loss of generality, we still assume that the dimensions of $X$ and $Y$ are $p\times n $ and $q\times n$ in order to simplify notations.

	In our proof, in order to avoid singular behaviors of $G$ on exceptional low-probability events, we will use a regularized resolvent \smash{$\wh G(z)$} defined as follows. 
			
	\begin{definition}[Regularized resolvent]\label{resol_not2}
	For $z = E+ \ii \eta \in \mathbb C_+,$ we define the regularized resolvent $\wh G(z)$ as
	$$ \wh G(z) := \left[H(z)- z n^{-10} \begin{pmatrix} I_{p+q} & 0 \\ 0 & 0 \end{pmatrix} \right]^{-1} .$$
	\end{definition}
		
	The main reason for introducing the regularized resolvent is that it satisfies the deterministic bound: 
	\be\label{op G}
	\| \wh G(z)\|  \lesssim  n^{10} \eta^{-1},\quad \text{for}\quad \eta=\im z. 
	\ee 
	This estimate has been proved in Lemma 3.6 of \cite{PartIII}. In particular, if we choose  $\eta \ge n^{-C}$ for a constant $C>0$, then \eqref{op G} justfies the assumption of Lemma \ref{lem_stodomin} (iii), which will be used in the proof when we bound expectations of polynomials of regularized resolvent entries. With a standard perturbation argument, we can easily control the difference between $\wh G(z)$ and $G(z)$. 
	
	\begin{claim}\label{removehat}
	Suppose there exists a high probability event $\Xi$ on which $ \| G(z)\|_{\max}=\OO(1)$ for $z$ belonging to some subset. Then, we have that
	\be\label{g-whg} \|G(z) - \wh G(z)\|_{\max} \le n^{-8}\quad \text{ on } \quad \Xi.\ee
	\end{claim}
	\begin{proof}
	For $t\in [0,1]$, we define 
    $$ G_t(z) := \left[H(z)- t z n^{-10} \begin{pmatrix} I_{p+q} & 0 \\ 0 & 0 \end{pmatrix} \right]^{-1}, \quad \text{with}\quad G_0(z)= G(z), \quad G_1(z)=\wh  G(z) .$$
	Taking the derivative with respect to $t$, we immediately obtain that
	\be\label{partialtG}\partial_t G_t(z) = zn^{-10} G_t(z) \begin{pmatrix} I_{p+q} & 0 \\ 0 & 0 \end{pmatrix} G_t(z) . \ee
	Thus, applying Gronwall's inequality to 
	$$ \|G_t(z)\|_{\max}\le \| G(z)\|_{\max} + C n^{-9} \int^t_0 \|G_s(z)\|_{\max}^2  \dd s, $$
	we get that $\max_{0\le t \le 1}\|G_t(z)\|_{\max}=\OO(1)$ on $\Xi.$ Then, using \eqref{partialtG} again, we get \eqref{g-whg}.
	\end{proof}
	Note that the bound \eqref{g-whg} is purely deterministic on $\Xi$, so we do not lose any probability in this claim. Moreover, such a small error $n^{-8}$ is negligible for our proof.  
		
    In the following proof, we will use the regularized resolvent \smash{$\wh G(z)$} with $z= \theta_l +\ii n^{-4}$, and prove the CLT for \smash{$\wh{\cal Q}(z)$} with $G(\theta_l)$ replaced by \smash{$\wh G(z)$}. The argument in the proof of Claim \ref{removehat} then allows us to show that $ \cal Q(\theta_l)$ satisfies the same asymptotic distribution. In the proof, it is helpful to keep in mind that the bound (\ref{op G}) always holds with $\eta=n^{-4}$, and hence Lemma \ref{lem_stodomin} (iii) can be applied without worry. To ease the notation, we also introduce the following notion of generalized entries.
    
    \begin{definition}[Generalized entries]
    	For $\mathbf v,\mathbf w \in \mathbb C^{\mathcal I}$, $\fa\in \mathcal I$ and an $\mathcal I\times \mathcal I$ matrix $\cal A$, we shall denote
    	\begin{equation}
    		\cal A_{\mathbf{vw}}:=\langle \mathbf v,\cal A\mathbf w\rangle, \quad  \cal A_{\mathbf{v}\fa}:=\langle \mathbf v,\cal A\mathbf e_\fa\rangle, \quad \cal A_{\fa\mathbf{w}}:=\langle \mathbf e_\fa,\cal A\mathbf w\rangle,
    	\end{equation}
    	where $\mathbf e_\fa$ is the standard unit vector along the $\fa$-th coordinate axis. 
    \end{definition}

For $1\le a \le 4r$, we denote the $a$-th column vector of $\mathscr V$ by $\bv_{a}$. 
With the Cram{\'e}r-Wold device, it suffices to prove that
$$\wh Q_\Lambda:=\sqrt{n}\sum_{ 1\le a\le b \le 4r} \lambda_{ab} \wh Q_{ab} = \sqrt{n} \sum_{a\le b} \lambda_{ab} (\wh G- \Pi)_{\bv_a \bv_b} $$ 
is asymptotically Gaussian for any fixed vector of parameters denoted by $\Lambda:=(\lambda_{ab})_{a\le b}$. By \eqref{aniso_outstrong}, we have the rough bound $|\wh Q_\Lambda| \prec 1$. For our purpose, it suffices to show that the moments of $\wh Q_{\Lambda}$ match those of a centered Gaussian random variable asymptotically. This follows immediately from the following claims: (i)  the mean of $\wh Q_{\Lambda}$ satisfies 
\be\label{1Qlambda}
\mathbb E\wh Q_{\Lambda}(z) =\oo(1), \quad \text{with}\quad z= \theta_l +\ii n^{-4},
\ee
and (ii) for any fixed integer $k\ge 2$, we have that
\be\label{kQlambda}
\mathbb E\wh Q_{\Lambda}^k(z) = (k-1)s_{\Lambda}^2 \mathbb E\wh Q_{\Lambda}^{k-2}(z) + \oo(1),\quad \text{with}\quad z= \theta_l +\ii n^{-4},
\ee
for a deterministic parameter $s_{\Lambda}^2$ as a function of $\Lambda$. Moreover, the covariance of $\wh {\cal Q}$ is also determined by $s_{\Lambda}^2$. 

	As described in Section \ref{sec_overview}, our main tool for the proof of \eqref{1Qlambda} and \eqref{kQlambda} is Gaussian integration by parts. Using the identity $\wh H\wh G=I$ and equation \eqref{useful}, we get that
	\be\label{simpleid} \begin{split} \wh G-\Pi &= \Pi \left( \Pi^{-1}  - \wh H\right)\wh G  \\
	&= \Pi \begin{bmatrix} - (m_{3c}+zn^{-10}) I_p & 0 &-X &0 \\ 0 & - (m_{4c}+zn^{-10})I_q & 0 & -Y \\ -X^{\top} & 0 & -m_{1c}I_n & 0 \\ 0 & -Y^{\top} & 0 & -m_{2c}I_n \end{bmatrix}\wh G .
	\end{split}\ee
	We first prove \eqref{1Qlambda}. With \eqref{simpleid}, we can write that
	\begin{align}
		&\mathbb E \wh Q_\Lambda:=\sqrt{n} \sum_{ a\le b} \lambda_{ab} \E \wh Q_{ab} \nonumber\\
		 =&\ \sqrt{n} \sum_{  a\le b}  \lambda_{ab}\E \left[ \begin{pmatrix} - m_{3c} I_p & 0 &0 &0 \\ 0 & - m_{4c}I_q & 0 & 0 \\ 0 & 0 & -m_{1c}I_n & 0 \\ 0 & 0 & 0 & -m_{2c}I_n \end{pmatrix}\wh G\right]_{\bw_a \bv_b} \nonumber\\
		&  -\sqrt{n} \sum_{  a\le b}  \lambda_{ab}\E \left[ \begin{pmatrix} 0 & \begin{pmatrix}  X & 0\\ 0 &  Y\end{pmatrix}\\ \begin{pmatrix} X^{\top} & 0\\ 0 &  Y^{\top}\end{pmatrix}  &  0\end{pmatrix} \wh  G \right]_{\bw_a \bv_b} +\OO(n^{-9}),\label{HGW}
	\end{align}
	where we have abbreviated $\bw_a:= \Pi\bv_a$. For the sum in line \eqref{HGW}, we expand it as
	\begin{align}
		& \E \left[ \begin{pmatrix} 0 & \begin{pmatrix}  X & 0\\ 0 &  Y\end{pmatrix}\\ \begin{pmatrix} X^{\top} & 0\\ 0 &  Y^{\top}\end{pmatrix}  &  0\end{pmatrix}  \wh G \right]_{\bw_a \bv_b} \nonumber\\
		&= - \sqrt{n}\mathbb E \sum_{i\in \cal I_1,\mu \in \cal I_3 }X_{i\mu}\left[ \bw_a(i) \wh G_{\mu \bv_b }+ \bw_a(\mu) \wh G_{i  \bv_b} \right]   \nonumber \\
		&\quad -  \sqrt{n}\mathbb E \sum_{j\in \cal I_2,\nu \in \cal I_4 }Y_{j\nu}\left[\bw_a(j) \wh G_{\nu \bv_b}+\bw_a(\nu) \wh  G_{j  \bv_b}\right] \nonumber \\
		&= n^{-1/2}\mathbb E \sum_{i\in \cal I_1,\mu \in \cal I_3 }\bw_a(i) \left[ \wh G_{\mu\mu} \wh G_{i\bv_b} + \wh G_{\mu i} \wh G_{\mu\bv_b} \right]  \nonumber\\
		&\quad + n^{-1/2}\mathbb E \sum_{ i \in \cal I_1 ,\mu\in \cal I_3}\bw_a(\mu)  \left[ \wh  G_{ii}\wh  G_{\mu \bv_b} + \wh G_{i\mu }\wh  G_{i \bv_b} \right]  \nonumber \\
		& \quad + n^{-1/2}\mathbb E \sum_{j\in \cal I_2,\nu \in \cal I_4 }\bw_a(j) \left[\wh  G_{\nu\nu} \wh G_{j\bv_b} +\wh  G_{\nu j}\wh  G_{\nu\bv_b} \right]  \nonumber\\
		&\quad + n^{-1/2}\mathbb E \sum_{j\in \cal I_2 ,\nu \in \cal I_4}\bw_a(\nu) \left[\wh  G_{jj}\wh  G_{\nu \bv_b} +\wh  G_{ j \nu}\wh  G_{j \bv_b} \right] ,\label{EQ1}
	\end{align}
	where in the second step we used Gaussian integration by parts with respect to $X_{i\mu}$ and $Y_{j\nu}$, 
	$$\mathbb EX_{i\mu}f(X_{i\mu})=n^{-1}\mathbb E f'(X_{i\mu}),\quad \mathbb EY_{j\nu}f(Y_{j\nu})=n^{-1}\mathbb E f'(Y_{j\nu}),$$
	and the identities 
	\be\label{partialG}\frac{\partial \wh G_{\bu \bv}}{\partial X_{i\mu}} = - \wh G_{\bu i}\wh G_{\mu \bv} - \wh G_{\bu \mu}\wh G_{i \bv}, \quad \frac{\partial \wh G_{\bu \bv}}{\partial Y_{j\nu}} = - \wh G_{\bu j}\wh G_{\nu \bv} - \wh G_{\bu \nu}\wh G_{j \bv},\ee
	for any vectors $\bu, \bv\in \C^{\cal I}$. With the notations in \eqref{defmal}, we can rewrite \eqref{EQ1} as
	\begin{align}
		  \eqref{EQ1} =&\ \sqrt{n}\E\left[ \begin{pmatrix} \wh m_3 I_p & 0 &0 &0 \\ 0 & \wh m_4 I_q & 0 & 0 \\ 0 & 0 & \wh m_1 I_n & 0 \\ 0 & 0 & 0 & \wh m_2 I_n \end{pmatrix}\wh G \right]_{\bw_a \bv_b} \nonumber\\
		  &+n^{-1/2}\mathbb E\left[  \langle \bw_a ,  J_1\wh G J_3\wh  G \bv_b\rangle+\langle \bw_a ,J_3 \wh  GJ_1 \wh G \bv_b\rangle\right] \nonumber\\
		&  + n^{-1/2}\mathbb E\left[\langle \bw_a,J_2 \wh  GJ_4 \wh G \bv_b\rangle+ \langle \bw_a,J_4\wh  GJ_2\wh G \bv_b\rangle\right] ,\label{EQ12}
	\end{align}
	where recall that $J_\al$ is defined in \eqref{defJal}. 
	We claim that 
	\be\label{claimmal}
	\max_{\al=1}^4|\wh m_\al (z)- m_{\al c}(z)|\prec n^{-2/3} ,
	\ee
	whose proof will be postponed until we complete the proof of Lemma \ref{Gauss lemma}. Moreover, $\wh GJ_\al\wh  G$, $\al=1,2,3,4$, satisfy the anisotropic local laws in Theorem \ref{lemmaGHF} below, which implies that for any deterministic unit vectors $\bu,\bv\in \C^{\cal I}$,
	\be\label{claimGal}
	\left| \langle \bu, \wh  G J_\al \wh G \bv \rangle\right|=\OO_\prec(1),\quad  \al = 1,2,3,4.
	\ee
	Now, plugging \eqref{EQ12} into \eqref{HGW} and using \eqref{claimmal} and \eqref{claimGal}, we obtain that 
	\be\label{k=1case} \mathbb E \wh Q_\Lambda=\OO_\prec(n^{-1/6 }),\ee
	which implies \eqref{1Qlambda}.
	
	It remains to prove \eqref{kQlambda}. With \eqref{simpleid}, we expand $\mathbb E\wh Q_{\Lambda}^k$ as 
	\begin{align}
		&\mathbb E\wh Q_{\Lambda}^k \nonumber\\
		&=  \mathbb E  \sqrt{n} \sum_{ a\le b}  \lambda_{ab}\E\left[\Pi   \begin{pmatrix} - m_{3c} I_p & 0 & - X &0 \\ 0 & - m_{4c}I_q  & 0 & -Y \\ -X^{\top} & 0 & -m_{1c}I_n & 0 \\ 0 & -Y^{\top} & 0 & -m_{2c}I_n \end{pmatrix}\wh G \right]_{\bv_a \bv_b} \wh Q_\Lambda^{k-1}  \nonumber\\
		& = \sqrt{n} \sum_{  a\le b}  \lambda_{ab}\E \left[ \begin{pmatrix} - m_{3c} I_p & 0 &0 &0 \\ 0 & - m_{4c}I_q & 0 & 0 \\ 0 & 0 & -m_{1c}I_n & 0 \\ 0 & 0 & 0 & -m_{2c}I_n \end{pmatrix}\wh G\right]_{\bw_a \bv_b} \wh Q_\Lambda^{k-1}    \label{IBPQk0}\\
		& - \sqrt{n}\mathbb E \sum_{  a\le b}  \lambda_{ab}\sum_{i\in \cal I_1,\mu \in \cal I_3 }\bw_a(i) X_{i\mu}\wh G_{\mu \bv_b }\wh Q_\Lambda^{k-1}  \label{IBPQk1.0}\\
		&-  \sqrt{n}\mathbb E \sum_{  a\le b}  \lambda_{ab}\sum_{j\in \cal I_2,\nu \in \cal I_4 }\bw_a(j) Y_{j\nu}\wh G_{\nu \bv_b}\wh Q_\Lambda^{k-1} \label{IBPQk1}\\
		&  -  \sqrt{n}\mathbb E \sum_{  a\le b}  \lambda_{ab}\sum_{i \in \cal I_1 ,\mu\in \cal I_3}\bw_a(\mu) X_{i\mu}\wh G_{i  \bv_b}\wh Q_\Lambda^{k-1}  \label{IBPQk2.0}\\
		& -  \sqrt{n}\mathbb E\sum_{  a\le b}  \lambda_{ab} \sum_{j\in \cal I_2,\nu \in \cal I_4 }\bw_a(\nu) Y_{j\nu}\wh G_{j  \bv_b} \wh Q_\Lambda^{k-1} +\OO(n^{-9}). \label{IBPQk2} 
	\end{align}
	Again, we apply Gaussian integration by parts to the terms in \eqref{IBPQk1.0}--\eqref{IBPQk2}. First, as we have seen in the $k=1$ case, the terms containing $\partial_{X_{i\mu}}\wh G_{\mu \bv_b }$, $\partial_{X_{i\mu}}\wh G_{i \bv_b }$, $\partial_{Y_{j\nu}}\wh G_{\nu \bv_b}$ and $\partial_{Y_{j\nu}}\wh G_{j \bv_b}$ will cancel the first term in \eqref{IBPQk0}, leaving an error of order $\OO_\prec(n^{-1/6 })$ as in \eqref{k=1case}. Thus, we get that
	\begin{align}
		&\mathbb E\wh Q_{\Lambda}^k   \nonumber \\
		&= - n^{-1/2} \sum_{a\le b} \lambda_{ab} \mathbb E\sum_{i\in \cal I_1,\mu \in \cal I_3, }\bw_a(i) \wh G_{\mu  \bv_b}\frac{\partial \wh Q_\Lambda^{k-1}}{\partial X_{i\mu}}   - n^{-1/2} \sum_{a\le b} \lambda_{ab} \mathbb E\sum_{j\in \cal I_2,\nu \in \cal I_4}\bw_a(j) \wh G_{\nu \bv_b} \frac{\partial \wh Q_\Lambda^{k-1}}{\partial Y_{j\nu}} \nonumber \\
		&  - n^{-1/2} \sum_{a\le b} \lambda_{ab} \mathbb E \sum_{ i \in \cal I_1 ,\mu\in \cal I_3}\bw_a(\mu)  \wh G_{i \bv_b}\frac{\partial \wh Q_\Lambda^{k-1}}{\partial X_{i\mu}}    - n^{-1/2} \sum_{a\le b} \lambda_{ab} \mathbb E \sum_{j\in \cal I_2 ,\nu \in \cal I_4}\bw_a(\nu)  \wh G_{j \bv_b}\frac{\partial \wh Q_\Lambda^{k-1}}{\partial Y_{j\nu}}  \nonumber\\
		&+  \OO_\prec(n^{-1/6 }) \nonumber\\
		&=   -(k-1)  \sum_{a\le b,a'\le b'} \lambda_{ab}  \lambda_{a'b'}  \mathbb E\sum_{i\in \cal I_1,\mu \in \cal I_3 }\bw_a(i) \wh G_{\mu  \bv_b} \frac{\partial \wh G_{\bv_{a'} \bv_{b'}}}{\partial X_{i\mu}}\wh Q_\Lambda^{k-2} \label{Q1line}  \\
		&- (k-1)  \sum_{a\le b,a'\le b'} \lambda_{ab}  \lambda_{a'b'}  \mathbb E\sum_{j\in \cal I_2,\nu \in \cal I_4 }\bw_a(j) \wh G_{\nu \bv_b} \frac{\partial \wh G_{\bv_{a'} \bv_{b'}}}{\partial Y_{j\nu}}\wh Q_\Lambda^{k-2} \label{Q2line}\\
		&  - (k-1)  \sum_{a\le b,a'\le b'} \lambda_{ab}  \lambda_{a'b'} \mathbb E \sum_{ i \in \cal I_1,\mu\in \cal I_3 }\bw_a(\mu)\wh   G_{i  \bv_b}  \frac{\partial \wh G_{\bv_{a'} \bv_{b'}}}{\partial X_{i\mu}}\wh Q_\Lambda^{k-2} \label{Q3line}\\
		&  - (k-1)  \sum_{a\le b,a'\le b'} \lambda_{ab}  \lambda_{a'b'}  \mathbb E \sum_{j\in \cal I_2,\nu \in \cal I_4 }\bw_a(\nu) \wh  G_{j  \bv_b} \frac{\partial \wh G_{\bv_{a'} \bv_{b'}}}{\partial Y_{j\nu}}\wh Q_\Lambda^{k-2} + \OO_\prec(n^{-1/6 }).\label{Q4line}
	\end{align}

	To calculate the terms \eqref{Q1line}--\eqref{Q4line}, we need to use the anisotropic local law of $GJ_\al G$, $\al=1,2,3,4$. We first define the deterministic matrix limits of $GJ_\al G$: 
	\be \label{defn_pi2}
	\Gamma^{(\al)}(z) := \begin{bmatrix} \begin{pmatrix} \gamma_1^{(\al)}(z)I_p & 0\\ 0 & \gamma_2^{(\al)}(z) I_q\end{pmatrix} & 0 \\ 0  & \begin{pmatrix}  \gamma_3^{(\al)}(z) I_n  & h_\al (z)I_n\\  h_\al (z)I_n &  \gamma_4^{(\al)}(z)  I_n\end{pmatrix}\end{bmatrix} ,\quad \al=1,2,3,4,\ee
	where the $\gamma$ functions are defined by 
	\begin{align}
		& \gamma_1^{(1)} :=\frac{ (1-c_1)^{-1} f_c^2 }{m_{3c}^{2}(f_c^{2} - t_c^{2})}  , \quad \gamma_2^{(1)} :=\frac{c_2^{-1} t_c^2}{h^2 (f_c^2 - t_c^2)} ,\quad \gamma_3^{(1)} :=\frac{(1-c_1)^{-1}  f_c^2 }{f_c^{2} - t_c^{2}} -1 , \nonumber \\
		&  \gamma_4^{(1)}: = \frac{c_2^{-1}m_{4c}^2 t_c^2 }{h^2 (f_c^2 - t_c^2)}, \quad  \gamma_1^{(2)} :=  \frac{c_1^{-1} t_c^2}{h^2 (f_c^2 - t_c^2)} , \quad \gamma_2^{(2)} :=\frac{ (1-c_2)^{-1} f_c^2 } {m_{4c}^{2} (f_c^{2} - t_c^{2})}  , \nonumber \\
		& \gamma_3^{(2)} := \frac{c_1^{-1}m_{3c}^2t_c^2  }{h^2 (f_c^2 - t_c^2)} , \quad \gamma_4^{(2)}: = \frac{(1-c_2)^{-1} f_c^2 }{f_c^{2} - t_c^{2}}  -1 ,\quad \gamma_1^{(3)} := c_1^{-1}\gamma_3^{(1)}  , \nonumber\\
		& \gamma_2^{(3)} := c_2^{-1}\gamma_3^{(2)} ,\quad \gamma_3^{(3)}:  = c_1^{-1}m_{3c}^2 \gamma_3^{(1)} , \quad\gamma_4^{(3)} :=\frac{c_1^{-1}c_2^{-1}h^2 t_c^2 f_c^2 }{ f_c^2 - t_c^2},\nonumber\\
		& \gamma_1^{(4)} := c_1^{-1}\gamma_4^{(1)}  ,\quad  \gamma_2^{(4)} := c_2^{-1}\gamma_4^{(2)} , \quad \gamma_3^{(4)} := \gamma_4^{(3)},\quad \gamma_4^{(4)}:=c_2^{-1}m_{4c}^2 \gamma_4^{(2)}. \label{defgamma} 
	\end{align}
	On the other hand, the functions $h_\al$ are defined by
	$$	  h_\al(z):= z^{1/2}h^2(z)\left\{ c_1\gamma_1^{(\al)}(z) \left[ 1+ (1-z)m_{2c}(z)\right]+ c_2\gamma_2^{(\al)}(z)\left[ 1+ (1-z) m_{1c}(z)\right]  \right\}.$$
	Here, we recall that $t_c$ is defined in \eqref{tc}, $m_{\al c}$, $\al=1,2,3,4,$ are defined in \eqref{m1c}--\eqref{m4c}, $h$ is defined in \eqref{hz}, and $f_c$ is defined in \eqref{fcz}. 
	
	\begin{theorem}\label{lemmaGHF}
	Suppose Assumption \ref{main_assm} holds. For any deterministic unit vectors $\mathbf u, \mathbf v \in \mathbb C^{\mathcal I}$, we have that
	\be\label{derivG}
	\langle \bu, G(\theta_l)J_\al G(\theta_l)\bv\rangle -  \langle\bu ,\Gamma^{(\al)} (\theta_l)\bv\rangle \prec n^{-1/2} . 
	\ee
	\end{theorem}
	We will prove Theorem \ref{lemmaGHF} in Section \ref{appd GJG}.  Again, by the argument in the proof of Claim \ref{removehat}, \eqref{derivG} also holds for $\wh G(z)J_\al \wh G(z)$ with $z= \theta_l + \ii n^{-4}$. Now, we use this estimate to calculate \eqref{Q1line}--\eqref{Q4line} term by term. First, for \eqref{Q1line}, using \eqref{partialG} we get that 
	\be
	\begin{split}
		-\mathbb E\sum_{i\in \cal I_1,\mu \in \cal I_3 }\bw_a(i)\wh  G_{\mu  \bv_b} \frac{\partial \wh G_{\bv_{a'} \bv_{b'}}}{\partial X_{i\mu}}\wh Q_\Lambda^{k-2} &= \mathbb E (\wh GJ_3\wh G)_{ \bv_{b'} \bv_b}\langle \bv_a , \Pi J_1 \wh G \bv_{a'}\rangle   \wh Q_\Lambda^{k-2} \label{Q1line2}\\
		& + \mathbb E (\wh GJ_3\wh G)_{ \bv_{a'} \bv_b}\langle \bv_a , \Pi J_1 \wh G \bv_{b'}\rangle   \wh Q_\Lambda^{k-2}.
	\end{split}
	\ee
	Now, using the local law \eqref{aniso_outstrong}, \eqref{V12} and the first equation in \eqref{selfm12}, we get that
	\be\label{Gausslocal1}\begin{split}\langle \bv_a , \Pi J_1 \wh G \bv_{a'}\rangle &=c_1\left(c_1^{-1}m_{1c}\right)^2 \delta_{aa'} \mathbf 1_{1\le a \le r}+\OO_\prec(n^{-1/2}) \\
	& = c_1m_{3c}^{-2} \delta_{aa'}\mathbf 1_{1\le a \le r}+\OO_\prec(n^{-1/2})  . \end{split}\ee
	Moreover, using \eqref{V12}, \eqref{V34} and the local law for $\wh GJ_3\wh G$ in Theorem \ref{lemmaGHF}, we get that
	\be\label{Gausslocal2}(\wh GJ_3\wh G)_{ \bv_{b'} \bv_b} = c_{\al(b)}\gamma^{(3)}_{\al(b)}\delta_{bb'}+\OO_\prec (n^{-1/2+2\tau_0 })  , \ee
	where we used the notation
	$$\al(b):= k   \ \ \ \text{ if }\ (k-1)r+1 \le b \le k r , \quad k=1,2,3,4,$$
	and let $c_{k}\equiv 1$ for $k=3,4$. Plugging \eqref{Gausslocal1} and \eqref{Gausslocal2} into \eqref{Q1line2}, we get that 
	\be
	\eqref{Q1line}=(k-1)  \sum_{1\le a \le r, a\le b} c_1c_{\al(b)}\frac{\lambda_{ab}^2}{m_{3c}^{2} } \gamma^{(3)}_{\al(b)} (1+\delta_{ab}) \mathbb E \wh Q_\Lambda^{k-2} +\OO_\prec(n^{-1/2+2\tau_0}).\label{Q1line3}
	\ee
	Similarly, we can get that
	\be
	\eqref{Q2line}=(k-1)  \sum_{r+1\le a \le 2r, a\le b} c_2c_{\al(b)}\frac{\lambda_{ab}^2}{m_{4c}^{2} } \gamma^{(4)}_{\al(b)} (1+\delta_{ab}) \mathbb E \wh Q_\Lambda^{k-2} +\OO_\prec(n^{-1/2+2\tau_0}).\label{Q2line3}
	\ee
	For \eqref{Q3line}, we have that
	\be
	\begin{split}
		&-\mathbb E \sum_{ i \in \cal I_1,\mu\in \cal I_3 }\bw_a(\mu) \wh  G_{i  \bv_b}  \frac{\partial \wh G_{\bv_{a'} \bv_{b'}}}{\partial X_{i\mu}}\wh Q_\Lambda^{k-2} 	\\
		&= \mathbb E\sum_{i\in \cal I_1,\mu \in \cal I_3 }(\wh GJ_1\wh G)_{ \bv_{b'} \bv_b}\langle \bv_a , \Pi J_3 \wh G \bv_{a'}\rangle  \wh  Q_\Lambda^{k-2} \label{Q3line2}\\
		& + \mathbb E\sum_{i\in \cal I_1,\mu \in \cal I_3 }(\wh GJ_1\wh G)_{ \bv_{a'} \bv_b}\langle \bv_a , \Pi J_3 \wh G \bv_{b'}\rangle   Q_\Lambda^{k-2}. 
	\end{split}
	\ee
	Using \eqref{aniso_outstrong} and \eqref{V34}, we get that
	\be\label{Gausslocal3}\langle \bv_a , \Pi J_3 \wh G \bv_{a'}\rangle = m_{3c}^2 \delta_{aa'} \mathbf 1_{2r+1\le a \le 3r}+ h^2 \delta_{aa'} \mathbf 1_{3r+1\le a \le 4r}+\OO_\prec(n^{-1/2+2\tau_0}) .\ee
	Using the local law for $\wh GJ_1\wh G$ in Theorem \ref{lemmaGHF} and \eqref{V34}, we get that
	\be\label{Gausslocal4}
	(\wh GJ_1\wh G)_{ \bv_{b'} \bv_b} = \gamma^{(1)}_{\al(b)}\delta_{bb'}+\OO_\prec(n^{-1/2+2\tau_0}), \quad \text{for }\ \ \al(b)=3,4.
	\ee
	Plugging \eqref{Gausslocal3} and \eqref{Gausslocal4} into \eqref{Q3line2} gives that
	\be
	\begin{split}
		\eqref{Q3line}&=(k-1)  \sum_{2r+1\le a \le 3r, a\le b}  {\lambda_{ab}^2}{m_{3c}^{2} } \gamma^{(1)}_{\al(b)} (1+\delta_{ab}) \mathbb E \wh Q_\Lambda^{k-2}  \label{Q3line3} \\
		&+(k-1)  \sum_{3r+1\le a \le 4r, a\le b}  {\lambda_{ab}^2}{h^{2} } \gamma^{(1)}_{\al(b)} (1+\delta_{ab}) \mathbb E\wh  Q_\Lambda^{k-2} +\OO_\prec(n^{-1/2+2\tau_0}).
	\end{split}
	\ee
	Similarly, we can get that
	\be
	\begin{split}
		\eqref{Q4line}&=(k-1)  \sum_{2r+1\le a \le 3r, a\le b}  {\lambda_{ab}^2}{h^2} \gamma^{(2)}_{\al(b)} (1+\delta_{ab}) \mathbb E \wh Q_\Lambda^{k-2} \label{Q4line3} \\
		&+(k-1)  \sum_{3r+1\le a \le 4r, a\le b}  {\lambda_{ab}^2}{m_{4c}^{2} } \gamma^{(2)}_{\al(b)} (1+\delta_{ab}) \mathbb E\wh  Q_\Lambda^{k-2} +\OO_\prec(n^{-1/2+\tau_0}).
	\end{split}
	\ee
	
	Combining \eqref{Q1line3}, \eqref{Q2line3}, \eqref{Q3line3} and \eqref{Q4line3}, we obtain that
	\begin{align*}
		&  \mathbb E\wh Q_{\Lambda}^k  = (k-1)s_{\Lambda}^2 \mathbb E\wh Q_{\Lambda}^{k-2} + \OO_\prec(n^{-1/6}),
	\end{align*}
	where $s_{\Lambda}^2$ is a function of $\Lambda$ defined by
	\begin{align*}
		s_{\Lambda}^2  :=&\ \sum_{1\le a \le r, a\le b} c_1c_{\al(b)}\frac{\lambda_{ab}^2}{m_{3c}^{2} } \gamma^{(3)}_{\al(b)} (1+\delta_{ab})  +  \sum_{r+1\le a \le 2r, a\le b} c_2c_{\al(b)}\frac{\lambda_{ab}^2}{m_{4c}^{2} } \gamma^{(4)}_{\al(b)} (1+\delta_{ab}) \\
		&+  \sum_{2r+1\le a \le 3r, a\le b}  {\lambda_{ab}^2} \left({m_{3c}^{2} }\gamma^{(1)}_{\al(b)}+{h^{2} }\gamma^{(2)}_{\al(b)}\right) (1+\delta_{ab}) \\
		&+  \sum_{3r+1\le a \le 4r, a\le b}  {\lambda_{ab}^2} \left(h^2\gamma^{(1)}_{\al(b)}+{m_{4c}^{2} }\gamma^{(2)}_{\al(b)}\right) (1+\delta_{ab}).
	\end{align*}
	This concludes \eqref{kQlambda}. Combining \eqref{1Qlambda} and \eqref{kQlambda}, we have shown that $\wh {\cal Q}_\Lambda(z)$ is asymptotically Gaussian with zero mean, which indicates that $\wh {\cal Q}(z)$ converges weakly to a centered Gaussian matrix by the Cram{\'e}r-Wold device. Then, the argument in the proof of Claim \ref{removehat} shows that $\cal Q(\theta_l)$ converges to the same limit. Using the definitions of $\gamma_\beta^{(\al)}$, $\al,\beta=1,2,3,4$, in \eqref{defgamma}, we obtain from $s_{\Lambda}^2$ that  
		\begin{align}\label{Q5line}
			\sqrt{n}\cal Q  \to \begin{pmatrix} b_{11}g_{11} & b_{12}g_{12} & b_{13}g_{13} & b_{14}g_{14} \\  
				b_{21}g_{21} & b_{22}g_{22} & b_{23}g_{23} & b_{24}g_{24} \\  
				b_{31}g_{31} & b_{32}g_{32} & b_{33}g_{33} & b_{34}g_{34} \\  
				b_{41}g_{41} & b_{42}g_{42} & b_{43}g_{43} & b_{44}g_{44} \\  
			\end{pmatrix},
		\end{align}
		where $g_{\al\beta}$ are Gaussian matrices as defined in Lemma \ref{Gauss lemma}, and through direct calculations, we can check that $b_{\al\beta}$ are given by
		\be\label{defQb}
		\begin{split}
			& b_{11}=a_{33}, \quad b_{12}=b_{21}= a_{34}, \quad b_{13}=b_{31}= a_{13}, \quad b_{14}=b_{41}=a_{23},\quad b_{22}=a_{44},\\
			&  b_{23}=b_{32}= a_{14}, \quad  b_{24}=b_{42}=a_{24}, \quad  b_{33}= a_{11} ,\quad  b_{34}=b_{43}=a_{12}  ,\quad b_{44}=a_{22}.
		\end{split}
		\ee
		In the above calculation, we also used that for $z= \theta_l + \ii n^{-4}$,
		$$ f_c(z)= \frac{m_{3c}(z)m_{4c}(z)}{h^2(z)} = t_l +\OO(n^{-4}) .$$
		Finally, combining \eqref{Q5line} with \eqref{permu}, we can obtain the asymptotic distribution in \eqref{eq_FGF}, upon renaming the matrices $g_{\al\beta}$ and the coefficients $b_{\al\beta}$. This concludes Lemma \ref{Gauss lemma}. 
		
		\vspace{5pt}
		
		Before the end of this section, we give the proof of \eqref{claimmal}. 
		
		\begin{proof}[Proof of \eqref{claimmal}]
		By the proof of Claim \ref{removehat}, it suffices to prove the estimate for $|m_\al(z) -m_{\al c}(z) |$ for $z=\theta_l + \ii n^{-4}$. In the following proof, we denote $z_0:=\theta_l + \ii \eta_0$ with $\eta_0=n^{-2/3}$. By the averaged local law \eqref{aver_out0}, we have 
			\begin{equation}\label{aniso_outstrong2}
				\left| m_{\al}(z_0) - m_{\al c}(z_0) \right|  \prec n^{-2/3}, \quad \al=1,2,3,4,
			\end{equation}
			where we also used that $\kappa =|\theta_l -\lambda_+| \sim 1$ due to \eqref{tlc}. Thus, to show \eqref{claimmal}, it suffices to prove that
			\begin{align}
			\left|m_{\al c} (z)-m_{\al c}(z_0)\right| &\prec n^{-2/3} , \label{prof_m} \\ 
			\left|m_{\al } (z)-m_{\al }(z_0)\right| &\prec n^{-2/3}. \label{prof_G}
			\end{align}
			The estimate \eqref{prof_m} follows directly from the definitions in \eqref{m1c}--\eqref{m4c}. 			
			We still need to prove \eqref{prof_G}. It follows from the spectral decomposition of the resolvent, which we introduce next.

			First, recalling the notations in \eqref{def Sxy}, we define  
			\be\label{def calH}\cal H:=S_{xx}^{-1/2}S_{xy}S_{yy}^{-1/2},\ee 
			and the resolvent
			$$ R(z):=\begin{pmatrix} R_1  & - z^{-1/2} R_1\cal H  \\ - z^{-1/2} \cal H^{\top} R_1  &  R_2 \end{pmatrix},$$
			where the two blocks $R_1$ and $R_2$ are defined as
			\be\label{Rxy}
			\begin{split}
				& R_1(z):=\left(\cal C_{XY}-z\right)^{-1}=\left(\cal H\cal H^{\top}-z\right)^{-1}, \quad R_2(z):=\left(\cal C_{YX}-z\right)^{-1}=\left(\cal H^{\top}\cal H-z\right)^{-1}. 
			\end{split}
			\ee
			By Theorem 2.10 of \cite{isotropic}, we have the following bounds on the extreme eigenvalues of $S_{xx}$ and $S_{yy}$: 
			\be\label{op rough1add} (1-\sqrt{c_1})^2 -  \e \le  \lambda_p(S_{xx})   \le  \lambda_1(S_{xx}) \le (1+\sqrt{c_1})^2 + \e ,
			\ee
			\be\label{op rough2add} (1-\sqrt{c_2})^2 -  \e \le  \lambda_q(S_{yy})  \le  \lambda_1(S_{yy}) \le (1+\sqrt{c_2})^2 + \e .
			\ee
			
			Next, consider a singular value decomposition of $\cal H$,
			\be\label{SVDH}\cal H = \sum_{k = 1}^{q} \sqrt {\lambda_k} \xi_k \zeta _{k}^{\top}, \ee
			where $\lambda_k$'s are the eigenvalues of the null SCC matrix $\cal C_{XY}$, and $\xi_k$'s and $\zeta_k$'s are respectively the left and right singular vectors. Then, the singular value decomposition  $R(z)$ is given by 
			\begin{equation}
				\begin{split}
					R\left( z \right) & = \sum\limits_{k = 1}^q \frac{1}{\lambda_k-z}\left( {\begin{array}{*{20}c}
							{{\xi _k \xi _k^{\top}  }} & {-z^{-1/2}\sqrt {\lambda _k } \xi _k \zeta _{ k}^{\top}}  \\
							{-z^{-1/2} \sqrt {\lambda _k } \zeta _{k} \xi _k^{\top}  } & {\zeta _k\zeta _k^{\top} }  \\
					\end{array}} \right) \\
					& - \frac1z \left( {\begin{array}{*{20}c}
							{\sum_{k=q+1}^p{\xi _k \xi _k^{\top}  }} & 0  \\
							{0 } & {0}  \\
					\end{array}} \right). \label{spectral1}
				\end{split}
			\end{equation}
			We denote the $(\cal I_1\cup \cal I_2)\times (\cal I_1\cup \cal I_2)$ block of $ G(z)$ by $ \cal G_L(z)$, the $(\cal I_1\cup \cal I_2)\times (\cal I_3\cup \cal I_4)$ block by $ \cal G_{LR}(z)$, the $(\cal I_3\cup \cal I_4)\times (\cal I_1\cup \cal I_2)$ block by $ \cal G_{RL}(z)$, and the $(\cal I_3\cup \cal I_4)\times (\cal I_3\cup \cal I_4)$ block by $ \cal G_R(z)$. 
			Using the Schur complement formula, we can check that
			\be\label{GL1}
			\begin{split}
				\cal G_L =\begin{pmatrix} S_{xx}^{-1/2} & 0 \\ 0 & S_{yy}^{-1/2} \end{pmatrix}R(z) \begin{pmatrix} S_{xx}^{-1/2} & 0 \\ 0 & S_{yy}^{-1/2} \end{pmatrix}, \quad
			\end{split}
			\ee
			\be\label{GR1}
			\begin{split}
				\cal G_R &=   \begin{pmatrix}  z  I_n & z^{1/2}I_n\\ z^{1/2}I_n &  z  I_n\end{pmatrix} \\
				&+   \begin{pmatrix}  z  I_n & z^{1/2}I_n\\ z^{1/2}I_n &  z  I_n\end{pmatrix}  \begin{pmatrix} X^{\top} & 0 \\ 0 & Y^{\top} \end{pmatrix} \cal G_L \begin{pmatrix} X & 0 \\ 0 &  Y \end{pmatrix} \begin{pmatrix}  z  I_n & z^{1/2}I_n\\ z^{1/2}I_n &  z  I_n\end{pmatrix}  ,
			\end{split}
			\ee
			\be\label{GLR1}
			\begin{split}
				& {\cal G}_{LR}(z)= -\cal G_L(z) \begin{pmatrix} X & 0 \\ 0 &  Y \end{pmatrix} \begin{pmatrix}  z  I_n & z^{1/2}I_n\\ z^{1/2}I_n &  z  I_n\end{pmatrix}  , \\
				&{\cal G}_{RL}(z)= -  \begin{pmatrix}  z  I_n & z^{1/2}I_n\\ z^{1/2}I_n &  z  I_n\end{pmatrix}  \begin{pmatrix} X^{\top} & 0 \\ 0 & Y^{\top} \end{pmatrix} {\cal G}_L(z).
			\end{split}
			\ee
			
			Now, we are ready to prove \eqref{prof_G}. We only give the proof for $\al=1$, and all the other cases can be proved in exactly the same way. 
			Using the rigidity estimate \eqref{rigidity}, we get that with high probability,
			\be\label{eigen_gap0}\min_{1\le k \le q}|\lambda_k -z|\gtrsim 1,\quad z=\theta_l + \ii n^{-4}.\ee
			Then, using \eqref{defmal}, \eqref{spectral1}, \eqref{GL1}, \eqref{eigen_gap0}, and \eqref{op rough1add},	we obtain that
			\begin{align} \nonumber
				\left|m_1(z) - m_1(z_0) \right| \prec \frac{\eta_0 }n\sum_{i=1}^p \sum_{k = 1}^{p}  { \left|\left\langle \mathbf e_i,S_{xx}^{-1/2}{\xi}_k\right\rangle\right|^2}  = \frac{\eta_0}{n}\tr(S_{xx}^{-1}) \prec  \eta_0 = n^{-2/3},
			\end{align}
			where $\mathbf e_i$ is the standard unit vector along the $i$-th direction. 
		\end{proof}

\section{Proof of Theorem \ref{lemmaGHF}}\label{appd GJG}

In this section, we give the proof of Theorem \ref{lemmaGHF}.  We first record the following simple estimate, which can be verified through direct calculations using \eqref{m1c}--\eqref{m4c}. 

\begin{lemma}[Lemma 3.2 of \cite{PartIII}]\label{lem_mbehavior}
	Fix any constants $c, C>0$. If \eqref{assm20} holds, then for $z\in \C_+ \cap \{z: c\le |z| \le C\}$ and $\al=1,2,3,4$, the following estimates hold:
	\begin{equation}\label{absmc}
		\vert m_{\al c}(z) \vert \sim 1,  \quad \left|z^{-1} - (m_{1c}(z)+m_{2c}(z)) + (z-1)m_{1c}(z)m_{2c}(z) \right|\sim 1.
	\end{equation} 
\end{lemma}

\subsection{Resolvents and limiting laws}

We begin the proof by introducing some new resolvents. With $H(\theta_l)$ in \eqref{linearize_block}, we define the following \emph{generalized resolvent}
\begin{equation}\label{linearize_blockw0}
	\cal R(\bw) := \left[H (X,Y, \theta_l)-   \begin{pmatrix}  w_1 I_p & 0 & 0 & 0 \\ 0 & w_2 I_q & 0 & 0 \\ 0 & 0 & w_3 I_n & 0 \\ 0 & 0 & 0 & w_4 I_n  \end{pmatrix} \right]^{-1} ,
\end{equation}
where $\bw=(w_1,w_2, w_3,w_4)\in \C_+^4$ is a new vector of spectral parameters. 
Then we have the simple identity
\be\label{derivJal}GJ_\al G=\left. \frac{\partial \cal R(\bw) }{\partial w_\al}\right|_{\bw = 0} .\ee
Hence, to obtain the local laws on $G(\theta_l)J_\al G(\theta_l)$, it suffices to study the local law $\cal R(\bw)$ for the spectral parameters $\bw$ around the origin. 

In the following proof, we only prove the local law for $GJ_1 G$, while the proofs for $GJ_\al G$ with $\al=2,3,4$ are similar. For this purpose, it suffices to use spectral parameters $\bw$ with $w_2=w_3=w_4=0$. With a slight abuse of notation, we shall prove a local law for the resolvent 
\begin{equation}\label{linearize_blockw}
	\cal R(z, z') :=  \left[H (X,Y, \theta_l)-   \begin{pmatrix}  z I_p & 0 & 0 & 0 \\ 0 & z' I_q & 0 & 0 \\ 0 & 0 & z' I_n & 0 \\ 0 & 0 & 0 & z' I_n  \end{pmatrix} \right]^{-1},\quad z,z'\in \C_+ . 
\end{equation}
Similar to \eqref{defmal}, we introduce the averaged partial traces
\be\label{defomegaal} \omega_\al(z,z') 
:= \frac{1}{n}\sum_{\fa \in \cal I_\al}  \cal R_{\fa\fa}(z,z') ,\quad \al=1,2,3,4. \ee
Since $H$ is symmetric and has real eigenvalues, we immediately obtain the following deterministic bound 
\be\label{op R}
\left\|\cal R(z,z')\right\| \le \frac{C }{\min(\im z,\im z')}. 
\ee
Most of the time we will choose $z'=0$. But, to avoid the singular behaviours of $\cal R$ on exceptional low-probability events, we sometimes will choose, say $z'=\ii n^{-4}$, so that $\left\|\cal R(z,z')\right\|=\OO(n^4)$ by \eqref{op R} and hence Lemma \ref{lem_stodomin} (iii) can be applied.

We now describe the deterministic limit of $\cal R(z,0)$. We first define the deterministic limit $(\omega_{\al c}(z))_{\al=1}^4$ of $(\omega_{\al }(z,0))_{\al=1}^4$, as the unique solution to the following system of self-consistent equations
\begin{equation}\label{selfomega}
	\begin{split}
		\frac{c_1}{\omega_{1c} }=- z -  \omega_{3c} ,\quad & \omega_{3c} =(\theta_l -1) \frac{1 + (1-\theta_l )\omega_{2c}}{[1 + (1-\theta_l )\omega_{1c} ][1 + (1-\theta_l )\omega_{2c} ]-\theta_l ^{-1}},\\
		\frac{c_2}{\omega_{2c} }=  -  \omega_{4c} , \quad  & \omega_{4c} =(\theta_l -1) \frac{1 + (1-\theta_l )\omega_{1c}}{[1 + (1-\theta_l )\omega_{1c} ][1 + (1-\theta_l )\omega_{2c} ]-\theta_l ^{-1}},
	\end{split}
\end{equation}
such that $\im \omega_{\al c}(z)>0$ whenever $z\in \C_+$. Moreover, we define the function
\begin{align}
	g_1(z):&= \frac{(\theta_l -1)\theta_l ^{-1/2}}{[1 + (1-\theta_l )\omega_{1c}(z) ][1 + (1-\theta_l )\omega_{2c}(z) ]-\theta_l ^{-1}}.
\end{align}
Then, the matrix limit of $\cal R(z,0)$ is defined by
\be \label{defn_piw}
\Gamma(z) := \begin{bmatrix} \begin{pmatrix} c_1^{-1}\omega_{1c}(z)I_p & 0\\ 0 & c_2^{-1}\omega_{2c}(z)I_q\end{pmatrix} & 0 \\ 0  & \begin{pmatrix}  \omega_{3c}(z)I_n  & g_1(z)I_n\\  g_1(z)I_n &  \omega_{4c}(z)  I_n\end{pmatrix}\end{bmatrix} .\ee
The following lemma gives the existence and uniqueness of the solution $(\omega_{\al c}(z))_{\al=1}^4$. We postpone its proof to Appendix \ref{appd sol}. 

\begin{lemma} \label{lem_mbehaviorw}
	There exist constants $c_0, C_0>0$ depending only on $c_1, c_2$ and $\delta_l$ in \eqref{tlc} such that the following statements hold. 
	If $|z|\le c_0$, then there exists a unique solution to \eqref{selfomega} under the condition
	\be\label{prior10}
	\max_{\al=1}^4 |\omega_{\al c}(z) - m_{\al c}(\theta_l)| \le c_0.
	\ee
	Moreover, the solution satisfies
	\be\label{Lipomega}
	\max_{\al=1}^4 |\omega_{\al c}(z) - m_{\al c}(\theta_l)| \le C_0 |z|.
	\ee
\end{lemma}

We also have the following stability estimate regarding the system of equations in \eqref{selfomega}, whose proof is postponed to Appendix \ref{appd sol}.

\begin{lemma} \label{lem_stabw}
	There exist constants $c_0, C_0>0$ depending only on $c_1, c_2$ and $\delta_l$ such that the self-consistent equations in \eqref{selfomega} are stable in the following sense. Suppose $|z|\le c_0$ and $\omega_{\al}: \C_+\mapsto \C_+$, $\al=1,2,3,4$, are analytic functions of $z$ such that 
	\be\label{prior1}
	\max_{\al=1}^4 |\omega_{\al}(z) - m_{\al c}(\theta_l)| \le c_0.
	\ee
	Suppose they satisfy the system of equations
	\begin{equation}\label{selfomegaerror}
		\begin{split}
			\frac{c_1}{\omega_{1} }+ z +  \omega_{3}= \mathcal E_1,\quad &\omega_{3} + (1-\theta_l ) \frac{1 + (1-\theta_l )\omega_{2}}{[1 + (1-\theta_l )\omega_{1} ][1 + (1-\theta_l )\omega_{2} ]-\theta_l ^{-1}}= \mathcal E_2,\\
			\frac{c_2}{\omega_{2} } + \omega_{4} = \mathcal E_3, \quad  &\omega_{4} +(1-\theta_l ) \frac{1 + (1-\theta_l )\omega_{1}}{[1 + (1-\theta_l )\omega_{1} ][1 + (1-\theta_l )\omega_{2} ]-\theta_l ^{-1}}= \mathcal E_4,
		\end{split}
	\end{equation}
	for some errors bounded as $ \max_{\al=1}^4 |\mathcal E_\al| \le \delta(z),$ where $\delta(z)$ is a deterministic function of $z$ satisfying that $\delta(z) \le (\log n)^{-1}.$ Then, we have 
	\begin{equation}
		\max_{\al=1}^4 \left|\omega_\al(z)-\omega_{\al c}(z)\right|\le C_0\delta(z).\label{Stability1}
	\end{equation}
\end{lemma}

The following theorem gives the anisotropic local law for $\cal R(z,0)$.

\begin{theorem} \label{thm_localw} 
	Suppose Assumption \ref{main_assm} holds. Then, for any deterministic unit vectors $\mathbf u, \mathbf v \in \mathbb C^{\mathcal I}$, the following anisotropic local law holds uniformly in $z\in \mathbf D:=\{z\in \C_+: |z| \le (\log n)^{-1}\}$: 
	\begin{equation}\label{aniso_outstrongw}
		\left| \langle \mathbf u, \cal R(z,0) \mathbf v\rangle - \langle \mathbf u, \Gamma(z)\mathbf v\rangle \right|  \prec  n^{-1/2} ,
	\end{equation}
	where $\Gamma(z)$ is defined in \eqref{defn_piw}.
\end{theorem}

The proof of this theorem will be given in Section \ref{sec GJG} below. Now, we use it to complete the proof of \eqref{derivG} when $\al=1$. 

\begin{proof}[Proof of \eqref{derivG} for $GJ_1 G$]
	Using \eqref{derivJal} and Cauchy's integral formula, we get that
	\begin{align}
		\langle \bu, G(\theta_l)J_\al G(\theta_l)\bv\rangle & = \frac{1}{2\pi \ii}\oint_{\cal C} \frac{\langle \bu, \cal R(w,0)\bv\rangle }{w^2}\dd w =  \frac{1}{2\pi \ii}\oint_{\cal C} \frac{\langle \bu, \Gamma(w)\bv\rangle }{w^2}\dd w +\OO_\prec(n^{-1/2}) \nonumber\\
		& = \langle \bu, \Gamma'(0)\bv\rangle+ \OO_\prec(n^{-1/2}),\label{apply derivlocal}
	\end{align}
	where $\cal C$ is the contour $\{w\in \C: |w| = (\log n)^{-1} \}$ and we used \eqref{aniso_outstrongw} in the second step. It remains to calculate $\Gamma'(0)$, which is reduced to calculating the derivatives $\dot m_{\al c}(\theta_l):=\omega'_\al(z=0)$, $\al=1,2,3,4$. 
	
	Using equation \eqref{selfomega} and implicit differentiation, we obtain that
	\begin{align*}
		&  c_1^{-1}\dot m_{1c}=   m_{3c}^{-2}  + \dot m_{1c} + \frac{\theta_l^{-1}}{[1 + (1-\theta_l)m_{2c} ]^2} \dot m_{2c},\quad  \dot m_{3c} = m_{3c}^2 \left(c_1^{-1}\dot m_{1c} - m_{3c}^{-2}\right), \\
		& c_2^{-1}\dot m_{2c} = \dot m_{2c} + \frac{\theta_l^{-1}}{[1 + (1-\theta_l)m_{1c}  ]^2} \dot m_{1c} , \quad \dot m_{4c} = c_2^{-1}\dot m_{2c} m_{4c}^2.
	\end{align*}
	Solving the above equations and using that (recall equation \eqref{hz})
	$$\frac{\theta_l^{-1}}{[1 + (1-\theta_l)m_{2c} ]^2}=\frac{h^2}{m^2_{3c}},\quad \frac{\theta_l^{-1}}{[1 + (1-\theta_l)m_{1c} ]^2}=\frac{h^2}{m_{4c}^2}, $$
	we get that $c_\al^{-1}\dot m_{\al c} = \gamma_{\al}^{(1)} $, $\al=1,2$, and $ \dot m_{\al c} = \gamma_{\al}^{(1)} $, $\al=3,4$, for $\gamma_{\al}^{(1)}$ defined in \eqref{defgamma}. Moreover, we can check that $g_1'(0)= h_1(z)$. Hence, we get $ \Gamma'(0)=\Gamma^{(1)}(\theta_l)$, which, together with \eqref{apply derivlocal}, concludes \eqref{derivG}.
\end{proof}

The proof of Theorem \ref{lemmaGHF} for $GJ_\al G$ with $\al=2,3,4$ is exactly the same, except that we need to use the following local law in Theorem \ref{thm_localwgen}. Recall the resolvent $\cal R(w_1,w_2, w_3,w_4)$ defined in \eqref{linearize_blockw0}. We define $(\omega_{\al c}(\bw))_{\al=1}^4$, as the unique solution to the following system of self-consistent equations
\begin{equation}\label{selft1}
	\begin{split}
		& \frac{c_1}{\omega_{1c} }=- w_1 -  \omega_{3c},\quad \frac{c_2}{\omega_{2c} }= - w_2 -  \omega_{4c}, \\
		& \omega_{3c} =(\theta_l-1) \frac{1 + (1-\theta_l)(\omega_{2c}  + w_4)}{[1 + (1-\theta_l)(\omega_{1c} +w_3)][1 + (1-\theta_l)(\omega_{2c} + w_4)]-\theta_l^{-1}},\\
		& \omega_{4c} =(\theta_l-1) \frac{1 + (1-\theta_l)(\omega_{1c}  + w_3)}{[1 + (1-\theta_l)(\omega_{1c} +w_3)][1 + (1-\theta_l)(\omega_{2c} + w_4)]-\theta_l^{-1}},
	\end{split}
\end{equation}
such that $\im \omega_{\al c}(\bw)>0$ whenever $\bw\in \C^4_+$. Define the matrix limit of $\cal R(\bw)$ as
\be \label{defn_piw^4}
\Gamma(\bw) := \begin{bmatrix} \begin{pmatrix} c_1^{-1}\omega_{1c}(\bw)I_p & 0\\ 0 & c_2^{-1}\omega_{2c}(\bw)I_q\end{pmatrix} & 0 \\ 0  & \begin{pmatrix}  \omega_{3c}(\bw)I_n  & \wt g(\bw)I_n\\  \wt g(\bw)I_n &  \omega_{4c}(\bw)  I_n\end{pmatrix}\end{bmatrix} ,\ee
where $\wt g(\bw)$ is defined by
\begin{align}\label{hzw}
	\wt g(\bw):&= \frac{(\theta_l -1)\theta_l ^{-1/2}}{[1 + (1-\theta_l )(\omega_{1c}+w_3) ][1 + (1-\theta_l )(\omega_{2c}+w_4) ]-\theta_l ^{-1}}.
\end{align}
Then, we have the following local law for $\cal R(\bw)$.

\begin{theorem} \label{thm_localwgen} 
	Suppose Assumption \ref{main_assm} holds. Fix any $ \al =1,2,3,4$. For any deterministic unit vectors $\mathbf u, \mathbf v \in \mathbb C^{\mathcal I}$, the following anisotropic local law holds uniformly in $w_\al \in  \{w_\al\in \C_+: |w_\al| \le (\log n)^{-1}\}$ if $w_\beta=0$ for $\beta\ne \al$: 
	\begin{equation}\label{aniso_outstrongwgen}
		\left| \langle \mathbf u, \cal R(\bw) \mathbf v\rangle - \langle \mathbf u, \Gamma(\bw)\mathbf v\rangle \right|  \prec  n^{-1/2}. 
	\end{equation}
\end{theorem}

This theorem can be proved in exactly the same way as Theorem \ref{thm_localw}. Moreover, with Theorem \ref{thm_localwgen}, the proof of Theorem \ref{lemmaGHF} for $GJ_\al G$, $\al=2,3,4$, is also the same as the $\al=1$ case. So we omit the details for both proofs.

\subsection{Proof of Theorem \ref{thm_localw} }\label{sec GJG}

In this section, we prove Theorem \ref{thm_localw}. We first prove the following a priori estimates on $\cal R(z,0)$. In the following proof, we will abbreviate $\cal R(z)\equiv\cal R(z,0)$.  

\begin{lemma}
	There exists a constant $C>0$ such that the following estimates hold 
	with high probability: 
	\be\label{priorim}
	\sup_{z\in \mathbf D}\|\cal R(z)\| \le C,
	\ee
	\be\label{priordiff} 
	\sup_{z\in \mathbf D}\left\| \cal R  (z) - G  (\theta_l) \right\| \le C|z|.
	\ee
\end{lemma}
\begin{proof}
	We denote the $(\cal I_1\cup \cal I_2)\times (\cal I_1\cup \cal I_2)$ block of $ \cal R $ by $ \cal R_L $, the $(\cal I_1\cup \cal I_2)\times (\cal I_3\cup \cal I_4)$ block by $ \cal R_{LR} $, the $(\cal I_3\cup \cal I_4)\times (\cal I_1\cup \cal I_2)$ block by $ \cal R_{RL} $, and the $(\cal I_3\cup \cal I_4)\times (\cal I_3\cup \cal I_4)$ block by $ \cal R_R $. Using the Schur complement formula, we obtain that
	\be\label{GL1w}
	\begin{split}
		\cal R_L & = \begin{pmatrix}  \cal R_1 & - \theta_l^{-1/2}\cal R_1S_{xy}S_{yy}^{-1} \\ - \theta_l^{-1/2}S_{yy}^{-1}S_{yx}  \cal R_1 &  \cal R_2 \end{pmatrix}, 
	\end{split}
	\ee
	where 
	$$\cal R_1 :=  \left(S_{xy}S_{yy}^{-1}S_{yx} - \theta_l S_{xx} - z\right)^{-1}, \quad \cal R_2 := - \theta_l^{-1}S_{yy}^{-1} + \theta_l^{-1} S_{yy}^{-1}S_{yx} \cal R_1 S_{xy}S_{yy}^{-1} .$$
	The other three blocks are given by
	\be\label{GR1w}
	\begin{split}
		\cal R_R &=   \begin{pmatrix}  \theta_l  I_n & \theta_l^{1/2}I_n\\ \theta_l^{1/2}I_n &  \theta_l  I_n\end{pmatrix}  \\
		&+   \begin{pmatrix}  \theta_l  I_n & \theta_l^{1/2}I_n\\ \theta_l^{1/2}I_n &  \theta_l  I_n\end{pmatrix}  \begin{pmatrix} X^{\top} & 0 \\ 0 & Y^{\top} \end{pmatrix} \cal R_L \begin{pmatrix} X & 0 \\ 0 &  Y \end{pmatrix}  \begin{pmatrix}  \theta_l  I_n & \theta_l^{1/2}I_n\\ \theta_l^{1/2}I_n &  \theta_l  I_n\end{pmatrix}  ,
	\end{split}
	\ee
	and
	\be\label{GLR1w}
	\begin{split}
		& {\cal R}_{LR} = -\cal R_L \begin{pmatrix} X & 0 \\ 0 &  Y \end{pmatrix} \begin{pmatrix}  \theta_l  I_n & \theta_l^{1/2}I_n\\ \theta_l^{1/2}I_n &  \theta_l  I_n\end{pmatrix}  , \\
		& {\cal R}_{RL}= -  \begin{pmatrix}  \theta_l  I_n & \theta_l^{1/2}I_n\\ \theta_l^{1/2}I_n &  \theta_l  I_n\end{pmatrix}   \begin{pmatrix} X^{\top} & 0 \\ 0 & Y^{\top} \end{pmatrix} {\cal R}_L.
	\end{split}
	\ee
	One can compare the above expressions with \eqref{GL1}--\eqref{GLR1}. With the estimates \eqref{op rough1add} and \eqref{op rough2add}, we see that it suffices to prove the following estimates for $\cal R_1$: 
	\be\label{opR1first}
	\sup_{z\in \mathbf D}\|\cal R_1(z)\|\lesssim 1 \quad \text{with high probability},
	\ee
	\be\label{priordiff111} 
	\sup_{z\in \mathbf D}\left\|\cal R_1  (z) -  \cal G_{(11)}  (\theta_l) \right\| \lesssim |z|\quad \text{with high probability},
	\ee
	where $\cal G_{(11)}$ is the $\cal I_1\times \cal I_1$ block of $G$ (as defined in Section \ref{sec mainthm}). With $\cal H$ in \eqref{def calH}, we can write $\cal R_1 $ as
	$$\cal R_1= S_{xx}^{-1/2} \left(\cal H \cal H^\top - \theta_l  - zS_{xx}^{-1}\right)^{-1} S_{xx}^{-1/2}. $$
	By \eqref{rigidity}, we have that with high probability, $\theta_l-\cal H \cal H^\top$ is positive definite and its smallest eigenvalue satisfies 
	$$\lambda_{p} (\theta_l-\cal H\cal H^\top) \ge  (\theta_l - \lambda_+)/2 \gtrsim 1.$$
	Combining this estimate with \eqref{op rough1add}, we obtain that with high probability,
	\be\nonumber
	\sup_{z\in \mathbf D}\|\cal R_1(z)\| \lesssim \frac1{\theta_l - \lambda_+ -\OO((\log n)^{-1})}\lesssim 1.
	\ee
	This concludes \eqref{opR1first}. With \eqref{opR1first}, we can easily conclude \eqref{priordiff111}:
	\begin{align*}\left|\langle \bu,\cal R_1  (z)\bv\rangle - \langle \bu, \cal G_{(11)}  (\theta_l)\bv\rangle \right|&=\left|\langle \bu,\left[\cal R_1  (z)-\cal R_1  (0)\right] \bv\rangle \right|   = |z| \left|\langle \bu,\cal R_1  (z) \cal R_1(0)\bv\rangle \right| \lesssim |z|,\end{align*}
	with high probability. \end{proof}

Combining \eqref{priordiff} with the local law \eqref{aniso_outstrong}, we immediately obtain the rough bound
\be\label{roughinitial}
\max_{z\in \mathbf D} \max_{\fa , \fb\in \cal I}|\cal R_{\fa \fb} (z)- \Pi_{\fa\fb}(\theta_l)|\le C(\log n)^{-1} \quad \text{with high probability. }
\ee
Then, we record some useful resolvent identities in Lemma \ref{lemm_resolvent} and Lemma \ref{lemm_resolventgroup}, which can be proved easily using the Schur complement formula. For simplicity, we abbreviate 
\be\label{simple_W}
W:= \begin{pmatrix} X & 0 \\ 0 & Y\end{pmatrix}.
\ee

\begin{lemma} 
	We have the following resolvent identities.
	\begin{itemize}
		\item[(i)] For $i\in \mathcal I_1\cup \cal I_2$, we have that 
		\begin{equation}
			\frac{1}{{\cal R_{ii} }} = - z \mathbf 1_{i\in \cal I_1} - \left( {W\cal R^{\left( i \right)} W^{\top}} \right)_{ii} . \label{resolvent1} 
		\end{equation}
		
		\item[(ii)] For $i\in \mathcal I_1\cup \cal I_2$ and $\fa\in \cal I \setminus\{i\}$, we have that
		\begin{equation}\label{resolvent3}
			\cal R_{i\fa} =-\cal R_{ii} \left(W\cal R^{(i)}\right)_{i\fa} .
		\end{equation}
		
		\item[(iii)] For $\fa \in \mathcal I$ and $\fb, \fc \in \mathcal I \setminus \{\fa\}$, we have that
		\begin{equation}\label{resolvent8}
			\cal R_{\fb\fc} = \cal R_{\fb\fc}^{\left( \fa \right)}  + \frac{\cal R_{\fb\fa} \cal R_{\fa\fc}}{\cal R_{\fa\fa}}. 
		\end{equation}
		
		\item[(iv)] All of the above identities hold for $\cal R^{(\mathbb T)}$ instead of $\cal R$ for any index set $\mathbb T \subset \mathcal I$.
	\end{itemize}
	\label{lemm_resolvent}
\end{lemma}



For $\mu,\nu \in \mathcal I_3$, we define the $2\times 2$ blocks 
\begin{equation}\label{Aij_group}
	\cal R_{[\mu\nu]}:=\left( {\begin{array}{*{20}c}
			{\cal R_{\mu\nu} } & {\cal R_{\mu \overline \nu} }  \\
			{\cal R_{ \overline\mu \nu} } & {\cal R_{ \overline \mu \overline \nu} }  \\
	\end{array}} \right),
\end{equation}
where we denote $\overline \mu : = \mu +n $ and $\overline \nu : = \nu +n $.
We call $\cal R_{[\mu\nu]}$ a diagonal block if $\mu=\nu$, and an off-diagonal block otherwise.
For $i\in \cal I_1$, $j \in \cal I_2$ and $\mu\in \cal I_3$, we define the vectors
\begin{equation}\label{Aij_group2}
			\cal R_{i,[\mu]}:=\left(  {\cal R_{i\mu} }, {\cal R_{i\overline \mu} } 
			\right), \quad \cal R_{[\mu],i}:=\left( {\begin{array}{*{20}c}
					{\cal R_{\mu i} }  \\
					{\cal R_{\overline \mu i} }\\
			\end{array}} \right).
		\end{equation}
		For $\mu\in \cal I_3$, we denote $H^{[\mu]}:=H^{(\mu\overline \mu)}$ and $\cal R^{[\mu]}:=\cal R^{(\mu\overline \mu)}$ in the sense of Definition \ref{defminor}. Then, we record the following resolvent identities, which again can be obtained directly from the Schur complement formula. 
		
		\begin{lemma}
			\label{lemm_resolventgroup}
			We have the following resolvent identities. 
			\begin{itemize}
				\item[(i)] For $\mu\in \mathcal I_3$, we have that
				\begin{equation}\label{eq_res11}
					\cal R_{[\mu\mu]}^{ - 1}  =\frac{1}{\theta_l-1}\begin{pmatrix}1 & -\theta_l^{-1/2} \\ -\theta_l^{-1/2} & 1 \end{pmatrix}  - \begin{bmatrix}( X^{\top} \cal R^{[\mu]} X)_{\mu\mu}& ( X^{\top} \cal R^{[\mu]}Y)_{\mu \overline \mu} \\  (Y^{\top}\cal R^{[\mu]}  X)_{\overline \mu  \mu} &  (Y^{\top} \cal R^{[\mu]}  Y)_{\overline \mu\overline\mu} \end{bmatrix}   .
				\end{equation}
				
				\item[(ii)]
				For $i \in \cal I_1\cup \cal I_2$ and $\mu\in \cal I_3$, we have that
				\begin{align}
					\cal R_{i, [\mu]} = \cal R_{[\mu],i}^\top = - \begin{bmatrix}( \cal R^{[\mu]} X)_{i\mu} ,  ( \cal R^{[\mu]} Y)_{i\overline \mu} \end{bmatrix}\cal R_{[\mu\mu]}. 
					\label{eq_res22}
				\end{align}    
				
				\item[(iii)] 
				For $\mu\ne \nu \in \cal I_3$, we have that
				\be\label{eq_res21}
				\begin{split}
					\cal R_{[\mu\nu]}  &= - \cal R_{[\mu\mu]} \begin{bmatrix}(X^{\top}\cal R^{[\mu]})_{\mu \nu} & (X^{\top}\cal R^{[\mu]})_{\mu\overline \nu} \\ (Y^{\top}\cal R^{[\mu]})_{\overline \mu \nu} & (Y^{\top}\cal R^{[\mu]})_{\overline\mu \overline \nu}\end{bmatrix}   \\
					& = -  \begin{bmatrix}(\cal R^{[\nu]}X)_{\mu  \nu  } & ( \cal R^{[\nu]} Y)_{\mu  \overline \nu}\\ (\cal R^{[\nu]}X)_{\overline \mu  \nu } & (\cal R^{[\nu]}Y)_{\overline \mu \overline\nu }\end{bmatrix} \cal R_{[\nu\nu]}. 
				\end{split}
				\ee

				\item[(iv)] For $\mu \in \mathcal I_3 $ and $\fa_1,\fa_2,\fb_1,\fb_2 \in \mathcal I\setminus \{\mu,\overline \mu\}$, we have that
				\begin{equation}\label{eq_res3}
				\begin{split}
					\begin{pmatrix}\cal R_{\fa_1\fb_1} & \cal R_{\fa_1\fb_2} \\ \cal R_{\fa_2\fb_1} & \cal R_{\fa_2\fb_2} \end{pmatrix} &= \begin{pmatrix}\cal R^{[\mu]}_{\fa_1\fb_1} & \cal R^{[\mu]}_{\fa_1\fb_2} \\ \cal R^{[\mu]}_{\fa_2\fb_1} & \cal R^{[\mu]}_{\fa_2\fb_2} \end{pmatrix}\\
					&+\begin{pmatrix}\cal R_{\fa_1\mu} & \cal R_{\fa_1\overline\mu} \\ \cal R_{\fa_2\mu} & \cal R_{\fa_2\overline\mu} \end{pmatrix}\cal R^{-1}_{[\mu\mu]}\begin{pmatrix}\cal R_{\mu\fb_1} & \cal R_{\mu\fb_2} \\ \cal R_{\overline \mu\fb_1} & \cal R_{\overline \mu\fb_2} \end{pmatrix}. 				\end{split}
				\end{equation}
			\end{itemize}
		\end{lemma}
		
		Using the above tools, we now prove the following entrywise version of Theorem \ref{thm_localw}.
		\begin{proposition} [Entrywise local law]\label{thm_localentry} 
			If Assumption \ref{main_assm} holds, then we have that
			\begin{equation}\label{entry_outstrongw}
				\max_{\fa,\fb\in \cal I}\left|  \cal R_{\fa\fb}(z,0)  - \Gamma_{\fa\fb}(z)  \right|  \prec  n^{-1/2} \quad \text{uniformly in $z\in \mathbf D$.}
			\end{equation}
		\end{proposition}
		
		
		For the proof of Proposition \ref{thm_localentry}, we introduce the following $\cal Z$ variables
		\begin{equation*}
			\cal   Z_{\fa} :=(1-\mathbb E_{\fa})\big(\cal R_{\fa\fa}\big)^{-1}, 
		\end{equation*}
		where $\mathbb E_{\fa}[\cdot]:=\mathbb E[\cdot\mid H^{(\fa)}],$ i.e., it is the partial expectation over the $\fa$-th row and column of $H$. By (\ref{resolvent1}), we have that for $i\in \cal I_\al$, $\al=1,2$, 
		\begin{equation}
			\cal Z_i = (\mathbb E_{i} - 1) \left( {W\cal R^{\left( i \right)} W^{\top}} \right)_{ii} = \sum_{\mu ,\nu\in \mathcal I_{\al+2}} \cal R^{(i)}_{\mu\nu} \left(\frac{1}{n}\delta_{\mu\nu} - W_{i\mu}W_{i\nu}\right). \label{Zi}
		\end{equation}
		We also introduce the matrix-valued $\cal Z$ variables 
		\begin{equation}\label{Zmu0}
			\cal   Z_{[\mu]} :=\left(1-\bbE_{[\mu]}\right)\left(\cal R_{[\mu\mu]} \right)^{-1}, 
		\end{equation}
		where $\mathbb E_{[\mu]}[\cdot]:=\mathbb E[\cdot\mid H^{[\mu]}],$ i.e., it is the partial expectation over the $\mu$-th and $\overline \mu$-th rows and columns of $H$. By (\ref{eq_res11}), we have that
		\begin{equation}\label{Zmu}
			\cal Z_{[\mu]} = \begin{bmatrix}\sum_{i, j \in \cal I_1} \cal R^{[\mu]}_{ij} (n^{-1}\delta_{ij} -X_{i\mu  }X_{j\mu}) & \sum_{i \in \cal I_1, j\in \cal I_2} \cal R^{[\mu]}_{ij} X_{i\mu}Y_{j\overline \mu} \\ \sum_{i \in \cal I_1, j\in \cal I_2}\cal R^{[\mu]}_{ji} X_{i \mu} Y_{j\overline \mu } & \sum_{i, j \in \cal I_2} \cal R^{[\mu]}_{ij} (n^{-1}\delta_{ij} - Y_{i\overline\mu }Y_{j\overline\mu}) \end{bmatrix} .
		\end{equation}
		We also define the random error to control the off-diagonal entries,
		\begin{equation}  \label{eqn_randomerror}
			\begin{split}
				\Lambda _o : &=  \max_{i \ne j \in \mathcal I_1\cup \cal I_2} \left| {\cal R_{ij}  } \right| +  \max_{\mu\ne \nu \in \cal I_3} \|\cal R_{[\mu\nu]} \|+ \max_{i\in \cal I_1\cup \cal I_2,\mu\in \cal I_3}  \left\| \cal R_{i,[\mu]}  \right\|  .
			\end{split}
		\end{equation}
		Now, we claim the following large deviation estimate for the $\cal Z$ variables and off-diagonal entries.
		\begin{claim}
			Under the setting of Theorem \ref{thm_localw}, we have that
			\begin{align}
				& \Lambda_o+|Z_{i}| + \|Z_{[\mu]}\|  \prec n^{-1/2}. \label{Zestimate1}
			\end{align}
		\end{claim}
		\begin{proof}
			For $i\in \cal I_{\al}$, $\al=1,2$, applying Lemma \ref{largedeviation} to $\cal Z_{i}$ in (\ref{Zi}), we get that 
			\begin{equation}\nonumber
				\begin{split}
					\left|\cal Z_{i}\right| 
					\prec  \frac{1}{n} \Big(  \sum_{\mu, \nu \in \cal I_{\al+2}}  {\big| \cal R_{\mu\nu}^{(i)}  \big|^2 } \Big)^{1/2} \le \frac{1}{\sqrt{n}} \Big[ \frac{1}{n}\sum_{\mu \in \cal I_{\al+2}}  {\left(\cal R^{(i)}   (\cal R^{(i)})^*\right)_{\mu\mu} }  \Big]^{1/2} \prec n^{-1/2},
				\end{split}
			\end{equation}
			where in the last step we applied \eqref{priorim} to $\cal R^{(i)}$ to get  $(\cal R^{(i)} (\cal R^{(i)})^* )_{\mu\mu} =\OO(1) $ with high probability (note $\cal R^{(i)}$ satisfies the same assumption as $\cal R$). Similarly, applying Lemma \ref{largedeviation} to $\cal Z_{[\mu]}$ in (\ref{Zmu}), we obtain that  
			\begin{equation}\label{estimate_Zmu} \|\cal Z_{[\mu]}\|\prec \frac{1}{n} \Big(  \sum_{i,j \in \cal I_{1}\cup \cal I_2}  {\big| \cal R_{ij}^{[\mu]}  \big|^2 }  \Big)^{1/2} =  \frac{1}{\sqrt{n}} \Big[  \frac{1}{n}\sum_{i \in \cal I_{1}\cup \cal I_2}  {\left(\cal R^{[\mu]} (\cal R^{[\mu]})^*\right)_{ii} } \Big]^{1/2} \prec n^{-1/2}.\end{equation}
			The proof of the off-diagonal estimate is similar. For $i\ne j \in \mathcal I_1\cup \cal I_2$, using \eqref{resolvent3}, Lemma \ref{largedeviation} and \eqref{priorim}, we obtain that 
			\begin{equation}\nonumber
				|\cal R_{ij}| \prec \frac{1}{\sqrt{n}} \Big( \sum_{\mu \in \cal I_3\cup \cal I_4}  {\big| \cal R_{\mu j}^{(i)} \big|^2 }  \Big)^{1/2} \prec n^{-1/2}. 
			\end{equation}
			For $\mu \ne \nu \in \mathcal I_3$, using \eqref{eq_res21}, Lemma \ref{largedeviation} and \eqref{priorim}, we obtain that 
			$$ \left\| \cal R_{[\mu\nu]} \right\| \prec    \frac{1}{n} \Big( \sum_{i \in \cal I_{1}\cup \cal I_2}  {\big| \cal R_{i\nu}^{[\mu]}  \big|^2 }  \Big)^{1/2}+\frac{1}{n} \Big(  \sum_{i \in \cal I_{1}\cup \cal I_2}  {\big| \cal R_{i\overline \nu}^{[\mu]}  \big|^2 }   \Big)^{1/2} \prec n^{-1/2}.$$
			Finally, for $i\in \cal I_1\cup \cal I_2$ and $\mu \in \mathcal I_3$, using \eqref{eq_res22}, Lemma \ref{largedeviation} and \eqref{priorim}, we obtain that  
			$$ \left\| \cal R_{i,[\mu]} \right\| \prec   \frac{1}{n} \Big( \sum_{j \in \cal I_{1}\cup \cal I_2 }  {\big| \cal R^{[\mu]}_{ij}  \big|^2 }  \Big)^{1/2}  \prec n^{-1/2}.$$
			Combining the above estimates, we conclude \eqref{Zestimate1}.
			\end{proof}

 A key component of the proof for Proposition \ref{thm_localentry} is to show that $\omega_{\al}$, $\al=1,2,3,4,$ satisfy the self-consistent equations in \eqref{selfomegaerror} up to some small errors $|\cal E_\al|\prec n^{-1/2}$.

	\begin{lemma}\label{lemm_selfcons_weak}
	Fix any constant $\e>0$. The following estimates hold uniformly in $z \in \mathbf D$: 
		\begin{align}
			& \left|\frac{c_1}{\omega_{1} }+ z +  \omega_{3}\right| \prec n^{-1/2}, \quad  \left|\frac{c_2}{\omega_{2} } + \omega_{4}\right| \prec n^{-1/2}, \label{selfcons_lemm}\\
			& \left| \omega_{3} +(1-\theta_l ) \frac{1 + (1-\theta_l )\omega_{2}}{[1 + (1-\theta_l )\omega_{1} ][1 + (1-\theta_l )\omega_{2} ]-\theta_l ^{-1}} \right|\prec n^{-1/2} ,\label{selfcons_lemm12} \\
			& \left|   \omega_{4} +(1-\theta_l) \frac{1 + (1-\theta_l )\omega_{1}}{[1 + (1-\theta_l )\omega_{1} ][1 + (1-\theta_l )\omega_{2} ]-\theta_l ^{-1}} \right|\prec n^{-1/2} .\label{selfcons_lemm13}
		\end{align}
	\end{lemma}
	
	\begin{proof}
		Similar to \eqref{defomegaal}, for $i\in \cal I_1\cup \cal I_2$ and $\mu \in \cal I_3$, we denote
		$$\omega_\al^{(i)}:=\frac1n\sum_{\fa \in \cal I_\al} \cal R^{(i)}_{\fa\fa},\quad   \omega_\al^{[\mu]}:=\frac1n\sum_{i \in \cal I_\al} \cal R^{[\mu]}_{\fa\fa},\quad \al=1,2,3,4.$$
		Using (\ref{resolvent1}) and (\ref{Zi}), we get that for $i\in \cal I_1$ and $j\in \cal I_2$,
		\begin{equation}\label{self_Gii}
			\frac{1}{ \cal R_{ii} }= - z - \omega_3 + \epsilon_i, \quad \frac{1}{\cal R_{jj}}= - \omega_4 + \epsilon_j,
		\end{equation}
		where 
		$$\epsilon_i :=\cal Z_i + \omega_3 - \omega_3^{(i)}, \quad \epsilon_j :=\cal Z_j + \omega_4 - \omega_4^{(j)}.$$
		On the other hand, using \eqref{eq_res11} and \eqref{Zmu0}, we get that for $\mu \in \cal I_3$,
		\begin{equation}\label{self_Gmu}
			\cal R_{[\mu\mu]}^{-1} = \frac{1}{\theta_l-1}\begin{pmatrix}1 & -\theta_l^{-1/2} \\ -\theta_l^{-1/2} & 1 \end{pmatrix}  - \left( {\begin{array}{*{20}c}
					{ \omega_{1}  } & {0}  \\
					{0} & { \omega_{2} } \end{array}} \right) + \epsilon_\mu,
		\end{equation}
		where 
		$$ \epsilon_\mu :=\cal Z_{\mu} + \left( {\begin{array}{*{20}c}
				{ \omega_{1}- \omega_{1}^{[\mu]}  } & {0}  \\
				{0} & { \omega_{2} -\omega_{2}^{[\mu]} } \end{array}} \right) .$$
		Now, using \eqref{resolvent8} and \eqref{Zestimate1}, we get that
		$$ \omega_3 - \omega_3^{(i)} = \frac{1}{n}\sum_{\mu \in \cal I_3}\frac{\cal R_{\mu i} \cal R_{i\mu}}{\cal R_{ii}} =\OO_\prec (n^{-1}),$$
		where in the second step we also used $|\cal R_{ii}|\gtrsim 1$ by \eqref{roughinitial} and \eqref{absmc}. We have similar estimates for $\omega_4 - \omega_4^{(j)}$,  $\omega_{1}-\omega_{1}^{[\mu]}$ and $\omega_{2}-\omega_{2}^{[\mu]}$. Together with \eqref{Zestimate1}, these estimates give that
		\begin{equation}\label{erri}
			\max_{i\in \cal I_1\cup \cal I_2} |\epsilon_i |  + \max_{\mu \in \cal I_3} \|\epsilon_\mu\|  \prec n^{-1/2}.
		\end{equation}
		
		Using the first equation in \eqref{self_Gii} and \eqref{erri}, we obtain that 
		\be\label{Gii0}
		\begin{split}
			&\omega_1= \frac1n\sum_{i\in \cal I_1} \cal R_{ii} = \frac1n\sum_{i\in \cal I_1}\frac{1}{- z - \omega_3 +\e_i} =  \frac{c_1}{- z - \omega_3} + \OO_\prec(n^{-1/2}), 
		\end{split}
		\ee
		where in the second step we used $| z + \omega_3| \gtrsim 1$ with high probability by \eqref{roughinitial}. This gives the first equation in \eqref{selfcons_lemm}. Similarly, using the second equation in \eqref{self_Gii}, we can obtain the second equation in \eqref{selfcons_lemm}. 	With \eqref{roughinitial} and \eqref{absmc}, we can check that
		\be\label{inverse_bound}
		\left\| \left[\frac{1}{\theta_l-1}\begin{pmatrix}1 & -\theta_l^{-1/2} \\ -\theta_l^{-1/2} & 1 \end{pmatrix}  - \left( {\begin{array}{*{20}c}
				{ \omega_{1}  } & {0}  \\
				{0} & { \omega_{2} } \end{array}} \right)\right]^{-1}\right\|\lesssim 1\quad \text{with high probability.}\ee
		Taking the matrix inverse of (\ref{self_Gmu}) and using \eqref{erri} and \eqref{inverse_bound}, we obtain that for $\mu \in \cal I_3$,
		\be\label{Gmumu0}
		\begin{split}
			\cal R_{[\mu\mu]} &=\frac{\theta_l-1}{[1+(1-\theta_l)\omega_1][1+(1-\theta_l)\omega_2]-\theta_l^{-1}}\begin{pmatrix}1+(1-\theta_l)\omega_2 & \theta_l^{-1/2} \\ \theta_l^{-1/2} & 1+(1-\theta_l)\omega_1 \end{pmatrix} \\
			&+ \OO_\prec(n^{-1/2} ) .
		\end{split}
		\ee
		After taking the average $n^{-1}\sum_{\mu\in \cal I_3}$ over the $(1,1)$-th and $(2,2)$-th entries of equation \eqref{Gmumu0}, we obtain the equations \eqref{selfcons_lemm12} and \eqref{selfcons_lemm13}.
	\end{proof}
	Combining Lemma \ref{lemm_selfcons_weak} with Lemma \ref{lem_stabw}, we conclude the proof of Proposition \ref{thm_localentry}.
	\begin{proof}[Proof of Proposition \ref{thm_localentry}]
		We apply Lemma \ref{lem_stabw}, where \eqref{prior1} is implied by \eqref{roughinitial}, and the equations in \eqref{selfomegaerror} follow from Lemma \ref{lemm_selfcons_weak}. Then, \eqref{Stability1} implies that
		\begin{equation}
			\max_{\al=1}^4 \left|\omega_\al(z)-\omega_{\al c}(z)\right|\prec n^{-1/2}.\label{Stabilityentry}
		\end{equation}
		Plugging \eqref{Stabilityentry} into \eqref{self_Gii} and \eqref{Gmumu0}, we then get the diagonal estimate
		$$\max_{i \in \mathcal I_1 } \left| {\cal R_{ii} - c_1^{-1}\omega_{1c} } \right| + \max_{j \in \mathcal I_2 } \left| {\cal R_{jj} - c_2^{-1}\omega_{2c} } \right| +  \max_{\mu \in \cal I_3 }\left\|\cal R_{[\mu\mu]}- \begin{pmatrix}  \omega_{3c}    & g_1  \\  g_1  &  \omega_{4c} \end{pmatrix}\right\|\prec n^{-1/2}.$$
		Combining it with the off-diagonal estimate in \eqref{Zestimate1}, we conclude \eqref{entry_outstrongw}.
	\end{proof}
	
	Finally, we can complete the proof of Theorem \ref{thm_localw} based on Proposition \ref{thm_localentry}. 
	\begin{proof}[Proof of Theorem \ref{thm_localw}]
		With the entrywise local law, Proposition \ref{thm_localentry}, the proof of \eqref{aniso_outstrongw} uses a polynomialization method developed in \cite{isotropic}. In fact, the argument is exactly the same as the one in Section 7 of \cite{PartIII}. Hence, we omit the details. However, we make one remark that in the proof, we need to bound the high moments 
		$$\mathbb E\left| \langle \mathbf u, \cal R(z,0) \mathbf v\rangle - \langle \mathbf u, \Gamma(z)\mathbf v\rangle \right|^{2a} $$
		for fixed large $a\in \N$. So for regularity reasons, we shall use the resolvent $\cal R(z+\ii n^{-4},z')$ with $z'=\ii n^{-4}$ in order to make use of the deterministic bound \eqref{op R} on exceptional low-probability events, which justifies the applicability of Lemma \ref{lem_stodomin} (iii). The structure of the proof is as follows. First, the argument in the proof of Claim \ref{removehat} allows us to extend the entrywise local law \eqref{entry_outstrongw} to $\cal R(z+\ii n^{-4},z')$. Then, we can prove the anisotropic local law \eqref{aniso_outstrongw} for $\cal R(z+\ii n^{-4},z')$ using the argument in Section 7 of \cite{PartIII}. After that, 
		applying the argument in the proof of Claim \ref{removehat} again allows us to extend the anisotropic local law to $\cal R(z,0)$.
	\end{proof}

 \section{Proof of Theorem \ref{main_thm1}}\label{secpfmain1}

With Proposition \ref{main_prop1} and Proposition \ref{redGthm}, we see that \eqref{limf} holds in the almost Gaussian case. Hence, to conclude Theorem \ref{main_thm1}, it suffices to show that the general case is sufficiently close to the almost Gaussian case regarding the outliers. In particular, by \eqref{redG}, \eqref{defPPP+QQQ} and \eqref{M-M0}, we only need to show that the asymptotic distribution of $\cal M(\theta_l)$ in \eqref{defM}
 for general $X$ and $Y$ is the same as that of $\cal M^g(\theta_l)$ defined for almost Gaussian $X\equiv X^g$ and $Y\equiv Y^g$. 
 Corresponding to \eqref{almost GuassianX} and \eqref{almost GuassianY}, we define the index set (``$s$" stands for ``small")
 $$\cal I_s:=\Big\{  k \in \cal I_1 : \max_{1\le i \le r}|\bu_i^a(k)| \le n^{-\tau_0} \Big\}\cup \Big\{ k \in \cal I_2: \max_{1\le i \le r}|\bu^b_i(k)| \le n^{-\tau_0} \Big\}.$$
 Corresponding to \eqref{linearize_block} and \eqref{def_resolvent}, we define a new self-adjoint block matrix $ H^g$ and its resolvent as
 \begin{equation}\nonumber
 	H^g(z) : = \begin{bmatrix} 0 & \begin{pmatrix}  X^g & 0\\ 0 &  Y^g \end{pmatrix}\\ \begin{pmatrix} (X^g)^{\top} & 0\\ 0 &  (Y^g)^{\top}\end{pmatrix}  & \begin{pmatrix} z I_n & z^{1/2}I_n\\ z^{1/2} I_n &  z I_n\end{pmatrix}^{-1}\end{bmatrix} , \quad G^g(z):= \left[ H^g(z)\right]^{-1} ,
 \end{equation}
 where $X^g$ and $Y^g$ are defined through 
 \be\label{defXYg}
 X^g_{i\mu}=\begin{cases} X_{i\mu}, \ &\text{if } i \notin \cal I_s \\ g^{(1)}_{i\mu}, \ &\text{if } i \in \cal I_s \end{cases}, \quad  Y^g_{i\mu}=\begin{cases} Y_{i\mu}, \ &\text{if } i \notin \cal I_s \\ g^{(2)}_{i\mu}, \ &\text{if } i \in \cal I_s \end{cases} .
 \ee
 Here, $g^{(1)}_{i\mu}$ and $g^{(2)}_{i\mu}$ are i.i.d.\;Gaussian random variables independent of $(X,Y)$ and with mean zero and variance $n^{-1}$. Note that $X^g$ and $Y^g$ satisfy the setting of Proposition \ref{main_prop1}.
 
 Define the set of pairs of indices
 $$\cal J_s:= \{ (i,\mu): i \in \cal I_1\cap  \cal I_s, \mu \in \cal I_3\}\cup  \{ (i,\mu): i \in  \cal I_2\cap  \cal I_s, \mu \in \cal I_4\}.$$ 
 We choose a bijective ordering map $\Phi$ on $\cal J_s$:
 \begin{equation*}
 	\Phi: \cal J_s \rightarrow \{1,\ldots,\gamma_{\max}\}, \quad \gamma_{\max}: =|\cal J_s|= |\cal I_s|\cdot n.
 \end{equation*} 
 Similar to \eqref{simple_W}, we introduce simplified notations
 \be\label{defnWadd} W:=\begin{pmatrix} X & 0\\ 0 &  Y\end{pmatrix},\quad W^g:=\begin{pmatrix}  X^g & 0\\ 0 &  Y^g\end{pmatrix}.\ee
 For any $1\le \gamma \le \gamma_{\max}$, we define the $(\cal I_1 \cup \cal I_2)\times (\cal I_3 \cup \cal I_4)$ matrix $W^{\{\gamma\}} $ such that 
 $$W_{i\mu}^{\{\gamma\}} =\begin{cases}W_{i\mu} , \ &\text{ if } \Phi(i,\mu)\leq \gamma \\   {W}^g_{i\mu}, \ & \text{ if } \Phi(i,\mu)> \gamma\end{cases},\quad \text{and}\quad  W_{i\mu}^{\{\gamma\}}=W_{i\mu}= {W}^g_{i\mu} \ \ \text{ for } \  (i,\mu)\notin \cal J_s.$$ 
 Correspondingly, we define 
 \begin{equation}\nonumber
 	H^{\{\gamma\}}(z) : = \begin{bmatrix} 0 & W ^{\{\gamma\}} \\ ( W^{\{\gamma\}})^{\top} & \begin{pmatrix}  z  I_n & z^{1/2}I_n\\ z^{1/2} I_n & z I_n\end{pmatrix}^{-1}\end{bmatrix} , \quad G^{\{\gamma\}}:= [ H^{\{\gamma\}}(z)]^{-1} .
 \end{equation}
Under the above definition, we have $G^{\{0\}}=G^g$ and $G^{\{\gamma_{\max}\}}=G$. For $\Phi(i,\mu)=\gamma$, we can write that
 \begin{equation}\label{R0} H^{\{\gamma\}} = Q^{\{\gamma\}} + W_{i\mu} E^{\{\gamma\}}, \quad H^{\{\gamma-1\}} = Q^{\{\gamma\}} +  W^g_{i\mu} E^{\{\gamma\}}, 
 \end{equation}
 where $E^{\{\gamma\}} $ is a matrix defined by 
 \be\label{defnEgamma_add}(E^{\{\gamma\}})_{ab}=\mathbf 1_{(a,b)=(i,\mu)}+\mathbf 1_{(a,b)=(\mu,i)},\ee 
 and $Q^{\{\gamma\}}$ is a random matrix with zero $(i,\mu)$-th and $(\mu,i)$-th entries. In particular, $\cal Q^{\{\gamma\}}$ is independent of $W_{i\mu}$ and $W^g_{i\mu}$.
 For simplicity of notations, for any $\gamma$ we denote that 
 \begin{equation}
 	T^{\{\gamma\}}:=G^{\{\gamma\}}, \quad  S^{\{\gamma\}}:=G^{\{\gamma-1\}}, \quad R^{\{\gamma\}}:=(Q^{\{\gamma\}} )^{-1}.  \label{R}
 \end{equation}
 Then, given any function $f$, we can write that
 \be\label{telescop}\E f\left(G\right)-\E f\left(G^g\right)= \sum_{\gamma=1}^{\gamma_{\max}}\left[\E f\left(T^{\{\gamma\}}\right)-\E f\left(S^{\{\gamma\}}\right)\right]. \ee
 We will estimate each term in the sum using resolvent expansions. More precisely, by \eqref{R0} we have that
 \begin{align*}
 	T^{\{\gamma\}}= \left(Q^{\{\gamma\}}  +  W_{i\mu} E^{\{\gamma\}}\right)^{-1}=\left(1+ W_{i\mu} R^{\{\gamma\}}E^{\{\gamma\}}\right)^{-1}R^{\{\gamma\}}.  
 \end{align*}
 For any fixed $k\in \N$, we can expand $T^{\{\gamma\}}$ till order $k$ as
 \begin{equation}
 	T^{\{\gamma\}}=\sum_{s=0}^k (-W_{i\mu} )^s \left(R^{\{\gamma\}}E^{\{\gamma\}}\right)^s R^{\{\gamma\}}+(-W_{i\mu} )^{k+1}\left(R^{\{\gamma\}}E^{\{\gamma\}}\right)^{k+1}T^{\{\gamma\}}. \label{RESOLVENTEXPANSION}
 \end{equation}
 We can also expand $R^{\{\gamma\}}$ in terms of $T^{\{\gamma\}}$ as
 \begin{equation}
 	\begin{split}
 		R^{\{\gamma\}}&=\left(1- W_{i\mu} T^{\{\gamma\}}E^{\{\gamma\}}\right)^{-1}T^{\{\gamma\}}  \\
		&=\sum_{s=0}^k   W_{i\mu}^s \left(T^{\{\gamma\}}E^{\{\gamma\}}\right)^s T^{\{\gamma\}}+  W_{i\mu}^{k+1}\left(T^{\{\gamma\}}E^{\{\gamma\}}\right)^{k+1}R^{\{\gamma\}}. \label{RESOLVENTEXPANSION2}
 	\end{split}
 \end{equation}
 We can get similar expansions for $S^{\{\gamma\}}$ and $R^{\{\gamma\}}$ by replacing $(T^{\{\gamma\}},W_{i\mu})$ with $(S^{\{\gamma\}},W^g_{i\mu})$.  We will combine these resolvent expansions with the Taylor expansion of $f$ to estimate the right-hand side of \eqref{telescop}. 

In the following proof, we use the regularized resolvent $\wh G(z)$ in Definition \ref{resol_not2} with $z = \theta_l +\ii n^{-4}$. We can also define $\wh G^g(z)$ and $\wh G^{\{\gamma\}}(z)$ in a similar way. By \eqref{op G},  $\wh S^{\{\gamma\}}$, $\wh T^{\{\gamma\}}$ and $\wh R^{\{\gamma\}}$ satisfy the deterministic bound 
 \begin{equation}
 	\max_{\gamma} \max \left\{ \| \wh S^{\{\gamma\}}(z) \|, \|\wh T^{\{\gamma\}}(z) \|, \|\wh R^{\{\gamma\}}(z) \| \right\} \lesssim n^{14}. \label{5BOUNDT}
 \end{equation}
 Again, because of this bound, Lemma \ref{lem_stodomin} (iii) can be used tacitly, and we will not emphasize this fact again in the following proof. 
 Using the expansion (\ref{RESOLVENTEXPANSION2}) for a sufficiently large $k$ (for example, $k=100$ will be enough), $|W_{i\mu}|\prec n^{-1/2}$, the anisotropic local law \eqref{aniso_outstrong} for $\wh T^{\{\gamma\}}$, and the bound (\ref{5BOUNDT}) for $\wh R^{\{\gamma\}}$, we can obtain that for any deterministic unit vectors $\mathbf u,\mathbf v \in \mathbb C^{\mathcal I}$, 
 \begin{equation}
 	\max_{\gamma} \left|\left\langle \bu,  \left[\wh R^{\{\gamma\}}(z)-\Pi(z) \right]\mathbf v\right\rangle\right| \prec n^{-1/2} .\label{5BOUND1}
 \end{equation}
 Moreover, using the same argument as in the proof of Claim \ref{removehat}, we can easily show that 
 \be\label{claim sameasymp}
 \text{$\cal M(\theta_l)$ has the same asymptotic distribution as } \wh{\cal M}(z),
 \ee
 where $\wh{\cal M}(z)$ is defined as (recall the notations in \eqref{defM})
 \be\label{def_whM}  \wh{\cal M}(z):=\sqrt{n}\mathscr U^{\top} \left[\wh G(z)-\Pi(z)\right] \mathscr U,   \quad  z = \theta_l + \ii n^{-4}, \quad \mathscr U:=\begin{pmatrix} \bU_a & 0 & 0 & 0 \\ 0 & \bU_b & 0 & 0 \\ 0 & 0 & \wt Z^\top & 0 \\ 0 & 0 & 0 & \wt Z^\top   \end{pmatrix}  .\ee
 By replacing $\wh G$ with $\wh G^g$ or $\wh G^{\{\gamma\}}$, we can also define $\wh {\cal M}^g$ or $\wh {\cal M}^{\{\gamma\}}$.
 Then, we will use the following comparison lemma to complete the proof of Theorem \ref{main_thm1}. 
 
 \begin{lemma}\label{xRalmost}
 Fix any $\gamma=\Phi(i,\mu)$ with $(i,\mu)\in \cal J_s$. We abbreviate 
 $$\cal M_R^{\{\gamma\}}: = \sqrt{n}\mathscr U^{\top} \left[\wh R^{\{\gamma\}}(z)-\Pi(z)\right] \mathscr U, \quad z = \theta_l + \ii n^{-4}.$$
 The matrices $\cal M^{\{\gamma\}}_S$ and $\cal M^{\{\gamma\}}_T$ are defined similarly by replacing $\wh R^{\{\gamma\}}$ with $\wh S^{\{\gamma\}}$ and $\wh T^{\{\gamma\}}$, respectively. Let $f\in C_b^{3}(\C^{4r\times 4r})$ be a function with bounded partial derivatives up to third order, and $a\equiv a_n$ be an arbitrary deterministic sequence of $4r\times 4r$ symmetric matrices. Then, we have that
 \be \label{xRTS}
 \begin{split}
 \mathbb E f\left(\cal M_T^{\{\gamma\}} + a\right)&=\mathbb Ef\left(\cal M_R^{\{\gamma\}} + a\right)+ \sum_{  k,l=1}^{4r} \cal Q_{kl}^{\{\gamma\}} \mathbb E\frac{\partial f}{\partial x_{kl}} \left(\cal M_R^{\{\gamma\}} + a\right) \\
 & + \mathscr A_\gamma + \OO_\prec(n^{-\tau_0} \cal E_{\gamma}), 
 \end{split}
 \ee
 \be \label{xRTS2}
 \mathbb Ef\left(\cal M_S^{\{\gamma\}} + a\right)= \mathbb Ef\left(\cal M_R^{\{\gamma\}} + a\right)+   \mathscr A_\gamma + \OO_\prec(n^{-\tau_0} \cal E_{\gamma}), 
 \ee
 where $\mathscr A_\gamma$ satisfies $\mathscr A_\gamma\prec n^{-\tau_0}$, and we denote
 $$ \cal Q_{kl}^{\{\gamma\}}:=\begin{cases} - n^{-1} \left( n^{3/2}\mathbb EX_{11}^3\right)\cdot\left(   \mathscr U_{\mu k}  \mathscr U_{il} +  \mathscr U_{i k}  \mathscr U_{\mu l}  \right), \ & \text{if } \ \mu\in \cal I_3   \\  - n^{-1} \left( n^{3/2}\mathbb EY_{11}^3\right)\cdot \left(   \mathscr U_{\mu k}  \mathscr U_{il} +  \mathscr U_{i k}  \mathscr U_{\mu l}  \right), \ & \text{if } \ \mu \in \cal I_4 \end{cases},$$
 and
 \be\label{add_def_Egamma}\cal E_{\gamma}:=\sum_{ k,l=1}^{ 4r} \sum_{\sigma_1,\sigma_2=0}^2 n^{-2+\sigma_1/2+\sigma_2/2} |\mathscr U_{ik}|^{\sigma_1}  |\mathscr U_{\mu l}|^{\sigma_2} .\ee
 \end{lemma}
 \begin{proof}
 The proof of this lemma is almost the same as the one for Lemma 7.13 of \cite{KY}, where the main inputs are the local laws \eqref{aniso_outstrong} and \eqref{5BOUND1}, the simple identity \eqref{telescop}, and the resolvent expansions \eqref{RESOLVENTEXPANSION} and \eqref{RESOLVENTEXPANSION2}. The cosmetic modifications are mainly due to the fact that our local law takes a different form than the one in Theorem 2.2 of \cite{KY}. So we ignore the details.
 \end{proof}

 Combining Proposition \ref{redGthm}, Proposition \ref{main_prop1} and Lemma \ref{xRalmost}, we can conclude the proof of Theorem \ref{main_thm1}.
 \begin{proof}[Proof of Theorem \ref{main_thm1}]
 	We fix any function $f\in C_c^\infty(\C^{4r\times 4r})$ and $\wt Z$ satisfying \eqref{V-F} and \eqref{Vmax}. Using \eqref{xRTS} and \eqref{xRTS2}, we get that
 	\be\label{xT-xR} 
 	\begin{split}\E_{X,Y}f\left(\cal M_T^{\{\gamma\}}+a\right) &=\E_{X,Y}f\left(\cal M_S^{\{\gamma\}}+a\right)  +  \sum_{ k,l=1}^{ 4r}  \cal Q_{kl}^{\{\gamma\}} \E_{X,Y}\frac{\partial f}{\partial x_{kl}} \left(\cal M_R^{\{\gamma\}} +a\right) \\
	&+ \OO_\prec(n^{-\tau_0} \cal E_{\gamma}),\end{split}\ee
 	where \smash{$\E_{X,Y}$} means the partial expectation with respect to $X$, $Y$, $X^g$ and $Y^g$ (for simplicity, we did not add $X^g$ and $Y^g$ to the subscript). Since $| \mathscr U_{\mu k} | \le n^{-1/2+\e}$ for $\mu \in \cal I_3\cup \cal I_4$ and $ |\mathscr U_{il} | \le n^{-\tau_0}$ for $i\in \cal I_s$, it is easy to check that
 	\be \nonumber
 	\| \cal Q^{\{\gamma\}}\|_{\max} \lesssim  \min \{n^{-3/2-\tau_0+\e}, \mathcal E_\gamma \} ,\quad \text{for} \quad 1\le \gamma \le \gamma_{\max},\ee
 	where $ \cal Q^{\{\gamma\}}$ is the $4r\times 4r$ matrix with entries  $ \cal Q_{kl}^{\{\gamma\}}$. Thus, for any fixed $1\le k,l\le 4r$ and $1\le \gamma \le \gamma_{\max}$, applying \eqref{xRTS} with $f$ replaced by $\partial_{x_{kl}}f$, we get that
 	$$ \E_{X,Y}\frac{\partial f}{\partial x_{kl}} \left(\cal M_R^{\{\gamma\}}+a \right) = \E_{X,Y}\frac{\partial f}{\partial x_{kl}} \left(\cal M_T^{\{\gamma\}}+a\right) + \OO_\prec( n^{-\tau_0} ).$$
 	Plugging it into \eqref{xT-xR}, we get that
 	\be\nonumber  \begin{split}
 		\E_{X,Y} f\left(\cal M_S^{\{\gamma\}}+a\right) & =  \E_{X,Y} f\left(\cal M_T^{\{\gamma\}}+a\right)  -  \sum_{k,l=1}^{ 4r}  \cal Q_{kl}^{\{\gamma\}} \E_{X,Y}\frac{\partial f}{\partial x_{kl}} \left(\cal M_T^{\{\gamma\}} +a\right) \\
		&+  \OO_\prec(n^{-\tau_0} \cal E_{\gamma}).
 	\end{split}\ee
 	On the other hand, we have the Taylor expansion
 	\be\nonumber  \begin{split} \E_{X,Y} f\left(\cal M_T^{\{\gamma\}} + a -   \cal Q^{\{\gamma\}}\right) & =  \E_{X,Y} f\left(\cal M_T^{\{\gamma\}}+a\right)-   \sum_{  k,l=1}^{ 4r}  \cal Q_{kl}^{\{\gamma\}} \E_{X,Y}\frac{\partial f}{\partial x_{kl}} \left(\cal M_T^{\{\gamma\}} +a\right) \\
	&+  \OO_\prec(n^{-\tau_0} \cal E_{\gamma}).
 	\end{split}\ee
 	Comparing the above two equations, we get that
 	\be\label{xT-xR2} \E_{X,Y} f\left(\cal M_T^{\{\gamma\}} + a -    \cal Q^{\{\gamma\}}\right) = \E_{X,Y} f\left(\cal M_S^{\{\gamma\}}+a\right) +  \OO_\prec(n^{-\tau_0} \cal E_{\gamma}).\ee
 	We iterate \eqref{xT-xR2} starting at $\gamma=1$ and $a=0$ and obtain that 
 	\be\label{xT-xR3} 
 	\E_{X,Y} f\left(\cal M_T^{(\gamma_{\max})}   - \sum_{\gamma=1}^{\gamma_{\max}}   \cal Q^{\{\gamma\}}\right) = \E_{X,Y} f\left(\cal M_T^{(0)}\right) +  \OO_\prec(n^{-\tau_0}),
 	\ee
 	where we also used the bound $ \sum_{\gamma}\mathcal E_\gamma =\OO(1)$, which can be verified directly using the definition \eqref{add_def_Egamma}. Now, using \eqref{rowsum_V}, we can bound that 
 	$$ \sum_{\gamma=1}^{\gamma_{\max}}  \cal Q^{\{\gamma\}} \prec n^{-1/2}.$$
  	Plugging it into \eqref{xT-xR3}, we obtain that
 	$$\E f\left(\cal M_T^{(\gamma_{\max})}\right) = \E f\left(\cal M_T^{(0)}\right) +  \OO_\prec(n^{-\tau_0}).$$
 	This shows that $\wh{\cal M}(z)$ has the same asymptotic distribution as $\wh{\cal M}^g(z)$ in the almost Gaussian case. Combining this fact with \eqref{claim sameasymp}, Proposition \ref{redGthm} and Proposition \ref{main_prop1}, we conclude \eqref{limf} when $f$ is smooth. Extension to any bounded continuous $f$ follows from a standard argument. 
 \end{proof}

\section{Proof of Theorem \ref{main_thm2}}\label{pf thm2}

In this section, we present the proof of Theorem \ref{main_thm2} based on a comparison with Theorem \ref{main_thm1}. We first truncate the entries of $X$, $Y$ and $Z$ using the moment condition \eqref{eq_8moment}. Choose a constant $c_\phi >0$ small enough such that $(n^{1/4-c_\phi})^{8+c_0} \ge n^{2+\e_0}$ and $(n^{1/4-c_\phi})^{4+c_0} \ge n^{1+\e_0}$ for a constant $\e_0>0$. Then, we introduce the following truncation on the entries of $ X , $ $  Y $ and $ Z$: 
$$ X'_{ij} = \mathbf 1_{|X_{ij}|\le n^{-1/4-c_\phi}}X_{ij},\quad Y'_{ij} = \mathbf 1_{|Y_{ij}|\le n^{-1/4-c_\phi}}Y_{ij},\quad Z'_{ij} = \mathbf 1_{|Z_{ij}|\le n^{-1/4-c_\phi}}Z_{ij}.$$
In other words, we restrict ourselves to the following event:
$$\Omega :=\left\{\max_{i,j} |X_{ij}|\le \phi_n,\max_{i,j} |Y_{ij}|\le \phi_n, \max_{i,j}|Z_{ij}|\le \phi_n\right\}, \quad \text{with}\ \ \ \phi_n:=n^{-1/4-c_\phi}.$$
Combining the condition (\ref{eq_8moment}) with Markov's inequality and using a simple union bound, we get that
\begin{equation}\label{XneX}
	\mathbb P(X' \ne X, Y' \ne Y, Z'\ne Z) =\OO ( n^{-\e_0}).
\end{equation}
Using (\ref{eq_8moment}) and integration by parts, it is easy to verify that  
\begin{align*}
	\mathbb E  \left|X_{ij}\right|1_{|X_{ij}|> \phi_n} =\OO(n^{-2-\e_0}), \quad \mathbb E \left|X_{ij}\right|^2 1_{|X_{ij}|> \phi_n} =\OO(n^{-2-\e_0}),
\end{align*}
which implies 
\be\label{IBP2}|\mathbb E  X'_{ij}| =\OO(n^{-2-\e_0}), \quad  \mathbb E |X'_{ij}|^2 = n^{-1} + \OO(n^{-2-\e_0}).\ee
Moreover, we trivially have that
$$\mathbb E  |X'_{ij}|^4 \le \mathbb E  |X_{ij}|^4 =\OO(n^{-2}).$$
Similar estimates also hold for the entries of $Y$ and $Z$. Now, we introduce the matrices 
$$\mathring X =\frac{X'-\E X'}{\text{Var}(X'_{11})}, \quad \mathring Y =\frac{Y'-\E Y'}{\text{Var}(Y'_{11})} ,\quad \mathring Z=\frac{Z'-\E Z'}{\text{Var}(Z'_{11})}.$$
Note that by \eqref{IBP2}, we have the estimates
\be\label{IBP3} 
\|\E X'\| =\OO(n^{-1-\e_0}) , \quad \text{Var}(X'_{11})= n^{-1} \left[ 1+\OO(n^{-1-\e_0})\right],
\ee
and similar estimates also hold for $\|\E Y'\|$, $\text{Var}(Y'_{11})$, $\|\E Z'\|$ and $\text{Var}(Z'_{11})$. Now, we define SCC matrices $\mathring {\cal C}_{ {\cal X} {\cal Y}}$ and $\mathring {\cal C}_{XY}$ by replacing $(X,Y,Z)$ with $(\mathring X,\mathring Y, \mathring Z)$ in \eqref{eq_SCC} and \eqref{eq_nullSCC}. 
With the estimate \eqref{IBP3}, we can readily bound the differences between the eigenvalues of $\mathring {\cal C}_{ {\cal X} {\cal Y}}$ and those of ${\cal C}_{ {\cal X} {\cal Y}}$ using  Weyl's inequality.



\begin{lemma}\label{claim compcirc}
	Under the above setting, we have that
	$$\mathbb P \left( \left\| {\cal C}_{ {\cal X} {\cal Y}} - \mathring {\cal C}_{ {\cal X} {\cal Y}}\right\| =\OO\left(n^{-1-\e_0}\right) \right) = 1-\OO\left(n^{-\e_0}\right).$$
\end{lemma} 
\begin{proof}
	This lemma is an easy consequence of \eqref{IBP3} and the singular value bounds in \eqref{op rough1add} and \eqref{op rough2add} (which hold by Theorem \ref{thm_localadd} (iv) below). Moreover, the probability bound is due to \eqref{XneX}.
\end{proof}

By the above lemma, it suffices to prove that Theorem \ref{main_thm2} holds under the following assumptions on $(X, Y,Z)$, which correspond to the above setting for $(\mathring X,\mathring Y, \mathring Z)$.

\begin{assumption}\label{main_assmadd}
	Assume that $X=(X_{ij})$, $Y=(Y_{ij})$ and $Z=(Z_{ij})$ are independent $p\times n$, $q\times n$ and $r\times n$ matrices, whose entries are real i.i.d. random variables satisfying \eqref{assm1}, \eqref{assmZ}, the bounded fourth moment condition
	\be\label{conditionA3} 
	\max\left\{ \mathbb{E} | X_{11} |^4 ,\mathbb{E} | Y_{11} |^4 ,\mathbb{E} | Z_{11} |^4 \right\} \lesssim n^{-2},
	\ee 
	and the following {\it{bounded support condition}} with $\phi_n=n^{-1/4-c_\phi}$:
	\begin{equation}
		\max\left\{ \max_{i,j}\vert X_{ij}\vert,\max_{i,j}\vert Y_{ij}\vert,\max_{i,j}\vert Z_{ij}\vert \right\}\le \phi_n.
		\label{eq_support}
	\end{equation}
	Moreover, we assume that Assumption \ref{main_assm} (iii)--(iv) hold.  
\end{assumption}

The local laws in Section \ref{sec_maintools} can be extended to the above setting. More precisely, we have proved the following theorem in \cite{PartI,PartIII}. 

\begin{theorem}\label{thm_localadd}
Suppose Assumption \ref{main_assmadd} holds. 
	\begin{itemize}
		\item[(i)] (Outliers: Theorem 2.9 of \cite{PartI}) If $t_i \ge t_c + n^{-1/3}+\phi_n$, then we have that
		\be\label{boundoutadd}
		|\wt\lambda_i -\theta_i | \prec  n^{-1/2}|t_i - t_c|^{1/2} + \phi_n|t_i - t_c| .
		\ee
		On the other hand, for any $i=\OO(1)$ with $t_i < t_c + n^{-1/3}+\phi_n$, we have that
		\be\label{boundedgeadd}
		|\wt\lambda_i - \lambda_+| \prec  n^{-2/3} + \phi_n^2.
		\ee
		
		\item[(ii)] (Anisotropic local law: 
		Theorem 3.9 of \cite{PartI}) 
		For any fixed $\e>0$ and deterministic unit vectors $\mathbf u, \mathbf v \in \mathbb C^{\mathcal I}$, the following estimate holds for all $z\in S_{out}(\epsilon)$:
		\begin{equation}\label{aniso_outstrongadd}
			\left| \langle \mathbf u, G(z) \mathbf v\rangle - \langle \mathbf u, \Pi (z)\mathbf v\rangle \right|  \prec \phi_n+n^{-1/2}(\kappa+\eta)^{-1/4} .
		\end{equation}
		
		\item[(iii)] (Eigenvalue rigidity: Theorem 2.5 of \cite{PartIII}) The eigenvalue rigidity estimate \eqref{rigidity} holds.

		\item[(iv)] (Singular value bounds: Lemma 3.3 of \cite{PartIII}) For any constant $\e>0$, the bounds \eqref{op rough1add} and \eqref{op rough2add} hold with high probability.

	\end{itemize}
	For the above results to hold,  it is not necessary to assume that the entries of $X$, $Y$ and $Z$ are identically distributed, that is, only independence and moment conditions are needed.

\end{theorem}
Moreover, Lemma \ref{largedeviation} can also be extended.

\begin{lemma}[Lemma 3.8 of \cite{EKYY1}]\label{largedeviationadd}
	Let $(x_i)$, $(y_j)$ be independent families of centered independent random variables, and $(\cal A_i)$, $(\cal B_{ij})$ be families of deterministic complex numbers. Suppose the entries $x_i$, $y_j$ have variances at most $n^{-1}$ and satisfy the bounded support condition (\ref{eq_support}). Then, the following large deviation bounds hold: 
	\begin{align*}
		 & \Big\vert \sum_i \cal A_i x_i \Big\vert \prec  \phi_n \max_{i} \left\vert \cal A_i \right\vert+ \frac{1}{\sqrt{n}}\Big(\sum_i |\cal A_i|^2 \Big)^{1/2} , \\
		&\Big\vert \sum_{i,j} x_i \cal B_{ij} y_j \Big\vert \prec \phi_n^2 \cal B_d  + \phi_n \cal B_o + \frac{1}{n}\Big(\sum_{i\ne j} |\cal B_{ij}|^2\Big)^{{1}/{2}} , \\
		 &\Big\vert \sum_{i}  \cal B_{ii} |x_i|^2 - \sum_{i} (\mathbb E|x_i|^2) \cal B_{ii}  \Big\vert  \prec \phi_n \cal B_d  ,\\ 
		& \Big\vert \sum_{i\ne j} x_i \cal B_{ij} x_j \Big\vert  \prec \phi_n\cal B_o + \frac{1}{n}\Big(\sum_{i\ne j} |\cal B_{ij}|^2\Big)^{{1}/{2}} ,
	\end{align*}
	where $\cal B_d:=\max_{i} |\cal B_{ii} |$ and $\cal B_o:= \max_{i\ne j} |\cal B_{ij}|.$
\end{lemma}  

Following the arguments in Section \ref{sec mainthm} and using Theorem \ref{thm_localadd}, we can obtain a similar equation as  \eqref{masterx4}:
\be\label{masterx4add}
\begin{split}
	&  \det \left[ f_c(\lambda) I_r -   \diag (t_1, \cdots, t_r) + \cal O^{\top} \cal E_r (\lambda) \cal O+\OO_\prec (n^{-1}+\phi_n^2) \right] =0.
\end{split}
\ee
Then, using \eqref{masterx4add} and \eqref{boundoutadd}, as in Proposition \ref{redGthm}, we can get that  
\be\label{redGadd}
\left| \left(\wt\lambda_{\al(i)} - \theta_l \right)- \mu_i \left\{ a(t_l) \left[  \diag (t_1, \cdots, t_r) - t_l - \cal O^{\top} \cal E_r (\theta_l)\cal O\right]_{\llbracket\gamma(l)\rrbracket}  \right]\right\} \prec   n^{-1/2-\e} ,
\ee
for a constant $\e>0$ depending on $c_\phi$ only. Again, the proof is the same as the one for Proposition 4.5 in \cite{KY_AOP}, so we omit the details. We also remark that this proof is the only place where we need to use the well-separation condition \eqref{well-sep}. 

With \eqref{masterx4add}, the problem is once again reduced to showing the CLT of ${\cal M}_0(\theta_l)$ in \eqref{wtMx}. Using Lemma \ref{largedeviationadd}, we can obtain a similar estimate as in \eqref{eq_iso}:
\be\label{Z_concentration2}\left\| ZZ^{\top} - I_r \right\| \prec \phi_n.\ee
Thus, similar to \eqref{V-F}, we can introduce an $n\times r$ partial orthogonal matrix $\wt Z $ such that 
\be\label{V-Fadd} \wt Z\wt Z^{\top} = I_r, \quad  \|\wt Z - Z\|_F\prec \phi_n .\ee
With \eqref{V-Fadd} and \eqref{aniso_outstrongadd}, we can check that
$$\|{\cal M}(\theta_l) - {\cal M}_0(\theta_l)\|\prec \sqrt{n}\phi_n^2 \le n^{- 2c_\phi},$$
where the matrix $\cal M$ is defined in \eqref{defM}. Thus, to prove Theorem \ref{main_thm2}, it suffices to prove the CLT for ${\cal M}(\theta_l) $. As in Section \ref{secpfmain1}, to avoid singular behaviors of the resolvent on exceptional low-probability events, we will use the regularized resolvent $\wh G(z)$ in Definition \ref{resol_not2} with $z= \theta_l +\ii n^{-4}$ throughout the rest of the proof. However, for simplicity of notations, we still use the notation $G(z)$ to denote the regularized resolvents in the following proof, while keeping in mind that the bound (\ref{op G}) holds for all resolvent entries appearing below with $\eta=n^{-4}$, and hence Lemma \ref{lem_stodomin} (iii) can be applied without worry. Finally, we remark that the rest of the proof will be conditional on $Z$ and $\wt Z$, i.e., they are regarded as deterministic matrices unless specified otherwise.

Given any random matrices $X$ and $Y$ satisfying Assumption \ref{main_assmadd}, we can construct matrices $\wt X$ and $\wt Y$, whose entries have the first four moments matching those of the entries of $X$ and $Y$, but with a smaller support $n^{-1/2}$. 

\begin{lemma} [Lemma 5.1 of \cite{LY}]\label{lem_decrease}
	Suppose $X$, $Y$ and $Z$ satisfy Assumption \ref{main_assmadd}. Then, there exist independent random matrices $\wt{X}=(\wt X_{ij})$, $\wt{Y}=(\wt Y_{ij})$ and $\wt{Z}=(\wt Z_{ij})$ satisfying Assumption \ref{main_assmadd}, 
	such that the condition (\ref{eq_support}) holds with $\phi_n$ replaced by $n^{-1/2}$. Moreover, they satisfy the following moment matching conditions: 
	\begin{equation}\label{match_moments}
		\mathbb EX_{ij}^k =\mathbb E\wt X_{ij}^k, \quad \mathbb EY_{ij}^k =\mathbb E\wt Y_{ij}^k, \quad \mathbb EZ_{ij}^k =\mathbb E\wt Z_{ij}^k, \quad k=1,2,3,4.
	\end{equation}
\end{lemma}

Note that $\wt X$, $\wt Y$ and $\wt Z$ satisfy the setting of Theorem \ref{main_thm1}. By replacing $(X,Y)$ with $(\wt X,\wt Y)$ in \eqref{linearize_block}, \eqref{def_resolvent} and \eqref{defM}, We can define $\wt H(z)$, $\wt G(z)$ and $\wt {\cal M}(z)$. In Section \ref{secpfmain1}, we have proved the CLT for $\wt{\cal M}(\theta_l)$. The rest of the proof is devoted to showing that ${\cal M}(\theta_l)$ has the same asymptotic distribution as $\wt{\cal M}(\theta_l)$. 
\begin{proposition}\label{main_propadd}
Suppose Assumption \ref{main_assmadd} holds. Let $\wt X$ and $\wt Y$ be two random matrices constructed as in Lemma \ref{lem_decrease}. Then, there exists a constant $\e>0$ such that for any function $f\in C_c^\infty(\C^{4r\times 4r})$, we have
$$ \mathbb Ef\left({\cal M}(z)\right)= \mathbb Ef(\wt{\cal M}(z)) + \OO(n^{-\e}), \quad \text{for}\quad z= \theta_l +\ii n^{-4}.$$
\end{proposition}
To prove this proposition, we will use the continuous comparison method introduced in \cite{Anisotropic}. We first introduce the following interpolation between $(X,Y)$ and $(\wt X, \wt Y)$.

\begin{definition}[Interpolating matrices]
	Introduce the notations $X^0:=\wt X $ and $X^1:=X$. Let $\rho_{i\mu}^0$ and $\rho_{i\mu}^1$ be the laws of $\wt X_{i\mu} $ and $X_{i\mu}$, respectively. For $\theta\in [0,1]$, we define the interpolated law
	$$\rho_{i\mu}^\theta := (1-\theta)\rho_{i\mu}^0+\theta\rho_{i\mu}^1.$$
	Let $\{X^\theta: \theta\in (0,1) \}$ be a collection of random matrices such that 
	for any fixed $\theta\in (0,1)$, $(X^0,X^\theta, X^1)$ is a triple of independent $\mathcal I_1\times \mathcal I_3$ random matrices, and the matrix $X^\theta=(X_{i\mu}^\theta)$ has law
	\begin{equation}\label{law_interpol}
		\prod_{i\in \mathcal I_1}\prod_{\mu\in \mathcal I_3} \rho_{i\mu}^\theta\left(\dd X_{i\mu}^\theta\right).
	\end{equation}
	Note that we do not require $X^{\theta_1}$ to be independent of $X^{\theta_2}$ for $\theta_1\ne \theta_2 \in (0,1)$. For $\lambda \in \mathbb R$, $i\in \mathcal I_1$ and $\mu\in \mathcal I_3$, we define the matrix $X_{(i\mu)}^{\theta,\lambda}$ through
	\be\label{Ximulambda} 
	\left(X_{(i\mu)}^{\theta,\lambda}\right)_{j\nu}:=\begin{cases}X_{i\mu}^{\theta}, &\text{ if }(j,\nu)\ne (i,\mu)\\ \lambda, &\text{ if }(j,\nu)=(i,\mu)\end{cases}.
	\ee
	In a similar way, we can define a collection of random matrices $\{Y^\theta: \theta\in [0,1] \}$ for $\theta\in [0,1]$ with $Y^0:=\wt Y$ and $Y^1:=Y$. We require that for any fixed $\theta\in (0,1)$, $Y^\theta$ is independent of $(X^0,X^\theta, X^1, Y^0, Y^1)$. For $\lambda \in \mathbb R$, $i\in \mathcal I_2$ and $\mu\in \mathcal I_4$, we define $Y_{(i\mu)}^{\theta,\lambda}$ in the same way as \eqref{Ximulambda}.
	We also introduce the resolvents 
	\[G^{\theta}(z):=G\left(X^{\theta},Y^{\theta},z\right),\ \ \ G^{\theta, \lambda}_{(i\mu)}(z):=\begin{cases}G\left(X_{(i\mu)}^{\theta,\lambda},Y^{\theta},z\right), & \text{ if } i\in \mathcal I_1, \mu\in \mathcal I_3 \\ G\left(X^{\theta},Y_{(i\mu)}^{\theta,\lambda},z\right), & \text{ if } i\in \mathcal I_2, \mu\in \mathcal I_4 \end{cases}.\]
\end{definition}


Using (\ref{law_interpol}) and fundamental calculus, it is easy to derive the following basic interpolation formula.
\begin{lemma}\label{lemm_comp_3}
	For any differentiable function $F:\mathbb C^{\mathcal I_1 \times\mathcal I_3} \times \mathbb C^{\mathcal I_2 \times\mathcal I_4}\rightarrow \mathbb C$, we have that
	\begin{equation}\label{basic_interp}
		\begin{split}
			\frac{\dd}{\dd\theta}\mathbb E F(X^\theta, Y^\theta)&=\sum_{i\in\mathcal I_1 , \mu\in\mathcal I_3}\left[\mathbb E F\left(X^{\theta,X_{i\mu}^1}_{(i\mu)}, Y^\theta\right)-\mathbb E F\left(X^{\theta,X_{i\mu}^0}_{(i\mu)}, Y^\theta\right)\right] \\
			&+ \sum_{i\in\mathcal I_2 , \mu\in\mathcal I_4}\left[\mathbb E F\left(X^\theta, Y^{\theta,Y_{i\mu}^1}_{(i\mu)}\right)-\mathbb E F\left(X^\theta, Y^{\theta,Y_{i\mu}^0}_{(i\mu)}\right)\right],
		\end{split}
	\end{equation}
	provided all the expectations exist.
\end{lemma}

We shall apply Lemma \ref{lemm_comp_3} to $F(X^\theta, Y^\theta) = f(\cal M\left(X^{\theta},Y^{\theta},z\right))$ for the function $f$ in Proposition \ref{main_propadd}, where $\cal M\left(X^{\theta},Y^{\theta},z\right)$ is defined by replacing $G(z)\equiv G(X,Y,z)$ with $G^{\theta}(z)\equiv G(X^\theta,Y^\theta,z)$. The main work is to show the following estimate for the right-hand side of (\ref{basic_interp}).

\begin{lemma}\label{lemm_comp_4}
	Under the assumptions of Proposition \ref{main_propadd}, there exists a constant $\e>0$ such that
	\begin{equation}\label{compxxx}
		\sum_{i\in\mathcal I_1}\sum_{\mu\in\mathcal I_3}\left[\mathbb Ef\left(\cal M\left(X^{\theta,X_{i\mu}^1}_{(i\mu)},Y^\theta\right)\right)-\mathbb Ef\left(\cal M\left(X^{\theta,X_{i\mu}^0}_{(i\mu)},Y^\theta\right)\right)\right] =\OO( n^{-\e}),
	\end{equation}
	\begin{equation}\label{compyyy}
		\sum_{i\in\mathcal I_2}\sum_{\mu\in\mathcal I_4}\left[\mathbb Ef\left(\cal M\left(X^\theta, Y^{\theta,Y_{i\mu}^1}_{(i\mu)}\right)\right)-\mathbb Ef\left(\cal M\left(X^\theta, Y^{\theta,Y_{i\mu}^0}_{(i\mu)}\right)\right)\right] =\OO(  n^{-\e}),
	\end{equation}
	for all $\theta\in[0,1]$. 
\end{lemma}
Combining Lemma \ref{lemm_comp_3} and Lemma \ref{lemm_comp_4}, we conclude Proposition \ref{main_propadd}. 
The proof of Lemma \ref{lemm_comp_4} is based on an expansion approach. As in \eqref{RESOLVENTEXPANSION} and \eqref{RESOLVENTEXPANSION2}, for any $i\in \mathcal I_1$, $\mu\in \mathcal I_3$, $\lambda,\lambda'\in \mathbb R$ and $K\in \mathbb N$, we have the resolvent expansion 
\begin{equation}\label{eq_comp_expansion}
	\begin{split}
		G_{(i \mu)}^{\theta,\lambda'} &= G_{(i\mu)}^{\theta,\lambda}+\sum_{k=1}^{K}  ({\lambda-\lambda'})^kG_{(i\mu)}^{\theta,\lambda}\left( E^{\{i,\mu\}} G_{(i\mu)}^{\theta,\lambda}\right)^k  \\
		&+ (\lambda-\lambda')^{K+1}G_{(i\mu)}^{\theta,\lambda'}\left(E^{\{i,\mu\}} G_{(i\mu)}^{\theta,\lambda}\right)^{K+1},
	\end{split}
\end{equation}
where $E^{\{i,\mu\}}$ is the matrix defined by $(E^{\{i,\mu\}})_{ab}=\mathbf 1_{(a,b)=(i,\mu)}+\mathbf 1_{(a,b)=(\mu,i)}$ as in \eqref{defnEgamma_add}. With this expansion, we can readily obtain the following estimate: if $y$ is a random variable satisfying $|y|\le \phi_n$, then for any deterministic unit vectors $\mathbf u,\mathbf v \in \mathbb C^{\mathcal I}$, we have that
\begin{equation}\label{comp_eq_apriori}
	\left\langle \bu,  \left[G_{(i\mu)}^{\theta,y}(z)-\Pi(z)\right]\bv\right\rangle \prec \phi_n  ,\quad \text{for} \quad 
	z= \theta_l +\ii n^{-4}.
\end{equation}
In fact, to prove this estimate, we will apply the expansion (\ref{eq_comp_expansion}) for a sufficiently large $K$, say $K=100$, with $\lambda'=y$ and $\lambda = X_{i\mu}^\theta$, so that \smash{$ G_{(i\mu)}^{\theta,\lambda}=G^\theta$}. Then, to bound the resulting expansion on the right-hand side of \eqref{eq_comp_expansion}, we will use $y\le \phi_n$, \smash{$|X_{i\mu}^\theta|\le \phi_n$}, the anisotropic local law \eqref{aniso_outstrongadd} for \smash{$G^{\theta}$}, and the rough bound in (\ref{op G}) for \smash{$G_{(i\mu)}^{\theta,y}$} in the last term.

\begin{proof}[Proof Lemma \ref{lemm_comp_4}]
	We only give the proof of \eqref{compxxx}, while \eqref{compyyy} obviously can be proved in the same way.  
	For simplicity of notations, we only provide the proof for a simpler version of \eqref{compxxx},
	\begin{equation}\label{compxxx2}
		\sum_{i\in\mathcal I_1}\sum_{\mu\in\mathcal I_3}\left[\mathbb Ef\left(M\left(X^{\theta,X_{i\mu}^1}_{(i\mu)},Y^\theta\right)\right)-\mathbb Ef\left(M\left(X^{\theta,X_{i\mu}^0}_{(i\mu)},Y^\theta\right)\right)\right] =\OO( n^{-\e}),
	\end{equation}
	where $M$ is defined as
	$$M(X,Y):= \sqrt{n} \langle \bu, (G(X,Y,z) - \Pi(z))\bv\rangle $$
	for some deterministic unit vectors $\bu,\bv\in \C^{\cal I}$ satisfying that  
	\be\label{delocaluv}
	\max_{\mu\in \cal I_3\cup \cal I_4}|u(\mu)| \prec \phi_n,\quad \max_{\mu\in \cal I_3\cup \cal I_4}|v(\mu)| \prec \phi_n.
	\ee
	The proof for \eqref{compxxx} is the same, except that we need to use multivariable Taylor expansions. 
	Here, the condition \eqref{delocaluv} is due to the corresponding bound on $\wt Z$,
	$$\|\wt Z\|_{\max}\le \|\wt Z-Z\|_{\max} + \|\wt Z\|_{\max} \prec \phi_n$$
	by \eqref{V-Fadd} and the bounded support condition in \eqref{eq_support}.
	
	In the following proof, for simplicity of notations, we fix a $\theta\in [0,1]$ and denote \smash{$M_{(i\mu)}(\lambda):= M(X^{\theta,\lambda}_{(i\mu)} ) $} while ignoring $Y^\theta$ from the argument. Recall that $\phi_n=n^{-1/4-c_\phi}$. Using \eqref{eq_comp_expansion} with $K=9$ and the local law \eqref{comp_eq_apriori}, we get that for a random variable $y$ satisfying $|y|\le \phi_n$,
	\be\label{My-0}M_{(i\mu)}(y)-M_{(i\mu)}(0)=  \sum_{k=1}^9  n^{1/2} (-y)^k x_k(i,\mu) + \OO_\prec ( n^{-2 - 10 c_\phi} ),\ee
	where 
	$$x_k(i,\mu) := \big\langle \bu, G_{(i\mu)}^{\theta,0}\big( E^{\{i,\mu\}} G_{(i\mu)}^{\theta,0}\big)^k\bv \big\rangle.$$
	By \eqref{comp_eq_apriori}, we have $x_k (i,\mu)\prec 1$ for $k\ge 1$. On the other hand, for $k=1$, using \eqref{comp_eq_apriori} and \eqref{delocaluv}, we can get a better bound
	\be\label{boundxk}
	x_1(i,\mu)= \big\langle \bu, G_{(i\mu)}^{\theta,0} E^{\{i,\mu\}} G_{(i\mu)}^{\theta,0} \bv \big\rangle= \big\langle \bu, \Pi E^{\{i,\mu\}} \Pi \bv \big\rangle + \OO_\prec( \phi_n) \prec  \phi_n.
	\ee
	Combining this bound with $|y| \le \phi_n$, we immediately obtain from \eqref{My-0} the rough bound 
	\be\label{My-02}
	M_{(i\mu)}(y)-M_{(i\mu)}(0)\prec n^{1/2}\phi_n^2 \le n^{-2c_\phi}.\ee
	
	Now, fix an integer $K\ge 1/c_\phi$. Using \eqref{My-0} and \eqref{My-02}, the Taylor expansion of $f$ up to the $K$-th order gives that for $\al\in\{0,1\}$,
	\begin{align*}
		 &\E f\left(M_{(i\mu)}(X_{i\mu}^{\al}) \right)-\mathbb E f \left(M_{(i\mu)}(0) \right)  \\
		 &  =\sum_{k=1}^{K} \E\frac{f^{(k)} \left(M_{(i\mu)}(0) \right)}{k!} \left[\sum_{l=1}^9  n^{1/2} (-X_{i\mu}^{\al})^l x_l(i,\mu) \right]^k  +\OO_\prec\left(n^{-2 - 2c_\phi }\right)\\
		 &= \sum_{k=1}^{K} \sum_{s=1}^{K+2k}\sum_{\mathbf s } ^* n^{k/2} \mathbb E(-X_{i\mu}^{\al})^s \E\frac{f^{(k)} \left(M_{(i\mu)}(0) \right)}{k!} \prod_{l=1}^k  x_{s_l}(i,\mu)   +\OO_\prec\left(n^{-2 - 2c_\phi }\right),
	\end{align*}
	where $\sum_{\mathbf s}^*$ means the sum over $\mathbf s =(s_1, \cdots, s_k)\in \N^k$ satisfying 
	\be\label{lkstat}
	1\le s_{i} \le 9, \quad \sum_{l=1}^k l \cdot s_l = s. 
	\ee
	Here, for the terms with $s>K+2k$, we have $n^{k/2} \mathbb E(-X_{i\mu}^{\al})^s  \le n^{-2 - 2c_\phi }$, so they are included into the error. 
	%
	Now, using the moment matching condition \eqref{match_moments}, we get that
	%
	\begin{align*}
		&\left|\E f\left(M_{(i\mu)}(X_{i\mu}^{1}) \right)-\E f\left(M_{(i\mu)}(X_{i\mu}^{0}) \right)\right|   \prec \sum_{k=1}^{K} \sum_{s=5}^{K+2k}\sum_{\mathbf s } ^* n^{k/2 - 2}\phi_n^{s-4}   \E\Big| \prod_{l=1}^k  x_{s_l}(i,\mu)\Big|   + n^{-2 - 2c_\phi } ,
	\end{align*}
	where we used that $\mathbb E |X_{i\mu}^{\al}|^s \le \phi_n^{s-4}\mathbb E |X_{i\mu}^{\al}|^4 \lesssim \phi_n^{s-4}n^{-2}$ for $s\ge 5$. Thus, to show \eqref{compxxx2}, we only need to prove that for any fixed $s\ge 5$ and $\mathbf s\in \N^k$ satisfying \eqref{lkstat},
	\begin{equation}\label{eq_comp_est}
		\sum_{i\in\mathcal I_1}\sum_{\mu\in\mathcal I_3}n^{k/2 - 2}\phi_n^{s-4}   \E\Big| \prod_{l=1}^k  x_{s_l}(i,\mu)\Big| \prec n^{-\e}\end{equation}
	for some constant $\e>0$. For the proof of \eqref{eq_comp_est}, we will consider three different cases.

	\vspace{5pt}
	
	\noindent{\bf Case 1:} Suppose $s_l \ge 2$ for all $ l = 1,\cdots, k$. Then, we have $s\ge \max\{2k,5\}$ and
	\be\label{n<2k}n^{k/2 - 2}\phi_n^{s-4} = n^{-2+k/2 - (s-4)/4 }n^{-(s-4)c_\phi} \le n^{-1-c_\phi}. \ee
	On the other hand, using \eqref{eq_comp_expansion} with $K=0$ and \eqref{comp_eq_apriori}, we get that
	\be\label{Gui}\begin{split} \big|\big\langle \mathbf e_i,G_{(i \mu)}^{\theta,0} \bu\big\rangle\big| & \le   \left|G^\theta_{i\bu }\right| + \left|X_{i\mu}^\theta\right| \big(\big|\big\langle \mathbf e_i,G_{(i \mu)}^{\theta,0} \mathbf e_{i}\big\rangle \big| \left| G ^{\theta}_{\mu \bu} \right|+\big|\big\langle \mathbf e_i,G_{(i \mu)}^{\theta,0} \mathbf e_{\mu}\big\rangle \big| \left| G ^{\theta}_{i\bu} \right| \big)  \\
	&\prec  \left|G^\theta_{i\bu }\right| + \phi_n \left|G^\theta_{\mu\bu}\right|.
	\end{split}
	\ee
	Similarly, we have that 
	\be\label{Gumu} \big|\big\langle \mathbf e_\mu,G_{(i \mu)}^{\theta,0} \bu\big\rangle\big|  \prec  \left|G^\theta_{\mu\bu }\right| + \phi_n \left|G^\theta_{i\bu}\right|.
	\ee
	Inserting \eqref{Gui} and \eqref{Gumu} into the definition of $x_l(i,\mu)$, we immediately get that
	\be\label{Rxl}
	|x_{l}(i,\mu)| \prec \left|G^\theta_{i\bu }\right|^2 + \left|G^\theta_{i\bv }\right|^2 + \left|G^\theta_{\mu\bu }\right|^2 + \left|G^\theta_{\mu\bv}\right|^2,\quad l\ge 1.
	\ee
	We claim that for any deterministic unit vector $\bu \in \C^{\cal I}$,
	\be\label{claim Ward}
	\sum_{i\in \cal I_1} \left|G^\theta_{i\bu }\right|^2 \prec 1, \quad  \sum_{\mu\in \cal I_3} \left|G^\theta_{\mu\bu }\right|^2 \prec 1.
	\ee
	We postpone its proof until we complete the proof of Lemma \ref{lemm_comp_4}. Combining \eqref{n<2k}, \eqref{Rxl} and \eqref{claim Ward}, we can bound that 
	\begin{align*}
		& \sum_{i\in\mathcal I_1}\sum_{\mu\in\mathcal I_3}n^{k/2 - 2}\phi_n^{s-4}   \E\Big| \prod_{l=1}^k  x_{s_l}(i,\mu)\Big| \\
		&\prec  \sum_{i\in\mathcal I_1}\sum_{\mu\in\mathcal I_3} n^{-1 - c_\phi} \left(\left|G^\theta_{i\bu }\right|^2 + \left|G^\theta_{i\bv }\right|^2 + \left|G^\theta_{\mu\bu }\right|^2 + \left|G^\theta_{\mu\bv}\right|^2\right) \prec n^{-c_\phi} .
	\end{align*}
	
	
	\noindent{\bf Case 2:} Suppose there are at least two $l$'s such that $s_l =1$. Without loss of generality, we assume that $s_1=s_2 = \cdots = s_j =1$ for some $2\le j\le k$. Then, we have 
	$ s \ge \max\{ 2k -j,5\} $, which gives that 
	\begin{align}\label{s>2k}
		n^{k/2 - 2}\phi_n^{s-4}   \E\Big| \prod_{l=1}^k  x_{s_l}(i,\mu)\Big| \prec n^{k/2 - 2}\phi_n^{s-4} \phi_n^{j-2} |x_1(i,\mu)|^2 \le n^{-1/2-c_\phi} |x_1(i,\mu)|^2,
	\end{align}
	where in the second step we used
	$$ n^{k/2 - 2}\phi_n^{s+j-6} = n^{-2+k/2 - (s+j-6)/4}n^{-(s+j-6)c_\phi}\le n^{-1/2-c_\phi}.$$
	Applying \eqref{Gui} and \eqref{Gumu} to \eqref{boundxk}, we can bound that
	\begin{align}
		|x_1(i,\mu)|  & \prec \left(\left|G^\theta_{i\bu }\right| +  \phi_n \left|G^\theta_{\mu\bu}\right|\right)\left(\left|G^\theta_{\mu\bv}\right| +  \phi_n \left|G^\theta_{i\bv}\right|\right)   + \left(\left|G^\theta_{\mu\bu }\right| +  \phi_n \left|G^\theta_{i\bu}\right|\right)\left(\left|G^\theta_{i\bv}\right| +  \phi_n \left|G^\theta_{\mu\bv}\right|\right) \nonumber\\
		& \lesssim \left|G^\theta_{i\bu }\right|\left|G^\theta_{\mu\bv}\right| + \left|G^\theta_{\mu\bu }\right|\left|G^\theta_{i\bv}\right|  + \phi_n \left(\left|G^\theta_{i\bu }\right|^2 + \left|G^\theta_{i\bv }\right|^2 + \left|G^\theta_{\mu\bu }\right|^2 + \left|G^\theta_{\mu\bv}\right|^2\right).\label{Rxl1}
	\end{align}
	Now, using \eqref{claim Ward} and \eqref{Rxl1}, we get that 
	\be\nonumber
	\begin{split} 
		 &\sum_{i\in\mathcal I_1}\sum_{\mu\in\mathcal I_3}  |x_1(i,\mu)|^2  \\
		&\prec  \sum_{i\in\mathcal I_1}\sum_{\mu\in\mathcal I_3}\left[ \left|G^\theta_{i\bu }\right|^2\left|G^\theta_{\mu\bv}\right|^2 +  \left|G^\theta_{\mu\bu }\right|^2\left|G^\theta_{i\bv}\right|^2 + \phi_n^2 \left(\left|G^\theta_{i\bu }\right|^4 + \left|G^\theta_{i\bv }\right|^4 + \left|G^\theta_{\mu\bu }\right|^4 + \left|G^\theta_{\mu\bv}\right|^4\right)\right] \\
		&  \prec 1+ n\phi_n^2 .
	\end{split}
	\ee
	Combining this bound with \eqref{s>2k}, we get that
	$$\sum_{i\in\mathcal I_1}\sum_{\mu\in\mathcal I_3}n^{k/2 - 2}\phi_n^{s-4}  \E\Big| \prod_{l=1}^k  x_{s_l}(i,\mu)\Big|  \prec   n^{-1/2-c_\phi} \cdot n\phi_n^2 \le n^{-3c_\phi}.$$
	
	
	\noindent{\bf Case 3:} Finally, suppose there is only one $l$ such that $s_l =1$. Without loss of generality, we assume that $s_1=1$ and $s_l\ge 2$ for $  l=2,\cdots, k$. Thus, we have 
	$ s \ge \max\{2k- 1,5\}$, which gives that
	\begin{align}
		 & n^{k/2 - 2}\phi_n^{s-4}   \E\Big| \prod_{l=1}^k  x_{s_l}(i,\mu)\Big| \nonumber \\
		&\prec n^{k/2 - 2}\phi_n^{s-4}  |x_1(i,\mu)| \left(   \left|G^\theta_{i\bu }\right|^2 + \left|G^\theta_{i\bv }\right|^2 + \left|G^\theta_{\mu\bu }\right|^2 + \left|G^\theta_{\mu\bv}\right|^2\right) \nonumber\\
		& \le n^{-3/4-c_\phi}  \left( \left|G^\theta_{i\bu }\right|\left|G^\theta_{\mu\bv}\right| + \left|G^\theta_{\mu\bu }\right|\left|G^\theta_{i\bv}\right| \right)\left(  \left|G^\theta_{i\bu }\right|^2 + \left|G^\theta_{i\bv }\right|^2 + \left|G^\theta_{\mu\bu }\right|^2 + \left|G^\theta_{\mu\bv}\right|^2\right) \nonumber  \\
		&\quad + n^{-3/4-c_\phi}\phi_n   \left(  \left|G^\theta_{i\bu }\right|^4 + \left|G^\theta_{i\bv }\right|^4 + \left|G^\theta_{\mu\bu }\right|^4 + \left|G^\theta_{\mu\bv}\right|^4\right), \label{Rxl3}
	\end{align}
	where  in the first step we used \eqref{Rxl}, and in the second step we used \eqref{Rxl1} and
	$$ n^{k/2 - 2}\phi_n^{s-4}= n^{-2+k/2- (s-4)/4 }n^{-(s-4)c_\phi}\le n^{-3/4-c_\phi}.$$
	Applying \eqref{claim Ward} and Cauchy-Schwarz inequality to \eqref{Rxl3}, we get that
	$$\sum_{i\in\mathcal I_1}\sum_{\mu\in\mathcal I_3}n^{k/2 - 2}\phi_n^{s-4}  \E\Big| \prod_{l=1}^k  x_{s_l}(i,\mu)\Big|  \prec n^{-2c_\phi}.$$
	
	\vspace{5pt}
	
	Combining the above three cases, we conclude \eqref{eq_comp_est} with $\e=c_\phi$, which further implies \eqref{compxxx2}. With similar arguments, we can conclude \eqref{compxxx} and \eqref{compyyy}. 
\end{proof}
\begin{proof}[Proof of \eqref{claim Ward}]
\eqref{claim Ward} is a simple corollary of the spectral decomposition of the resolvent in \eqref{spectral1}. 
	Using the rigidity estimate \eqref{rigidity} given by Theorem \ref{thm_localadd} (iii), we get that 
	\be\label{eigen_gap}
	\min_{1\le k \le p}|\lambda_k -z|\gtrsim 1,\quad \text{for}\quad z=\theta_l + \ii n^{-4}.
	\ee
	Combining it with the SVD \eqref{spectral1}, we see that $\|R(z)\|=\OO(1)$ with high probability. Then, using \eqref{GL1}--\eqref{GLR1} and \eqref{op rough1add}--\eqref{op rough2add} given by Theorem \ref{thm_localadd} (iv), we obtain that 
	$\|G(z)\|=\OO(1)$ with high probability. Thus, we have that for any unit vector $\bu \in \C^{\cal I}$, 
	\be 
	\sum_{\fa\in \cal I } \left|G_{\fa \bu }\right|^2 \le \|GG^*\|=\OO(1) \quad \text{with high probability},
	\ee
	where $G^*$ denotes the conjugate transpose of $G$. 
	If $G\equiv \wh G$ is the regularized resolvent, then we can apply Claim \ref{removehat} to get that 
	$$\sum_{\fa\in \cal I } \left|\wh G_{\fa \bu }\right|^2=\OO(1) \quad \text{with high probability}.$$
	The above argument also works for the resolvent $G^\theta$, which concludes \eqref{claim Ward}.
\end{proof}

Finally, we can complete the proof of Theorem \ref{main_thm2} using Proposition \ref{main_propadd}.
\begin{proof}[Proof of Theorem \ref{main_thm2}]
	First, suppose $X$, $Y$ and $Z$ satisfy Assumption \ref{main_assmadd}, and let $\wt X$, $\wt Y$ and $\wt Z$ be random matrices constructed in Lemma \ref{lem_decrease}. Then, Theorem \ref{main_thm2} holds for the SCC matrix defined with $(\wt X,\wt Y,\wt Z)$, because they satisfy the assumptions of Theorem \ref{main_thm1}. By Proposition \ref{main_propadd} (recall that it is proved for the regularized resolvents following the convention stated above Lemma \ref{lem_decrease}), we have that 
	$$\wh{\cal M}(z) \stackrel{d}{\sim}\wh{\wt{\cal M}}(z).$$ 
	By the argument in the proof of Claim \ref{removehat}, this implies that ${\cal M}(\theta_l)$ and $\wt{\cal M}(\theta_l)$ also have the same asymptotic distribution. Moreover, by classical CLT, the asymptotic distribution of $\sqrt{n} \left(ZZ^{\top} - I_r\right)$ is still given by \eqref{CLTe1}, which only depends on the first four moments of $Z$ entries.  Hence, by \eqref{redGadd}, we can conclude Theorem \ref{main_thm2} for the SCC matrix defined with $(X, Y, Z)$ satisfying Assumption \ref{main_assmadd}. Finally, using the cut-off argument at the beginning of this section and Lemma \ref{claim compcirc}, we conclude Theorem \ref{main_thm2}.
\end{proof}

\appendix

\section{Proof of Lemma \ref{asymp_Gauss}} \label{appd asymG}

In this section, we provide a proof of Lemma \ref{asymp_Gauss} using the Stein's method and cumulant expansions. With a slight abuse of notation, we consider the following $r\times r$ matrix
$$Q:=\sqrt{n} U^{\top} Y V + \sqrt{n} V^{\top} Y^{\top} U +  \sqrt{n} O^{\top}  ( 1- \E) (YY^{\top}) O, $$
where $Y$ is a $\rho\times n $ random matrix with i.i.d.\;entries satisfying \eqref{assm1} and \eqref{eq_highmoment}, 
$U$ and $O$ are two $\rho\times r$ deterministic matrices satisfying $\|U\|\le 1$ and $\|O\|\le 1$, and $V$ is an $n\times r$ deterministic matrix satisfying $\|V\|\le 1$ and
\be\label{Vdelocal}\|V\|_{\max}\le n^{-c}\ee
for some constant $0<c<1/2$. Moreover, we assume that $r=\OO(1)$ and $\rho=\OO(n^{\tau})$ for a small enough constant $\tau>0$. Then, we claim that $Q$ is asymptotically Gaussian with zero mean. Note that the items (i)--(iv) of Lemma \ref{asymp_Gauss} all follow from this general claim. In particular, if the entries of $Y$ are i.i.d. Gaussian, then the condition \eqref{Vdelocal} is not necessary, because  we can rotate $V$ as $YV\mapsto (Y O_n) (O_n^{\top} V)$, where the orthogonal matrix $O_n$ is chosen such that \eqref{Vdelocal} holds for $O_n^{\top} V$ and the distribution of $Y$ is unchanged: \smash{$Y O_n\stackrel{d}{=}Y$}. 

It is trivial to see that $\E Q=0$. To show that $Q$ is asymptotically Gaussian, with the Cram{\'e}r-Wold device, we need to prove that 
$$Q_{\Lambda}:=\sum_{a\le b}\lambda_{ab}Q_{ab}$$
is asymptotically Gaussian for any fixed vector of parameters denoted by $\Lambda:=(\lambda_{ab})_{a\le b}$. For this purpose, we use the Stein's method \cite{Stein}, 
i.e. we will show that for any $f\in C_c^\infty(\R)$, 
\be\label{stein}
\mathbb E Q_{\Lambda} f (Q_{\Lambda}) = s_{\Lambda}^2\E f' (Q_{\Lambda}) + \oo(1)   
\ee
for some deterministic parameter $s_{\Lambda}^2$.
This gives the CLT for $\sqrt{n}\sum_{a \le b} \lambda_{ab} Q_{ab}$, which implies that $ Q$ converges weakly to a centered Gaussian matrix, whose covariances can be determined through $s_{\Lambda}^2$. 

For simplicity, we denote $X:=\sqrt{n}Y$, such that the entries of $X$ are i.i.d. random variables with mean zero and variance one. Moreover, for any fixed $l\in \mathbb N$, there is a constant $\mu_l>0$ such that
\begin{equation}\label{assmX1}
	\mathbb E|X_{11}|^l \le \mu_l.
\end{equation}
We will prove \eqref{stein} with the following cumulant expansion formula, whose proof can be found in \cite[Proposition 3.1]{Cumulant1} and \cite[Section II]{Cumulant2}. 

\begin{lemma}\label{lemma_add_cumu}
	Let $f\in C^{l+1}(\mathbb R)$ for some fixed $l\in \N$. Suppose $\xi$ is a centered random variable whose first $l+2$ moments are finite. Let $\kappa_k(\xi)$ be the $k$-th cumulant of $\xi$. Then, we have that
	\be\label{cumuexp}
	\E [\xi f(\xi)] = \sum_{k=1}^l \frac{\kappa_{k+1}(\xi)}{k!} \E f^{(k)}(\xi) + \mathcal E_l,
	\ee 
	where the error term satisfies that for any $\chi>0$,
	\be\label{errorel} | \mathcal E_l| \le C_l \mathbb E \left[|\xi|^{l+2}\right] \sum_{|t|\le \chi}|f^{(l+1)}(t)| + C_l \mathbb E \left[|\xi|^{l+2}\mathbf 1(|\xi|>\chi)\right] \sum_{t\in \R}|f^{(l+1)}(t)| .\ee
\end{lemma}

We now expand the left-hand side of \eqref{stein} as
\begin{align}
	\mathbb E Q_{\Lambda} f (Q_{\Lambda}) &= \mathbb E \sum_{a\le b} \lambda_{ab}  \sum_{1\le i\le \rho,1\le \mu \le n} X_{i\mu}(U_{ia}V_{\mu b}+ U_{i b}V_{\mu a})   f (Q_{\Lambda}) \nonumber\\
	&+ \mathbb E \sum_{a\le b} \lambda_{ab}   \sum_{1\le i,j \le \rho}\frac1{\sqrt{n}}\sum_{1\le \mu \le n} (X_{i\mu}X_{j\mu} - \delta_{ij}) O_{ia}O_{jb}  f (Q_{\Lambda}) .\label{eq_cumu_expan}
\end{align}
We first study the first term on the right-hand side of \eqref{eq_cumu_expan}. For any fixed $a\le b$, we apply the expansion \eqref{cumuexp} with $\xi = X_{i\mu}$ and $l=2$ to get that
\begin{align}
	& \sum_{1\le i\le \rho,1\le \mu \le n} U_{ia}V_{\mu b} \mathbb E_{X_{i\mu}} \left[X_{i\mu} f (Q_{\Lambda})\right]   =\sum_{1\le i\le \rho,1\le \mu \le n}  U_{ia}V_{\mu b}  \mathbb E_{X_{i\mu}} \left[ \frac{\partial Q_{\Lambda}}{\partial X_{i\mu}}f' (Q_{\Lambda})\right] \nonumber\\
	&+ \frac{\kappa_3 }{2}\sum_{1\le i\le \rho,1\le \mu \le n}  U_{ia}V_{\mu b}  \mathbb E_{X_{i\mu}}\left[2 \sum_{a'\le b'} \lambda_{a'b'} \frac{O_{ia'}O_{ib'}}{\sqrt{n}}   f' (Q_{\Lambda}) + \left(\frac{\partial Q_{\Lambda}}{\partial X_{i\mu}}\right)^2f'' (Q_{\Lambda}) \right] \nonumber\\
	&+ \mathcal E_2 (X_{i\mu}),\label{Cumusingle}
\end{align}
where $\kappa_3\equiv \kappa_3(X_{i\mu})$ is the third cumulant of $X_{i\mu}$, $\mathcal E_2 (X_{i\mu})$ satisfies \eqref{errorel} for the function $ f (Q_{\Lambda}(X_{i\mu}))$, and
\be\label{bound1Q}
\begin{split}
	\frac{\partial Q_{\Lambda}}{\partial X_{i\mu}} &= \sum_{a'\le b'} \lambda_{a'b'} \left[(U_{ia'}V_{\mu b'}+ U_{i b'}V_{\mu a'}) + \sum_{1\le j \le \rho}\frac{X_{j\mu}}{\sqrt{n}}  ( O_{ia'}O_{jb'} + O_{ja'}O_{ib'} )\right]  \prec n^{-c}  .
\end{split}\ee
Here, we used \eqref{Vdelocal} in the second step. The expectation of the first term on the right-hand side of \eqref{Cumusingle} is 
\begin{align}
	& \mathbb E\sum_{1\le i\le \rho,1\le \mu \le n}  U_{ia}V_{\mu b}   \frac{\partial Q_{\Lambda}}{\partial X_{i\mu}}f '(Q_{\Lambda}) \nonumber\\
	& = \sum_{a'\le b'}  \lambda_{a'b'} \sum_{1\le i\le \rho,1\le \mu \le n}  U_{ia}V_{\mu b}  (U_{ia'}V_{\mu b'}+ U_{i b'}V_{\mu a'}) \mathbb Ef '(Q_{\Lambda})+ \OO_\prec(n^{-1/2+\tau}),\label{1stterm}
\end{align}
where we used Lemma \ref{largedeviation} to bound that
\be\label{largeVmu}\Big|\sum_{1\le \mu \le n} n^{-1/2}V_{\mu b}  X_{j\mu}\Big|\prec n^{-1/2}\Big(\sum_\mu |V_{\mu b}|^2\Big)^{1/2} \le n^{-1/2} .\ee
Next, using \eqref{bound1Q} and $\rho=\OO(n^{\tau})$, we can bound that
\begin{align}
	 &\mathbb E \sum_{1\le i\le \rho,1\le \mu \le n}  U_{ia}V_{\mu b}  \left(\frac{\partial Q_{\Lambda}}{\partial X_{i\mu}}\right)^2 f'' (Q_{\Lambda}) \nonumber\\
	 & \prec n^{-c}  \sum_{a'\le b'} \sum_{1\le i\le \rho,1\le \mu \le n}  |U_{ia}||V_{\mu b}|   (|U_{ia'}||V_{\mu b'}|+ |U_{i b'}||V_{\mu a'}|) \nonumber \\
	& + n^{-c}  \sum_{a'\le b'}  \sum_{1\le i,j\le \rho,1\le \mu \le n}  |U_{ia}||V_{\mu b}|   \frac1{\sqrt{n}}   =\OO(n^{-c+\tau/2}), \label{3rdterm}
\end{align}
where we used Cauchy-Schwarz inequality in the second step. 
Finally, we bound $\cal E_2$ by taking $\chi = n^{\e}$ for a small constant $\e>0$. We need to bound
\begin{align}\nonumber
	\frac{\partial^3 f (Q_{\Lambda}) }{\partial X_{i\mu}^3} = 4 \sum_{a'\le b'} \lambda_{a'b'} \frac{O_{ia'}O_{ib'}}{\sqrt{n}} \frac{\partial Q_{\Lambda}}{\partial X_{i\mu}}   f'' (Q_{\Lambda}) + \left(\frac{\partial Q_{\Lambda}}{\partial X_{i\mu}}\right)^3f''' (Q_{\Lambda}).
\end{align}
Using the compact support condition of $f$, it is easy to check that 
\begin{align}\nonumber
	\sup_{|X_{i\mu}|\le n^\e} \left| \frac{\partial^3 f (Q_{\Lambda}) }{\partial X_{i\mu}^3}\right|\lesssim &\ \sum_{a\le b} \frac1{\sqrt{n}}\left(\frac{n^\e}{\sqrt n} +\frac{\sum_{j\ne i }|X_{j\mu}|}{\sqrt n}+ |V_{\mu a}| + |V_{\mu b}|\right) \\
	&+  \sum_{a\le b} \left(\frac{n^\e}{\sqrt n} +\frac{\sum_{j\ne i }|X_{j\mu}|}{\sqrt n}+ |V_{\mu a}| + |V_{\mu b}|\right)^3 ,\nonumber
\end{align}
and
\begin{align}\nonumber
	\sup_{X_{i\mu}\in \R } \left| \frac{\partial^3 f (Q_{\Lambda}) }{\partial X_{i\mu}^3}\right| =\OO(1). 
\end{align}
On the other hand, applying Markov's inequality to \eqref{assmX1}, we obtain the bound 
$$\mathbb E \left[|X_{i\mu}|^{4}\mathbf 1(|X_{i\mu}|>n^\e)\right] \le n^{-D} \quad \text{for any constant $D>0$}.$$
Combining the above three estimates, we obtain that
\begin{align}
	\left| \mathcal E_2 (X_{i\mu})\right| & \lesssim  \mathbb E\sum_{1\le i\le \rho,1\le \mu \le n} \left| U_{ia}\right| \left|V_{\mu b}\right|  \sum_{a'\le b'} \frac1{\sqrt{n}}\left(\frac{n^\e}{\sqrt{n}}+\frac{\sum_{j \ne i}|X_{j\mu}|}{\sqrt n} + |V_{\mu a'}| + |V_{\mu b'}|\right)  \nonumber\\
	&+\mathbb E\sum_{1\le i\le \rho,1\le \mu \le n}  \left|U_{ia}\right|\left|V_{\mu b}\right|    \sum_{a'\le b'} \left(\frac{n^\e}{\sqrt{n}}+\frac{\sum_{j\ne i }|X_{j\mu}|}{\sqrt n} + |V_{\mu a'}| + |V_{\mu b'}|\right)^3 + n^{-D} \nonumber\\
	&\lesssim n^{-2c+\tau/2},\label{4thterm}
\end{align}
where we used \eqref{Vdelocal} in the second step. Now, plugging \eqref{1stterm}, \eqref{3rdterm} and \eqref{4thterm} into \eqref{Cumusingle}, we obtain that
\begin{align}
	 &\mathbb E \sum_{1\le i\le r,1\le \mu \le n} U_{ia}V_{\mu b} X_{i\mu} f (Q_{\Lambda}) \nonumber\\
	 &= \sum_{a'\le b'}  \lambda_{a'b'} \sum_{1\le i\le \rho,1\le \mu \le n}  U_{ia}V_{\mu b}  (U_{ia'}V_{\mu b'}+ U_{i b'}V_{\mu a'})  \mathbb Ef' (Q_{\Lambda}) \nonumber\\
	 &+  \kappa_3  \sum_{a'\le b'} \lambda_{a'b'} \sum_{1\le i\le \rho,1\le \mu \le n}  U_{ia}V_{\mu b} \frac{O_{ia'}O_{ib'}}{\sqrt{n}} \mathbb Ef' (Q_{\Lambda})+ \OO (n^{-c+\tau/2}).\label{firstIBP}
\end{align}

Then, we calculate the second term on the right-hand side of \eqref{eq_cumu_expan}. For any $a\le b$, we need to study
\begin{align} \nonumber
	\sum_{1\le i,j \le \rho, 1\le \mu \le n}\frac1{\sqrt{n}}  O_{ia}O_{jb} \mathbb E_{X_{i\mu}}\left[(X_{i\mu}X_{j\mu} - \delta_{ij})  f (Q_{\Lambda})\right] .
\end{align}
We only consider the hardest case with $i=j$, and the $i\ne j$ case can be handled in a similar way. For any fixed $1\le i \le \rho$, we apply the expansion \eqref{cumuexp} with $\xi = X_{i\mu}$ and $l=3$ to get that
\begin{align} 
	&  \sum_{ 1\le \mu \le n}\frac1{\sqrt{n}}  \mathbb E_{X_{i\mu}}\left[X_{i\mu}X_{i\mu} f (Q_{\Lambda}) - f (Q_{\Lambda})   \right] = \frac1{\sqrt{n}}  \sum_{ 1\le \mu \le n}\mathbb E_{X_{i\mu}} X_{i\mu} \frac{\partial Q_{\Lambda}}{\partial X_{i\mu}}f' (Q_{\Lambda}) \nonumber\\
	& +\frac{ \kappa_3  }{2\sqrt{n}}  \sum_{ 1\le \mu \le n} E_{X_{i\mu}} \left[2 \frac{\partial Q_{\Lambda}}{\partial X_{i\mu}}f' (Q_{\Lambda}) + C_iX_{i\mu}f' (Q_{\Lambda})  +X_{i\mu} \left(\frac{\partial Q_{\Lambda}}{\partial X_{i\mu}}\right)^2f'' (Q_{\Lambda})\right] \nonumber\\
	&  +\frac{ \kappa_4   }{6\sqrt{n}}  \sum_{ 1\le \mu \le n} E_{X_{i\mu}} \left[ 3C_if' (Q_{\Lambda}) +3 \left(\frac{\partial Q_{\Lambda}}{\partial X_{i\mu}}\right)^2f'' (Q_{\Lambda})  + 3C_iX_{i\mu}\frac{\partial Q_{\Lambda}}{\partial X_{i\mu}}f'' (Q_{\Lambda}) \right] \nonumber\\
	& +\frac{ \kappa_4 }{6\sqrt{n}}  \sum_{ 1\le \mu \le n} E_{X_{i\mu}} \left[ X_{i\mu} \left(\frac{\partial Q_{\Lambda}}{\partial X_{i\mu}}\right)^3 f''' (Q_{\Lambda})\right] + \cal E_3(X_{i\mu}), \label{Cumudouble}
\end{align}
where $\kappa_4\equiv \kappa_4(X_{i\mu})$ is the fourth cumulant of $X_{i\mu}$, $\mathcal E_3 (X_{i\mu})$ satisfies \eqref{errorel} for the function $X_{i\mu} f (Q_{\Lambda}(X_{i\mu}))$, and we have abbreviated that
\be\label{cumuCi}C_i:=\frac{\partial^2 Q_{\Lambda}}{\partial X_{i\mu}^2}= 2 \sum_{a'\le b'} \lambda_{a'b'} \frac{O_{ia'}O_{ib'}}{\sqrt{n}} =\OO(n^{-1/2}). \ee
Using \eqref{bound1Q}, we can bound that
\begin{align*}
	\frac{1}{ \sqrt{n}}  \sum_{ 1\le \mu \le n}  \left(\frac{\partial Q_{\Lambda}}{\partial X_{i\mu}}\right)^2f'' (Q_{\Lambda}) &\prec n^{-c} \sum_{a'\le b'} \lambda_{a'b'} \frac{1}{\sqrt{n}} \sum_\mu \left( |V_{\mu a'}| +  |V_{\mu b'}|  + n^{-1/2+\tau}\right)  \lesssim n^{-c+\tau}. 
\end{align*}
Similarly, we can get the bounds
\begin{align*}
	& \frac{1}{ \sqrt{n}}  \sum_{ 1\le \mu \le n} X_{i\mu} \left(\frac{\partial Q_{\Lambda}}{\partial X_{i\mu}}\right)^2f'' (Q_{\Lambda}) \prec n^{-c+\tau}, \quad \frac{1}{ \sqrt{n}}  \sum_{ 1\le \mu \le n} X_{i\mu} \left(\frac{\partial Q_{\Lambda}}{\partial X_{i\mu}}\right)^3 f''' (Q_{\Lambda}) \prec n^{-2c+\tau}, \\
	&\qquad\qquad\qquad\qquad\qquad\quad \frac{1}{ \sqrt{n}}  \sum_{ 1\le \mu \le n}  C_iX_{i\mu}\frac{\partial Q_{\Lambda}}{\partial X_{i\mu}}f'' (Q_{\Lambda}) \prec n^{-1/2+\tau}.
\end{align*}
On the other hand, with Lemma \ref{largedeviation}, we obtain the estimates
$$ \frac1{n} \sum_\mu X_{i\mu}X_{j\mu} = \delta_{ij} +\OO_\prec(n^{-1/2}),\quad \frac1{\sqrt{n}}\sum_{ 1\le \mu \le n} X_{i\mu} \prec n^{-1/2}.$$
Using these two estimates and \eqref{largeVmu}, we get that
\begin{align*}
	&  \frac1{\sqrt{n}}  \sum_{ 1\le \mu \le n}X_{i\mu} \frac{\partial Q_{\Lambda}}{\partial X_{i\mu}} f' (Q_{\Lambda}) \\
	& = \sum_{a'\le b'} \lambda_{a'b'} \sum_{1\le j \le \rho} \left(\frac1{n} \sum_\mu X_{i\mu}X_{j\mu}\right) \left( O_{ia'}O_{jb'} + O_{ja'}O_{ib'} \right)f' (Q_{\Lambda}) +\OO_\prec(n^{-1/2 })\\
	&=2 \sum_{a'\le b'} \lambda_{a'b'}  O_{ia'}O_{ib'} f' (Q_{\Lambda}) +\OO_\prec(n^{-1/2}); \\
	 &\frac1{\sqrt{n}}  \sum_{ 1\le \mu \le n} \frac{\partial Q_{\Lambda}}{\partial X_{i\mu}}f' (Q_{\Lambda}) \\
	&= \sum_{a'\le b'} \lambda_{a'b'} \left[\frac1{\sqrt{n}}\sum_\mu (U_{ia'}V_{\mu b'}+ U_{i b'}V_{\mu a'}) + \sum_{1\le j \le r}\frac1{n} \sum_\mu X_{j\mu} ( O_{ia'}O_{jb'} + O_{ja'}O_{ib'} ) \right]\\
	&= \sum_{a'\le b'} \lambda_{a'b'} \frac1{\sqrt{n}}\sum_\mu (U_{ia'}V_{\mu b'}+ U_{i b'}V_{\mu a'}) +\OO_\prec(n^{-1/2});\\
	& \frac1{\sqrt{n}}\sum_{ 1\le \mu \le n} C_i X_{i\mu}f' (Q_{\Lambda})   = \OO_\prec(n^{-1/2}).
\end{align*}
Finally, $\cal E_3(X_{i\mu})$ can be estimated in a similar way as $\cal E_2(X_{i\mu})$,
$\mathbb E\mathcal E_3 (X_{i\mu}) \le n^{-c}.$
We omit the details of its proof. Combining the above estimates and using Lemma \ref{lem_stodomin} (iii), we obtain that
\begin{align} 
	\sum_{1\le i,j \le \rho, 1\le \mu \le n}\frac1{\sqrt{n}}  O_{ia}O_{jb} \mathbb E \left[(X_{i\mu}X_{j\mu} - \delta_{ij})  f (Q_{\Lambda})\right] = s_{i}^2 \mathbb E f' (Q_{\Lambda}) + \OO_\prec(n^{-c+2\tau})\nonumber
\end{align}
for a deterministic $s_i^2$. Combining this equation with \eqref{firstIBP}, we obtain \eqref{stein}, which concludes Lemma \ref{asymp_Gauss}.

\section{Proof of Lemma \ref{lem_mbehaviorw} and Lemma \ref{lem_stabw}} \label{appd sol}

The proofs of Lemma \ref{lem_mbehaviorw} and Lemma \ref{lem_stabw} are standard applications of the contraction principle. 
\begin{proof}[Proof of Lemma \ref{lem_mbehaviorw}]
	We abbreviate $m_{\al c}\equiv m_{\al c}(\theta_l)$ and $\e_\al (z):= \omega_{\al c}(z) - m_{\al c}(\theta_l)$ with $|\e_\al| \le \wt c$ for a sufficiently small constant $\wt c>0$. From \eqref{selfomega}, we obtain the following equations for $(\omega_{1c},\omega_{2c})$:
	\begin{equation}\label{selfomega1}
		\begin{split}
			& \frac{c_1}{\omega_{1c} }=- z +   (1-\theta_l ) \frac{1 + (1-\theta_l )\omega_{2c}}{[1 + (1-\theta_l )\omega_{1c} ][1 + (1-\theta_l )\omega_{2c} ]-\theta_l ^{-1}},\\
			&\frac{c_2}{\omega_{2c} }=   (1-\theta_l) \frac{1 + (1-\theta_l )\omega_{1c}}{[1 + (1-\theta_l )\omega_{1c} ][1 + (1-\theta_l )\omega_{2c} ]-\theta_l ^{-1}}.
		\end{split}
	\end{equation}
	On the other hand,  using \eqref{selfm12}--\eqref{selfm32}, we can check that $m_{1 c}(\theta_l)$ and $m_{2 c}(\theta_l) $ satisfy the following equations:
	\begin{equation}\label{selfomega2}
		\begin{split}
			& \frac{c_1}{m_{1c} (\theta_l)}= (1-\theta_l ) \frac{1 + (1-\theta_l )m_{2c} (\theta_l)}{[1 + (1-\theta_l )m_{1c} (\theta_l)][1 + (1-\theta_l )m_{2c} (\theta_l)]-\theta_l ^{-1}},\\
			&\frac{c_2}{m_{2c}(\theta_l) }=   (1-\theta_l) \frac{1 + (1-\theta_l )m_{1c} (\theta_l)}{[1 + (1-\theta_l )m_{1c} (\theta_l) ][1 + (1-\theta_l )m_{2c} (\theta_l)]-\theta_l ^{-1}}.
		\end{split}
	\end{equation}
	Subtract \eqref{selfomega2} from \eqref{selfomega1}, we get that 
	\begin{equation}\label{selfomega3}
		\begin{split}
			 \frac{c_1\e_{1}}{(m_{1c} +\e_1) m_{1c} } 
			&= z +    \frac{ (1-\theta_l )^2 [g(m_{2c} + \e_2)g( m_{2c}) \e_1+\theta_l^{-1} \e_2]}{ [g(m_{1c} + \e_1)g( m_{2c} +\e_2)-\theta_l^{-1} ][g(m_{1c} )g( m_{2c} )-\theta_l^{-1} ]},\\
			 \frac{c_2\e_2}{(m_{2c}+\e_2) m_{2c}} 
			&=  \frac{ (1-\theta_l )^2[g(m_{1c} + \e_1)g( m_{1c}) \e_2+\theta_l^{-1} \e_1]}{ [g(m_{1c} + \e_1)g( m_{2c} +\e_2)-\theta_l^{-1} ][g(m_{1c} )g( m_{2c} )-\theta_l^{-1} ]},
		\end{split}
	\end{equation}
	where we have abbreviated $g(x):=1 + (1-\theta_l ) x $. Inspired by the above equations, we define iteratively a sequence of vectors ${\bm\e}^{(k)}=(\e_1^{(k)},\e_2^{(k)})\in \C^2$ such that ${\bm\e}^{(0)}=\b0\in \C^2$, and  
		\begin{align*}
			& \left\{\frac{c_1}{m_{1c}^2  } -  \frac{ (1-\theta_l )^2 g(m_{2c})^2 }{ [g(m_{1c})g( m_{2c} )-\theta_l^{-1} ]^2}\right\} \e^{(k+1)}_{1}  -  \frac{ (1-\theta_l )^2 \theta_l^{-1}}{ [g(m_{1c})g( m_{2c} )-\theta_l^{-1} ]^2} \e_2^{(k+1)} \\
			& = z + \frac{c_1 (\e_1^{(k)})^2}{m_{1c}^2 (m_{1c}+\e_1^{(k)})} \\
			&\quad+  \frac{ (1-\theta_l )^2}{g(m_{1c} )g( m_{2c})-\theta_l^{-1} } \left\{\frac{ g(m_{2c} + \e_2^{(k)})g( m_{2c}) \e_1^{(k)}+\theta_l^{-1} \e_2^{(k)}}{ g(m_{1c} + \e_1^{(k)})g( m_{2c} +\e_2^{(k)})-\theta_l^{-1} }  - \frac{ g(m_{2c})^2 \e_1^{(k)}+\theta_l^{-1} \e_2^{(k)}}{ g(m_{1c} )g( m_{2c})-\theta_l^{-1} } \right\}, \\
			& \left\{\frac{c_2}{m_{2c}^2  } -  \frac{ (1-\theta_l )^2 g(m_{1c})^2 }{ [g(m_{1c})g( m_{2c} )-\theta_l^{-1} ]^2}\right\} \e^{(k+1)}_{2}  -  \frac{ (1-\theta_l )^2 \theta_l^{-1}}{ [g(m_{1c})g( m_{2c} )-\theta_l^{-1} ]^2} \e_1^{(k+1)}  \\
			&= \frac{c_2 (\e_2^{(k)})^2}{m_{2c}^2 (m_{2c}+\e_2^{(k)})} \\
			&\quad +  \frac{ (1-\theta_l )^2}{g(m_{1c} )g( m_{2c})-\theta_l^{-1} } \left\{\frac{ g(m_{1c} + \e_1^{(k)})g( m_{1c}) \e_2^{(k)}+\theta_l^{-1} \e_1^{(k)}}{ g(m_{1c} + \e_1^{(k)})g( m_{2c} +\e_2^{(k)})-\theta_l^{-1} }  - \frac{ g(m_{1c})^2 \e_2^{(k)}+\theta_l^{-1} \e_1^{(k)}}{ g(m_{1c} )g( m_{2c})-\theta_l^{-1} } \right\}.
		\end{align*}
	In other words, the above two equations define a mapping $\mathbf f:\ell^\infty (\Z_2)\to \ell^\infty (\Z_2)$, so that
	\be\label{iteration} 
	{\bm \e}^{(k+1)}= \mathbf f({\bm \e}^{(k)}), \quad \mathbf f(\mathbf x):=S^{-1}\begin{pmatrix} z\\ 0 \end{pmatrix}+ S^{-1} \mathbf e(\mathbf x),
	\ee
	where 
	$$
	S:=\begin{bmatrix} \frac{c_1}{m_{1c}^2  } -  \frac{ \theta_l^2 (1-\theta_l )^2 }{ (1-t_l)^2}g(m_{2c})^2  & -  \frac{ (1-\theta_l )^2\theta_l }{ (1-t_l)^2} \\ -   \frac{ (1-\theta_l )^2\theta_l }{ (1-t_l)^2}  &\frac{c_2}{m_{2c}^2  } -  \frac{ \theta_l^2(1-\theta_l )^2 }{ (1-t_l)^2}g(m_{1c})^2  \end{bmatrix},  
	$$
	and
	$$\mathbf e(\mathbf x):= \begin{bmatrix} \frac{c_1 x_1^2}{m_{1c}^2 (m_{1c}+x_1)} -  \frac{ \theta_l(1-\theta_l )^2}{1-t_l } \left\{\frac{ g(m_{2c} + x_2)g( m_{2c}) x_1 +\theta_l^{-1} x_2}{ g(m_{1c} +x_1)g( m_{2c} +x_2)-\theta_l^{-1} }  - \frac{ g(m_{2c})^2 x_1+\theta_l^{-1} x_2}{ g(m_{1c} )g( m_{2c})-\theta_l^{-1} } \right\} \\ \frac{c_2 x_2^2}{m_{2c}^2 (m_{2c}+x_2)} -  \frac{ \theta_l (1-\theta_l )^2}{1-t_l } \left\{\frac{ g(m_{1c} + x_1)g( m_{1c}) x_2+\theta_l^{-1} x_1}{ g(m_{1c} + x_1)g( m_{2c} +x_2)-\theta_l^{-1} }  - \frac{ g(m_{1c})^2 x_2 +\theta_l^{-1} x_1}{ g(m_{1c} )g( m_{2c})-\theta_l^{-1} } \right\} \end{bmatrix}.$$
	Here, we have used $\theta_lg(m_{1c})g( m_{2c} ) = f_c(\theta_l) = t_l$ (which follows from \eqref{hz} and \eqref{fcz}) to simplify the expressions a little bit.
	
	With a direct calculation, we can check that under \eqref{tlc}, there exist constants $\wt c, \wt C>0$ depending only on $c_1, c_2$ and $\delta_l$ such that
	\be\label{dust}
	\|S^{-1}\|_{\ell^\infty\to \ell^\infty} \le \wt C, \quad \text{and} \quad \|\mathbf e(\mathbf x)\|_\infty \le \wt C\|\mathbf x\|_\infty^2 \ \ \text{for} \ \ \|\bx\|_\infty \le \wt c.
	\ee
	With \eqref{dust}, it is easy to check that there exists a sufficiently small constant $\tau>0$ depending only on $\wt C$, such that $\mathbf f$ is a self-mapping 
	$$\mathbf f: B_r \left(\ell^\infty(\Z_2)\right)\to B_r \left(\ell^\infty(\Z_2)\right), \quad B_r\left(\ell^\infty(\Z_2)\right):=\{\bx\in \ell^\infty(\Z_2): \|\bx\|_\infty \le r \},$$
	as long as $r\le \tau$ and $|z| \le c_\tau$ for some constant $c_\tau>0$ depending only on $c_1, c_2,\delta_l$ and $\tau$. 
	Now, it suffices to prove that $h$ restricted to $B_r \left(\ell^\infty(\Z_2)\right)$ is a contraction, which implies that ${\bm \e}:=\lim_{k\to\infty} {{\bm \e}}^{(k)}$ exists and is a unique solution to \eqref{selfomega3} subject to the condition $\|{\bm \e}\|_\infty \le r$.

	From the iteration relation \eqref{iteration}, using \eqref{dust}, we obtain that
	\be\label{k1k}
	{\bm \e}^{(k+1)} - {\bm \e}^{(k)}= S^{-1}\left[\mathbf e({\bm \e}^{(k)}) -\mathbf e({\bm \e}^{(k-1)})\right] \le \wt C \left\|\mathbf e({\bm \e}^{(k)}) -\mathbf e({\bm \e}^{(k-1)})\right\|_\infty .
	\ee
	From the expression of $\mathbf e$, we see that as long as $r$ is chosen to be sufficiently small compared to $\theta_l^{-1}-g(m_{1c})g( m_{2c} )= (1-t_l)\theta_l^{-1}$, then
	$$\left\|\mathbf e({\bm \e}^{(k)}) -\mathbf e({\bm \e}^{(k-1)})\right\|_\infty \le C \left(\|{\bm\e}^{(k)}\|_\infty +\|{\bm\e}^{(k-1)}\|_\infty\right)\|{\bm \e}^{(k)} - {\bm \e}^{(k-1)}\|_\infty$$
	for some constant $C >0$ depending only on $c_1, c_2$ and $\delta_l$. Thus, we can choose a sufficiently small constant $0<r \le \min\{\tau, (2C)^{-1}\}$ such that $Cr \le 1/2$. Then, $\mathbf f$ is indeed a contraction mapping on $ B_r \left(\ell^\infty(\Z_2)\right)$, which proves both the existence and uniqueness of the solution to \eqref{selfomega3} if we choose $c_0$ in \eqref{prior10} as $c_0=\min\{c_\tau, r\}$. After obtaining $\omega_{1c}=m_{1c}+\e_1$ and $\omega_{2c}=m_{2c}+\e_2$, we can define $\omega_{3c}$ and $\omega_{4c}$ using the first and third equations in \eqref{selfomega}.
	
	Note that with \eqref{dust} and ${\bm \e}^{(0)}=\b0$, we get from \eqref{iteration} that $ \|{\bm \e}^{(1)}\|_\infty\le \wt C|z| .$
	Then, with the contraction property of $\mathbf f$, we get that
	$$\|{\bm \e}\|_\infty \le \sum_{k=0}^\infty \|{\bm \e}^{(k+1)}-{\bm \e}^{(k)}\|_\infty \le 2\wt C|z|.$$
	This gives the bound \eqref{Lipomega} for $\omega_{1c}$ and $\omega_{2c}$. Then, using the first and third equations in \eqref{selfomega}, we immediately obtain the bound \eqref{Lipomega} for $\omega_{3c}$ and $\omega_{4c}$ as long as $c_0$ is sufficiently small. 
\end{proof}

\begin{proof}[Proof of Lemma \ref{lem_stabw}]
	As in the proof of Lemma \ref{lem_mbehaviorw}, we subtract the equations \eqref{selfomegaerror} from \eqref{selfomega}, and consider the contraction principle for the functions $\e_\al (z):= \omega_{\al }(z) - \omega_{\al c}(z)$. We omit the details.
\end{proof}








\end{document}